\documentclass[a4paper,twoside,10pt]{article}
\usepackage[a4paper,left=2cm,right=2cm, top=3cm, bottom=3cm]{geometry}
\usepackage{cuted}

\usepackage{enumerate}
\usepackage{amsthm}%
\usepackage{mathrsfs}
\usepackage{mathtools}
\usepackage{accents}
\usepackage{orcidlink}
\usepackage{graphicx}
\usepackage{amscd,amsmath,amssymb,mathrsfs,bbm,listings}
\usepackage{cancel}
\usepackage{tikz}
\usetikzlibrary{decorations.pathreplacing}
\usepackage{epstopdf}
\graphicspath{{Figs/}}
\usepackage{amsmath}
\usepackage{footmisc}
\usepackage{amsthm}
\usepackage{amssymb}
\usepackage{stmaryrd}
\SetSymbolFont{stmry}{bold}{U}{stmry}{m}{n}
\usepackage{float}
\usepackage{bigints}
\usepackage{cite}
\usepackage{color}
\usepackage[abs]{overpic}
\usepackage[font=footnotesize,labelfont=bf]{caption}
\usepackage{cases}
\usepackage{tikz}
\usepackage{algorithmic}
\usepackage{rotating}
\usepackage{blkarray}
\usetikzlibrary{matrix,calc,arrows}
\usepackage{pdfcomment}
\usepackage{filecontents}
\usepackage{pgfplots}
\usepackage{subcaption}
\usepackage{soul,xcolor}
\usepackage{enumitem}  
\usepackage{verbatim}
\usepackage{graphicx}
\usepackage{alphabeta}
\usepackage{mwe}
\usepackage{algorithm}
\usepackage{array}

\newtheorem{theorem}{Theorem}[section]
\newtheorem{lemma}[theorem]{Lemma}
\newtheorem{proposition}[theorem]{Proposition}
\newtheorem{assumption}{Assumption}[section]

\newtheorem{remark}[theorem]{Remark}

\newcommand{\eremk}{\hbox{}\hfill\rule{0.8ex}{0.8ex}}

\newcolumntype{C}{>{\centering\arraybackslash}p{2.5cm}}
\numberwithin{equation}{section}

\definecolor{darkblue}{rgb}{0.0, 0.0, 0.8}
\hypersetup{
    colorlinks=true,
    citecolor=red,
    linkcolor=darkblue,
    pdfborder={0 0 0}
}

\DeclareMathOperator*{\essinf}{ess\,inf}
\newcommand{\Ccoer}{C_{\mathrm{coer}}}
\newcommand{\Csob}{C_{\mathrm{Sob}}}

\newcommand{\N}{\mathbb{N}}
\newcommand{\R}{\mathbb{R}}
\newcommand{\bx}{\boldsymbol{x}}
\newcommand{\by}{\boldsymbol{y}}
\newcommand{\bn}{\boldsymbol{n}}

\newcommand{\Norm}[2]{\|#1\|_{#2}}
\newcommand{\Seminorm}[2]{|#1|_{#2}}

\newcommand{\dpt}{\partial_t}
\newcommand{\dptau}{\partial_\tau}
\newcommand{\QT}{Q_T}

\newcommand{\bnOmega}{\bn_{\Omega}}
\newcommand{\bnK}{\bn_{K}}

\newcommand{\D}{\boldsymbol{D}}

\newcommand{\dext}{\text{d}_{\mathrm{upd}}}
\newcommand{\Dmax}{D_{\mathrm{max}}}

\newcommand{\vz}{\boldsymbol{z}}

\newcommand{\vsigma}{\boldsymbol{\sigma}}

\newcommand{\vr}{\boldsymbol{r}}

\newcommand{\brho}
{\boldsymbol{\rho}}
\newcommand{\bff}{\boldsymbol{f}}

\newcommand{\up}{u_p}
\newcommand{\uq}{u_q}

\newcommand{\fp}{f_p}
\newcommand{\fq}{f_q}

\renewcommand{\sp}{s_p}
\newcommand{\sq}{s_q}

\renewcommand{\wp}{w_p}
\newcommand{\wq}{w_q}

\newcommand{\In}{I_n}
\newcommand{\tn}{t_n}
\newcommand{\tnmo}{t_{n - 1}}
\newcommand{\tnpo}{t_{n + 1}}
\newcommand{\Th}{\mathcal{T}_h}
\newcommand{\Thm}{\mathcal{T}_{h_m}}
\newcommand{\Tt}{\mathcal{T}_{\tau}}
\newcommand{\nmo}{(n - 1)}
\newcommand{\n}{(n)}
\newcommand{\npo}{(n + 1)}

\newcommand{\kpo}{(k + 1)}

\newcommand{\LDG}{{\mathrm{LDG}}}
\newcommand{\nablaLDG}{\nabla_{\LDG}}
\newcommand{\divLDG}{\mathrm{div}_{\LDG}}
\newcommand{\DG}{{\mathrm{DG}}}

\newcommand{\dx}{\, \mathrm{d}\boldsymbol{x}}
\newcommand{\dtau}{\, \mathrm{d}\tau}

\newcommand{\dt}{\, \mathrm{d}t}
\newcommand{\dV}{\, \mathrm{d}V}

\newcommand{\Wh}{\mathcal{W}^{\,\ell}(\Th)}
\newcommand{\Whm}{\mathcal{W}^{\,\ell}(\mathcal{T}_{h_m})}
\newcommand{\Zh}{\boldsymbol{\mathcal{Z}}^{\,\ell}(\Th)}
\newcommand{\Rh}{\boldsymbol{\mathcal{R}}^\ell(\Th)}
\newcommand{\Rhm}{\boldsymbol{\mathcal{R}}^\ell(\Thm)}
\newcommand{\vh}{v_h}
\newcommand{\wh}{w_{h}}

\newcommand{\whp}{w_{p, h}}
\newcommand{\whq}{w_{q, h}}
\newcommand{\zh}{\boldsymbol{z}_{h}}

\newcommand{\zhp}{\boldsymbol{z}_{p,h}}
\newcommand{\zhq}{\boldsymbol{z}_{q,h}}
\newcommand{\sigmah}{\boldsymbol{\sigma}_{h}}

\newcommand{\sigmahp}{\boldsymbol{\sigma}_{p,h}}
\newcommand{\sigmahq}{\boldsymbol{\sigma}_{q,h}}

\newcommand{\rh}{\boldsymbol{r}_{h}}
\newcommand{\rhp}{\boldsymbol{r}_{p,h}}
\newcommand{\rhq}{\boldsymbol{r}_{q,h}}
\newcommand{\phih}{\boldsymbol{\phi}_h}

\newcommand{\phiqh}{\boldsymbol{\phi}_{q,h}}

\newcommand{\lambdah}{\lambda_h}
\newcommand{\psim}{\psi_m}
\newcommand{\psihm}{\psi_{h_m}}
\newcommand{\bphi}{\boldsymbol{\phi}}
\newcommand{\bphim}{\boldsymbol{\phi}_{m}}

\newcommand{\psih}{\psi_h}

\newcommand{\uph}{u_{p,h}}
\newcommand{\uqh}{u_{q,h}}

\newcommand{\whstar}{w_{\star, h}}
\newcommand{\zhstar}{\boldsymbol{z}_{\star,h}}
\newcommand{\sigmahstar}{\boldsymbol{\sigma}_{\star,h}}
\newcommand{\sigmahmstar}{\boldsymbol{\sigma}_{\star,m}}
\newcommand{\rhstar}{\boldsymbol{r}_{\star,h}}

\newcommand{\ustar}{u_{\star}}
\newcommand{\ustarh}{u_{\star,h}}
\newcommand{\fstar}{f_{\star}}
\newcommand{\sstar}{s_{\star}}

\newcommand{\vhp}{v_{p, h}}
\newcommand{\vhq}{v_{q, h}}
\newcommand{\vhstar}{v_{\star, h}}

\newcommand{\cvW}{\mathbf{W}}
\newcommand{\cvWstar}{\mathbf{W}_{\star,h}}
\newcommand{\cvWp}{\mathbf{W}_{p,h}}
\newcommand{\cvWq}{\mathbf{W}_{q,h}}
\newcommand{\cvZstar}{\mathbf{Z}_{\star,h}}

\newcommand{\cvSigmastar}{\mathbf{\Sigma}_{\star,h}}
\newcommand{\cvSigmap}{\mathbf{\Sigma}_{p,h}}
\newcommand{\cvSigmaq}{\mathbf{\Sigma}_{q,h}}
\newcommand{\cvRstar}{\mathbf{R}_{\star,h}}

\newcommand{\cvVstar}{\mathbf{V}_{\star,h}}
\newcommand{\cvVp}{\mathbf{V}_{p,h}}
\newcommand{\cvVq}{\mathbf{V}_{q,h}}

\newcommand{\MM}{\mathbb{M}}
\newcommand{\MMstar}{\MM_\star}
\newcommand{\GG}{\mathcal{G}}

\newcommand{\diffsigma}{\mathcal{D}_{\mathbf{\Sigma}_{\star}}\!\!}
\newcommand{\diffw}{\mathcal{D}_{\mathbf{W}_{\star}}\!\!}
\newcommand{\diffwp}{\mathcal{D}_{\mathbf{W}_{p}}\!\!}
\newcommand{\diffwq}{\mathcal{D}_{\mathbf{W}_{q}}\!\!}

\newcommand{\equilp}{{\mathcal E}_p}
\newcommand{\equilq}{{\mathcal E}_q}

\newcommand{\PiW}{\Pi_{\mathcal{W}}}
\newcommand{\PiR}{\boldsymbol{\Pi}_{\boldsymbol{\mathcal{R}}}}
\newcommand{\Pp}[2]{\mathbb{P}^{#1}(#2)}

\def\Fh{\mathcal{F}_h}
\def\Fho{\mathcal{F}_h^\mathcal{I}}
\def\Fhmo{\mathcal{F}_{h_m}^\mathcal{I}}
\def\FhN{\mathcal{F}_h^{\mathcal{N}}}

\newcommand{\mvl}[1]{\{ \!\!\{#1\}\!\!\}}  
\newcommand{\jump}[1]{\llbracket #1\rrbracket}

\newcommand{\mvlparamF}{\gamma_F}

\definecolor{newred}{rgb}{0.8, 0.0, 0.0}

\title{Structure-preserving local discontinuous Galerkin discretization \\ of conformational conversion systems}

\author{Paola F. Antonietti\,\orcidlink{0000-0002-2138-3878}\thanks{MOX-Dipartimento di Matematica, Politecnico di Milano, Piazza Leonardo da Vinci 32, Milan, 20133, Italy (\href{mailto:paola.antonietti@polimi.it}{paola.antonietti@polimi.it}, \href{mailto:mattia.corti@polimi.it}{mattia.corti@polimi.it})} 
\and
Mattia Corti\,\orcidlink{0000-0002-7014-972X}\footnotemark[1]
\thanks{Faculty of Mathematics, University of Vienna, Oskar-Morgenstern-Platz 1, 1090 Vienna, Austria (\href{mailto:mattia.corti@univie.ac.at}{mattia.corti@univie.ac.at},
\href{mailto:ilaria.perugia@univie.ac.at}{ilaria.perugia@univie.ac.at})}
\and
Sergio G\'omez\,\orcidlink{0000-0001-9156-5135}\thanks{Department of Mathematics and Applications, University of Milano-Bicocca, 20125 Milan, Italy (\href{mailto:sergio.gomezmacias@unimib.it}{sergio.gomezmacias@unimib.it})}
\thanks{IMATI-CNR ``E. Magenes", Via Ferrata 5, 27100 Pavia, Italy}
\and 
Ilaria Perugia\,\orcidlink{0000-0003-1368-2883}\footnotemark[2]
}

\date{}

\begin{document}
\maketitle

\begin{abstract}
\noindent We investigate a two-state conformational conversion system and introduce a novel structure-preserving numerical scheme that couples a local discontinuous Galerkin space discretization with the backward Euler time-integration method. The model is first reformulated in terms of auxiliary variables involving suitable nonlinear transformations, which allow us to enforce positivity and boundedness at the numerical level.
Then, we prove a discrete entropy-stability inequality, which we use to show the existence of discrete solutions, as well as to establish the convergence of the scheme by means of some discrete compactness arguments. 
As a by-product of the theoretical analysis, we also prove the existence of global weak solutions satisfying the system's physical bounds. 
Numerical results validate the theoretical results and assess the capabilities of the proposed method in practice.
\end{abstract}



\section{Introduction}

Conformational conversion systems are a class of coupled %
(semilinear) reaction–diffusion systems of partial differential equations (PDEs) that describe how multiple conformational states of molecules or particles change over space and time, including their ability to interconvert and diffuse. The population associated with each conformation is represented by one variable governed by its own PDE, where the diffusion terms model the spatial spreading, and the reaction terms describe production, elimination, and interconversion between conformations. 
Conformational conversion systems %
take into account three main mechanisms: i) state transitions, representing chemical or physical changes; ii) spatial dynamics, describing diffusion of each state; and iii) external forces, accounting for external inputs. 
These systems are commonly used across numerous biological, chemical, and physical processes, where elements of a spatially distributed system can change their internal structure (i.e., their conformational state), while simultaneously undergoing %
diffusion through space. 
For instance, in cell biology, they are used to describe protein conformational changes and how these propagate within a cell. In neuroscience, they are used to model the spread of misfolded proteins in neurodegenerative diseases \cite{fornari_spatially-extended_2020,matthaus_diffusion_2006,weickenmeier_physics-based_2019}. In chemical kinetics, they appear as spatially extended catalytic reaction models, while in materials science, they are used to describe phase transitions in innovative materials. Similar formulations have also been proposed for ecosystem modeling, such as resource–consumer systems with negligible resource competition, known as ``MacArthur-type" models \cite{cui_effect_2020,marsland_minimum_2020}, or for plant–water interaction dynamics in soil models, where they appear as two-component reaction–diffusion systems \cite{kabir_numerical_2022}.

Due to the nonlinear coupling between reaction kinetics, multiple interacting states, and diffusion occurring at different scales, the mathematical analysis of conformational conversion systems poses significant challenges. A key point is that, under suitable assumptions on the model's data, it is often possible to prove the positivity and boundedness of solutions, which are essential for ensuring physical consistency. 
Approximating the solution to conformational conversion systems adds its own challenges, as it demands 
schemes that preserve, at the discrete level, the key structural properties of the corresponding continuous model. In particular, positivity preservation is not automatically ensured by standard numerical schemes, even when it holds for the continuous PDE system. 
Consequently, the development of structure-preserving numerical methods is an active area of research. Notable contributions in the framework of numerical discretization approaches that can preserve physical bounds at the discrete level, within a variational setting,
are, e.g., the nodally bound-preserving finite element method~\cite{Barrenechea-Volker-Knobloch-2024,Amiri-Barrenechea-Pryer-2025}, the proximal Galerkin method~\cite{Keith-Surowiec-2024,Keith-Masri-Zeinhofer-2025,Fu_Keith_Masri:2025}, and structure-preserving schemes for cross-diffusion~\cite{Braukhoff-Perugia-Stocker2022,Gomez-Jungel-Perugia:2024} and  %
(advection)--diffusion--reaction problems~\cite{BonizzoniBraukhoffJungelPerugia_2020,Corti_Bonizzoni_Antonietti:2024,Antonietti-Corti-Gomez-Perugia_2026,Lemaire_Moatti:2024,Moatti:2023}.
In the literature, structure-preserving schemes %
for reaction--diffusion systems %
distinguish between \emph{reversible} (satisfying detailed balance or possessing a gradient-flow structure) and \emph{irreversible} systems. 
Recent works~\cite{Liu_Structure_2021,Liu_Convergence_2022,Fu_High_2023} focus on reversible reaction--diffusion systems, 
making use of an associated Lyapunov functional to reformulate the model as a minimization problem.
Such a structure is not available in our model, due to the presence of an irreversible conversion term in the reaction.

\paragraph{Novelty.} In this paper, we consider a two-state conformational conversion system and propose a novel structure-preserving scheme based on a local discontinuous Galerkin (LDG) space discretization coupled with the %
backward Euler time-integration scheme. 
The key point is to reformulate (via a suitable change of variables involving nonlinear transformations) the model problem as a system in terms of entropy variables, thereby ensuring solution positivity, boundedness of one of the two components, and an entropy-stability inequality at the discrete level. The entropy densities underlying these changes of variables are related to Legendre functions in the framework of proximal Galerkin methods~\cite{Keith-Surowiec-2024}.

Conformational conversion systems typically feature coupled dynamics across multiple states, whose solutions may develop sharp, time-dependent fronts. This motivates the use of a discontinuous Galerkin (DG) framework, which offers greater flexibility in mesh design and local approximation spaces than the standard $H^1$-conforming finite elements. 
More specifically, DG methods inherently support high-order, possibly elementwise varying, local polynomial approximations on nonmatching meshes of arbitrary shape, and are well-suited for adaptive refinement strategies, without imposing interelement continuity constraints. This flexibility enables accurate and efficient approximation of the underlying dynamics and effective front tracking.
In particular, among DG methods, the LDG approach~\cite{Cockburn_Shu:1998,Castillo_etal:2000} is especially well suited to the reformulation in terms of entropy variables, where the diffusion becomes nonlinear. Indeed, through the introduction of auxiliary variables, it allows us to design a method with the following properties: \emph{i)} nonlinearities are not embedded in spatial differential operators and interface terms, and can therefore be evaluated elementwise in a simple and efficient manner, endowing the method with an outstanding parallelizable structure; \emph{ii)} the discretization matrices for the resulting linear spatial differential operators are the standard ones, and need to be assembled only once at the beginning of the simulation; \emph{iii)} the chain rule required for entropy stability can be weakly imposed; and \emph{iv)} the block-diagonal mass matrices, which are typical of DG methods, allow 
for the algebraic elimination of all auxiliary variables, so the number of degrees of freedom does not increase with respect to the original problem. These properties result in a significant reduction in computational cost, particularly in three-dimensional problems, compared to other DG methods.

For the  proposed proposed structure-preserving backward Euler-LDG scheme, we
establish convergence of the discrete scheme (up to subsequences) under minimal regularity assumptions. As a byproduct of our analysis, we also prove the existence of global weak solutions satisfying the physical bounds of the model. 
A distinctive novelty of this work, compared to the framework in~\cite{Gomez-Jungel-Perugia:2024} for cross-diffusion systems, is that we get a ``degeneracy" in the entropy estimate, which provides
a bound in the~$L^2$ norm on~$\nabla \sqrt{q}$, instead of~$\nabla q$, being~$q$ one of the conformational variables. The ideas used to address the theoretical challenges resulting from this ``degeneracy" are key to extend the framework in~\cite{Gomez-Jungel-Perugia:2024} to a broader class of nonlinear models.

\paragraph{Structure of the manuscript.} The remainder of the paper is organized as follows: Section~\ref{sec:2} presents the model problem and the structure-preserving numerical method, and Sections~\ref{SEC::ANALYSIS} and~\ref{sec:convergence} present its theoretical analysis. More precisely, in Section~\ref{SEC::MODEL}, we introduce the conformational conversion system under investigation. 
In Section~\ref{SEC::REFORMULATION}, we reformulate the model using a suitable change of variables, which arises from the underlying entropy structure of the system,  and present its discretization using the proposed structure-preserving backward Euler-LDG method in Section~\ref{SEC::METHOD}. In Section~\ref{SEC::ENTROPY_CONT}, we prove an entropy stability bound of the system. Then, we derive a discrete analogue for our structure-preserving scheme in Section~\ref{SEC::ENTROPY_DISCR}, and use it to prove the existence of discrete solutions in Section~\ref{sec:EXISTENCE-DISCRETE}. Section~\ref{sec:convergence} is devoted to establishing the convergence of the structure-preserving scheme. First, we prove that the scheme converges to a regularized semidiscrete-in-time formulation as the mesh size~$h$ goes to zero; then, we analyze the convergence of such a formulation as~$(\varepsilon, \tau) \to (0, 0)$, being $\varepsilon$ and $\tau$ a suitable penalty term and the time integration step, respectively. Section~\ref{sec:results} discusses numerical results aimed at validating the theoretical results and assessing the proposed structure-preserving method in practice.  
In Section~\ref{sec: Conclusions}, we draw some conclusions and discuss further developments. Finally, Appendix~\ref{SEC::NEWTON} presents a detailed derivation of the linear systems that arise from applying Newton's method to the backward Euler-LDG method.

\section{Model problem and numerical approximation}\label{sec:2}

In this section, we first introduce the model problem (Section~\ref{SEC::MODEL}) and then reformulate it in a convenient form (Section~\ref{SEC::REFORMULATION}). Next, we define the structure-preserving backward Euler-LDG method to discretize it (Section~\ref{SEC::METHOD}), and conclude by presenting the matrix formulation of the fully discrete method (Section~\ref{SEC::MATRIX}). An explicit derivation of the linear systems that arise from applying Newton's method is postponed to Appendix~\ref{SEC::NEWTON}. Here and in the following, we use standard notation for~$L^p$, Sobolev, and Bochner spaces.

\subsection{The conformational conversion system}\label{SEC::MODEL}
We define the space--time cylinder~$\QT \coloneqq  \Omega \times (0, T)$, where~$\Omega \subset \R^d$ ($d\in \{2, 3\}$) is a polytopic domain with Lipschitz boundary~$\Gamma\coloneqq \partial \Omega$ and outward-pointing normal unit vector~$\bnOmega$, and~$T>0$ is some final time.
We consider the following system: find~$p:\QT \rightarrow \R$ and~$q : \QT \rightarrow \R$ such that
\begin{subequations}
\label{EQN::HETERODIMER}
\begin{alignat}{3}
\dpt p - \nabla \cdot (\D \nabla p) & = - p(\lambda_p + \mu_{pq} q) + \kappa_p & & \quad \text{ in } \QT,\\
\dpt q - \nabla \cdot (\D \nabla q) & = -q(\lambda_q - \mu_{pq} p) & & \quad \text{ in } \QT,\\
(\D \nabla p) \cdot \bnOmega = 0  \ \text{ and } \ (\D \nabla q) \cdot \bnOmega & = 0 & & \quad  \text{ on } \Gamma \times (0, T),\\
p(\cdot, 0) = p_{0} \ \text{ and } \ 
q(\cdot, 0) & = q_{0} & & \quad \text{ in } \Omega.
\end{alignat}
\end{subequations}
The unknowns~$p$ and~$ q$ represent the populations associated with the two conformations. As such, they must be nonnegative. 
We assume that the tensor~$\D= \D(\bx)\in \R^{d\times d}$, which characterizes the diffusion of both~$p$ and~$q$, belongs to~$L^{\infty}(\Omega)^{d\times d}$ and is uniformly positive definite in~$\Omega$: there exist~$\dext, \Dmax>0$ such that
\begin{equation}\label{eq:D}
\Norm{\D}{L^{\infty}(\Omega)^{d\times d}}=D_{\max}<\infty\qquad \text{and}\qquad
\forall \vz\in\R^d, \quad \vz^\top \D\vz
\ge\dext|\vz|^2 \quad \text{a.e. in}\ \Omega.
\end{equation}
The parameter~$\kappa_p>0$ is the production rate of~$p$, $\lambda_p>0$ and~$\lambda_q> 0$ are the clearance rates of~$p$ and~$q$, respectively, and~$\mu_{pq}>0$ is the conversion rate from~$p$ to~$q$.
For convenience, we set
\[
\Upsilon_{pq}\coloneqq \kappa_p \mu_{pq} - \lambda_p\lambda_q, \qquad
\equilp\coloneqq \frac{\kappa_p}{\lambda_p},\qquad \equilq\coloneqq \frac{\Upsilon_{pq}}{\lambda_q\mu_{pq}}.
\]
Traveling wave solutions exist under the assumption~$\Upsilon_{pq}>0$;~\cite[\S2.2]{matthaus_diffusion_2006}. For convenience, we set
\begin{equation}\label{EQ::HETER_FPFQ}
\fp(p, q) \coloneqq  -p(\lambda_p + \mu_{pq} q) + \kappa_p,  \qquad \fq(p, q) \coloneqq  -q(\lambda_q - \mu_{pq} p).
\end{equation}
For the initial conditions, we assume that
\begin{equation}\label{eq:ic}
\begin{split}
p_0 & \in \big\{\mu \in L^{\infty}(\Omega) \ : \ 0\le \mu
\le \equilp\quad \text{a.e.~in}\ \Omega \big\} \quad \text{and} \quad 
q_0 \in \big\{\mu\in L^\infty(\Omega) \ : \ \mu \geq 0 \quad \text{a.e.~in}\ \Omega \big\},
\end{split}
\end{equation}
so that 
\begin{equation}\label{EQ:HETER_BOUNDSCQ}
0\le p \le \equilp\quad \text{and} \quad q \ge 0 \qquad 
 \text{a.e.~in}\ Q_T;
\end{equation}
see~\cite[\S2.1]{Mattia}.

\begin{remark}[Initial datum~$p_0$]\label{REM::INITIAL_DATA}
The assumption~$p_0\le \equilp=\kappa_p/\lambda_p$ in $\Omega$ is motivated by the fact that, in this system, $\equilp$ is assumed to be the maximum of the population $p$ (which is associated with the unstable equilibrium).
\eremk
\end{remark}

\begin{remark}[The Fisher-Kolmogorov equation]\label{rem:FK}
For~$p\gg q$, neglecting the time derivative and the diffusion of~$p$, and using a Taylor approximation,
system~\eqref{EQN::HETERODIMER} reduces to the following Fisher-Kolmogorov equation in the rescaled variable~$c\coloneqq  q/q_M$, with~$q_M\coloneqq \Upsilon_{pq}/(\equilp\mu_{pq}^2)$: 
\begin{subequations}
\label{EQN::FISHER-KPP}
\begin{alignat}{3}
\dpt c - \nabla \cdot (\D \nabla c) & = \alpha\, c (1 - c) & & \quad \text{ in } \QT,\\
(\D \nabla c) \cdot \bnOmega & = \mathbf{0}  & & \quad  \text{ on } \Gamma \times (0, T),\\
c(\cdot, 0) & = c_0 & & \quad \text{ in } \Omega,
\end{alignat}
\end{subequations}
with~$\alpha\coloneqq \Upsilon_{pq}/{\lambda_p}$ and~$c_0 \coloneqq  q_0/q_M$; see~\cite[\S2]{Goriely}.
\eremk
\end{remark}
\begin{remark}[The heterodimer model]\label{rem:Heterodimer}
An example of a system of the form \eqref{EQN::HETERODIMER} is the so-called heterodimer model, which describes the spatial and temporal dynamics of protein conformational changes. The heterodimer model is relevant in the modeling of neurodegenerative diseases, e.g.~proteinopathies such as Alzheimer's, Parkinson's, and prion diseases~\cite{fornari_spatially-extended_2020,matthaus_diffusion_2006,weickenmeier_physics-based_2019}.
In this context, the unknowns~$p$ and~$q$ represent the quantities of healthy and misfolded proteins, respectively.
\eremk
\end{remark}

The following change of variables enforces the bounds in~\eqref{EQ:HETER_BOUNDSCQ} on %
$p$ and~$q$:
\begin{equation}\label{EQ::HETER_UPUQ}
p=u_p(w_p)\coloneqq \equilp\left(\frac{e^{w_p}}{1+e^{w_p}}\right), \qquad
q=u_q(w_q)\coloneqq \equilq e^{w_q}, 
\qquad w_p,w_q:Q_T\mapsto\R.
\end{equation}

This choice is naturally motivated by the underlying entropy structure of the system, which is at the core of the boundedness-by-entropy setting of~\cite{Jungel:2015}; see Remark~\ref{rem:math_entropy} below. Indeed, it corresponds to writing~$u_p=(s_p')^{-1}$, $u_q=(s_q')^{-1}$,
with the \emph{entropy density functions}~$s_p$ and~$s_q$ defined, respectively, by 
\begin{equation}\label{EQ::HETER_SPSQ}
s_p(p)\coloneqq  p\log p+\left(\equilp-p\right)\log\left(\equilp-p\right)+\max\{\equilp\log(2\equilp^{-1}),0\}\ge 0,\qquad s_q(q)\coloneqq  q(\log{} (\equilq^{-1} q) -1) + \equilq \ge 0,
\end{equation}
where~$\log=\log_e$, for which
\[
s_p'(p)=\log p -\log\left(\equilp-p\right)=\log\left(\frac{p}{\equilp- p}\right),
\qquad
s_q'(q)=\log \frac{q}{\equilq},
\] 
and
\[
s_p''(p)=\frac{1}{p(1-\equilp^{-1} p)},
\qquad
s_q''(q)=\frac{1}{q}.
\]

The changes of variables in~\eqref{EQ::HETER_UPUQ} can be rigorously defined using the functional setting in~\cite[App.~1]{Keith-Surowiec-2024}. To this end, we recall the following spaces (see~\cite[Prop.~A.7]{Keith-Surowiec-2024}):
\begin{alignat*}{3}
\mathrm{int}\, L_{(0,\equilp)}^\infty(\Omega) & \coloneqq  \big\{\mu \in L^{\infty}(\Omega) \ : \ \text{there exists~$\epsilon > 0$ such that~$\epsilon < \mu < \equilp - \epsilon$} \big\} = \up(L^{\infty}(\Omega)), \\
\mathrm{int}\, L_+^{\infty}(\Omega) & \coloneqq  \big\{\mu \in L^{\infty}(\Omega) \ : \ \text{there exists~$\epsilon > 0$ such that~$\mu > \epsilon$} \big\} = \uq(L^{\infty}(\Omega)). 
\end{alignat*}
Due to~\cite[Prop.~A.8]{Keith-Surowiec-2024}, we have that~$\up: H^1(\Omega) \cap L^{\infty}(\Omega) \to H^1(\Omega) \cap \mathrm{int}\, L_{(0,\equilp)}^\infty(\Omega)$ and~$\uq: H^1(\Omega) \cap L^{\infty}(\Omega) \to H^1(\Omega) \cap \mathrm{int}\, L_+^\infty(\Omega)$ are isomorphisms, and the following identities hold:
\begin{equation}\label{EQ::HETER_CHAINRULE}
\nabla w_p=s_p''(p)\nabla p, \quad
\nabla w_q=s_q''(q)\nabla q \qquad \forall \wp, \wq \in H^1(\Omega) \cap L^{\infty}(\Omega),
\end{equation}
which follow from the chain rule and the relations~$\wp=s_p'(p)$ and~$\wq=s_q'(q)$. 
Local versions of the identities in~\eqref{EQ::HETER_CHAINRULE} are employed in the definition of the proposed LDG method.

From assumption~\eqref{eq:ic} on the initial data, it follows that the initial entropy is bounded, namely
\begin{equation}\label{eq:initial_entropy}
\int_\Omega s_p(p_0)\dx +\int_\Omega s_q(q_0)\dx <+\infty.
\end{equation}

\begin{remark}[The boundedness-by-entropy setting]\label{rem:math_entropy}
Define~$\brho\coloneqq (p,q)$, $\bff(\brho)\coloneqq (f_p(p,q),f_q(p,q))$, $s(\brho)\coloneqq s_p(p)+s_q(q)$, and the set of physically meaningful values~$\mathcal{R}\coloneqq (0,\equilp)\times(0,+\infty)$.
With the specific entropy densities defined in~\eqref{EQ::HETER_SPSQ}, the problem in~\eqref{EQN::HETERODIMER} lies within the boundedness-by-entropy framework of~\cite{Jungel:2015} in a suitably relaxed sense as in~\cite{Jungel_Zurek:2021}.
In particular, the following properties hold:
\begin{enumerate}[label = \roman*), ref = \roman*)]
\item \label{prop:i} $s\in C^2(\mathcal{R};[0,+\infty))\cap C^0(\overline{\mathcal{R}};[0,+\infty))$,  with~$s':\mathcal{R}\to \R^2$ invertible;
\item \label{prop:ii} $s_p$ and~$s_q$ are convex functions and, for a positive constant~$c_q$, $s_q(q)\ge q-c_q$ for all~$q\in (0,+\infty)$;
\item \label{prop:iii} the initial condition~$\brho_0\coloneqq (p_0,q_0)$ is such that
$\int_\Omega s(\brho_0)\dx <+\infty$,
see~\eqref{eq:initial_entropy};
\item\label{prop:iv} there exist positive constants~$C_p$ and~$C_q$ such that
\[
\begin{split}
\vz^\top \left(s_p''(p)\D\right)\vz
&\ge C_p|\vz|^2 \qquad \forall\vz\in\R^d,\\
(\nabla q)^\top \left(s_q''(q)\D\right)\nabla q&\ge C_q|\nabla \sqrt{q}|^2,
\end{split}
\]
see~\eqref{EQ::HETER_UPD1} and~\eqref{EQ::HETER_SQRTQ} below;
\item \label{prop:v} there exists a positive constant~$C$ such that
$\bff(\brho)\cdot s'(\brho)\le C\left(1+s(\brho)\right)$,
see~Proposition~\ref{PROP::HETER_FBOUND}.
\end{enumerate}
From~\ref{prop:i}--\ref{prop:v}, we derive an entropy stability estimate for system~\eqref{EQN::HETERODIMER}; see Theorem~\ref{TH::HETER_STAB}.
\par
As discussed in~\cite[\S1.1]{Jungel:2015}, when condition~\ref{prop:v} is satisfied with~$C=0$, then~$S[\brho](t)\coloneqq \int_\Omega s\left(\brho(\bx,t)\right)\dx$ is a Lyapunov functional along the solutions to the system. 
In contrast, Proposition~\ref{PROP::HETER_FBOUND} below shows that, for our model, property~\ref{prop:v} holds with a strictly positive constant~$C$, since the production rate~$\kappa_p$, the conversion rate~$\mu_{pq}$, and~$\Upsilon_{pq}$ are assumed to be strictly positive.
As a consequence, $S[\brho]$ cannot be interpreted as a classical Lyapunov functional. Instead, it primarily serves  as a mathematical tool to guarantee boundedness and stability of solutions. 
\eremk
\end{remark}

\begin{remark}[Irreversibility and gradient-flow structure]
System~\eqref{EQN::HETERODIMER} does not exhibit a simple gradient-flow structure. Indeed, the reaction mechanism contains the irreversible conversion~$p\overset{\mu_{pq}}{\rightharpoonup}q$ without a reverse reaction. 
Consequently, compared to~\cite{Liu_Convergence_2022,Fu_High_2023}, the system does not satisfy a detailed-balance or reversibility condition that would typically allow one to 
identify a Lyapunov functional whose variational derivative reproduces the full dynamics.
 Specifically, since~$\partial_q f_p(p,q)=-\mu_{pq}p$ and~$\partial_p f_q(p,q)=\mu_{pq}q$, we have~$\partial_q f_p\ne \partial_p f_q$. %
 As a result, while the diffusion component alone lies within the classical gradient-flow framework, the full reaction--diffusion system~\eqref{EQN::HETERODIMER} cannot be expressed as a gradient flow with respect to the~$L^2$-metric. Additionally, the  conversion term $(-\mu_{pq} p q, \mu_{pq} p q)$ does not correspond to detailed balance, %
 which may prevent the possibility of a gradient-flow formulation in a metric space of Wasserstein- or Hellinger--Kantorovic-type measures\cite{maas_modeling_2020}.
\eremk
\end{remark}

\subsection{Auxiliary variables and problem reformulation}\label{SEC::REFORMULATION}
In view of the discretization introduced in Section~\ref{SEC::METHOD} below, we introduce the following auxiliary variables in~$Q_T$, whose motivation is discussed in detail in Remark~\ref{RMK::EXPLAIN} below:
\begin{subequations}
\label{EQ::HETERODIMER_VARIABLES}
\begin{align}
\label{EQ::HETER_PQ}
(w_p,w_q)\ \text{s.t.}\  (p,q) &\,= \left(u_p(w_p),u_q(w_q)\right), \\
\label{EQ::HETER_Z}
(\vz_p,\vz_q)  &\coloneqq  -(\nabla w_p,\nabla w_q), \\
\label{EQ::HETER_SIGMA1}
    \D s_p''(p)  \vsigma_p  &\coloneqq  - \D s_p''(p)  \nabla p = \D \vz_p, \quad \\
\label{EQ::HETER_SIGMA2}    
    \D s_q''(q) \vsigma_q  &\coloneqq  - \D s_q''(q)  \nabla q = \D \vz_q, \quad \\
    \label{EQ::HETER_Q}
    (\vr_p,\vr_q)  &\coloneqq  (\D \vsigma_p,\D\vsigma_q).
\end{align}
\end{subequations}
As~$s_p''$ is uniformly bounded away from zero by~$4\equilp^{-1}$,~$s_p''(p)\D$ is uniformly positive definite:
\begin{equation}\label{EQ::HETER_UPD1}
\vz^\top \left(s_p''(p)\D\right)\vz
\ge\left(\inf_{(\bx,t)\in Q_T}s_p''\left(p(\bx,t)\right)\right)
\,\dext|\vz|^2\ge 4\,\equilp^{-1}\dext|\vz|^2 \qquad \forall\vz\in\R^d.
\end{equation}
This is not the case for~$s_q''(q)\D$, since $s_q''$ is not bounded away from zero. However, we have
\begin{equation}\label{EQ::HETER_SQRTQ}
(\nabla q)^\top \left(s_q''(q)\D\right)\nabla q
=\frac{1}{q}(\nabla q)^\top \D\nabla q
\ge \dext\frac{|\nabla q|^2}{q}=4\, \dext |\nabla \sqrt{q}|^2,
\end{equation}
where, in the last step, we have used the identity~$|\nabla \sqrt{q}|^2 = \frac14 \frac{|\nabla q|^2}{q}$. A similar situation occurs for the Shigesada--Kawasaki--Teramoto (SKT) cross-diffusion system when the so-called~\emph{detailed-balance condition} does not hold; see, e.g., \cite{Chen_Jungel_Wang:2023}. 

\begin{remark}[Possible degeneracy]\label{RMK::SIGMA}
The first identities in equations~\eqref{EQ::HETER_SIGMA1} and~\eqref{EQ::HETER_SIGMA2} impose~$\vsigma_p= -\nabla p$ and~$\vsigma_q=-\nabla q$. This follows from the invertibility of $\D$ and the strict positivity of~$s_p''$ and~$s_q''$. However, while~$s_p''(p)\D$ is uniformly positive definite, see~\eqref{EQ::HETER_UPD1}, this is not true for~$s_q''(q)\D$ and, for large values of $q=u_q(w_q)$, the equation in~\eqref{EQ::HETER_SIGMA2} may degenerate. 
The lack of uniform positive definiteness of~$s_q''(q)\D$ is due to the choice of the entropy. This aspect is elaborated in Remark~\ref{rem:virtual_entropy} below.
\eremk
\end{remark}
\begin{remark}[Auxiliary variables]\label{RMK::EXPLAIN}
After the change of variables in~\eqref{EQ::HETER_PQ}, the definition of the auxiliary variables follows the standard LDG approach, with an adjustment to prevent nonlinearities from appearing under differential operators. 
This is done by explicitly imposing the chain rule~\eqref{EQ::HETER_CHAINRULE}, so that it is preserved in a weak sense at the discrete level.
More precisely, we define~$\vr_p\coloneqq \D\vsigma_p=-\D\nabla p$ and~$\vr_q\coloneqq \D\vsigma_q = -\D\nabla q$ (equation~\eqref{EQ::HETER_Q} and first identities in equations~\eqref{EQ::HETER_SIGMA1} and~\eqref{EQ::HETER_SIGMA2}). The second identities in equations~\eqref{EQ::HETER_SIGMA1} and~\eqref{EQ::HETER_SIGMA2}, together with equation~\eqref{EQ::HETER_Z}, impose the chain rule~\eqref{EQ::HETER_CHAINRULE}, avoiding the appearance of~$\nabla p=\nabla\left(u_p(w_p)\right)$ and~$\nabla q=\nabla\left(u_q(w_q)\right)$ in the formulation.
\eremk
\end{remark}
With the variables defined in~\eqref{EQ::HETERODIMER_VARIABLES}, %
problem~\eqref{EQN::HETERODIMER} can be rewritten as follows: find~$w_p,w_q: \QT \rightarrow \R$ and~$\vr_p,\vr_q: \QT \rightarrow \R^d$ such that
\begin{subequations}
\label{EQN::HETERODIMER_REWRITTEN}
\begin{alignat}{3}
\dpt p + \nabla \cdot \vr_p & = f_p\left(p,q\right) & & \quad \text{ in } \QT,\\
\dpt q + \nabla \cdot \vr_q & = f_q\left(p,q\right) & & \quad \text{ in } \QT,\\
\vr_p \cdot \bnOmega = 0  \ \text{ and } \ \vr_q \cdot \bnOmega & = 0 & & \quad  \text{ on } \Gamma \times (0, T), \\
p(\cdot, 0) = p_{0} \text{ and } q(\cdot, 0) & = q_{0} & & \quad \text{ in } \Omega,
\end{alignat}
\end{subequations}
where, in~$Q_T=\Omega\times (0,T)$,~$p=u_p(w_p)$ and~$q=u_q(w_q)$ are understood.
The unknowns~$\vr_p$ and~$\vr_q$, which appear explicitly in formulation~\eqref{EQN::HETERODIMER_REWRITTEN}, are the fluxes of~$p$ and~$q$, respectively; see Remark~\ref{RMK::EXPLAIN}.

\subsection{The structure-preserving backward Euler-LDG discretization}\label{SEC::METHOD}

In this section, we introduce our structure-preserving discretization of system~\eqref{EQN::HETERODIMER} based on the reformulation in~\eqref{EQN::HETERODIMER_REWRITTEN} in terms of the auxiliary variables in~\eqref{EQ::HETERODIMER_VARIABLES}.

\paragraph{Meshes.}
Let~$\{\Th\}_{h>0}$ be a family of conforming, locally quasi-uniform, simplicial partitions of the space domain~$\Omega$ with shape-regular elements. For any element~$K\in\Th$, we denote by~$h_K$ its diameter and by~$\bnK$ the unit normal~$d$-dimensional vector to~$\partial K$, pointing %
outward from~$K$. 
Moreover, we denote by~$(\partial K)^{\circ}$ the union of the facets of~$K$ that belong to~$\Fho$.
The index~$h$ in~$\Th$ represents the mesh size defined as~$h\coloneqq \max_{K\in\Th}h_K$. 
Let also~$\Tt$ be a partition of the time interval~$(0, T)$ of the form~$0 \coloneqq  t_0 < t_1 < \ldots < t_{N_t} \coloneqq  T$. For~$n = 1, \ldots, N_t$, we define the time interval~$\In \coloneqq  (\tnmo, \tn)$ and the time step~$\tau_n \coloneqq  \tn - \tnmo$. 
The subscript~$\tau$ in~$\Tt$ represents the mesh size defined as~$\tau\coloneqq \max_{1\le n\le N_t}\tau_n$.

\paragraph{Piecewise polynomial spaces.}
Given %
a degree of approximation in space~$\ell \in \N$ with~$\ell \geq 1$, we define the following (discontinuous) piecewise polynomial spaces of uniform degree:
\begin{equation*}
\Wh \coloneqq  \prod_{K \in \Th} \Pp{\ell}{K} \qquad \text{ and } \qquad \Rh \coloneqq  \prod_{K \in \Th} \Pp{\ell}{K}^d,
\end{equation*}
where~$\Pp{\ell}{K}$ denotes the space of scalar-valued polynomials of degree at most~$\ell$ defined on~$K$. Moreover, we denote by~$\PiW$ and~$\PiR$ the~$L^2(\Omega)$- and~$L^2(\Omega)^d$-orthogonal projections in~$\Wh$ and~$\Rh$, respectively.

\begin{remark}[Space of $d$-vector-valued polynomials]
In combination with~$\Wh$, the space~$\boldsymbol{\mathcal{R}}^{\ell-1}(\Th)$ can be used in place of~$\Rh$, thereby reducing the number of degrees of freedom without compromising accuracy~\cite{Castillo_etal:2002}. 
However, using the same polynomial basis for both spaces simplifies the computation of the discrete operators involved in the method. Moreover, some numerical studies suggest that both versions yield comparable efficiency (see, e.g., \cite{Cockburn_etal:2002}).
\eremk
\end{remark}

\paragraph{Mesh size function, stability parameters, jumps, and averages.}
We denote the set of all the mesh facets in~$\Th$ by~$\Fh = \Fho \cup \FhN$, where~$\Fho$ and~$\FhN$ are the sets of internal and (Neumann) boundary facets, respectively. We define the mesh size function~$\mathsf{h} \in L^{\infty}(\Fho)$ as 
\begin{equation}
\label{EQN::DEF-h}
\mathsf{h}(\bx) \coloneqq  \min\{h_{K_1}, h_{K_2}\} \quad \text{ if }\bx \in F, \text{ and~$F\in \Fho$ is shared by~$K_1, K_2\in \Th$},
\end{equation}

and the stability parameters
\[
\eta_F \coloneqq  \eta_0 \ell^2  \dfrac{2(\boldsymbol{n}_{K_1}^T\mathbf{D}_{|_{K_1}}\boldsymbol{n}_{K_1})(\boldsymbol{n}_{K_2}^T\mathbf{D}_{|_{K_2}}\boldsymbol{n}_{K_2})}{\boldsymbol{n}_{K_1}^T\mathbf{D}_{|_{K_1}}\boldsymbol{n}_{K_1}+\boldsymbol{n}_{K_2}^T\mathbf{D}_{|_{K_2}}\boldsymbol{n}_{K_2}} > 0,
\]
where~$\ell$ is the polynomial degree in the space discretization, and~$\eta_0>0$ is a constant independent of the problem coefficients and discretization parameters (in practice,~$\eta_0=\mathcal{O}(1)$). 

For any piecewise smooth, scalar-valued function~$\mu$, and any~$d$-vector-valued function~$\boldsymbol{\mu}$, we define the normal jumps and weighted mean values as follows:  
on each facet $F\in \Fho$ shared by $K_1, K_2 \in \Th$,
\begin{alignat*}{3}
\jump{{\mu}}_{\sf N} & \coloneqq  {{\mu}}_{|_{K_1}} \bn_{K_1} + 
{{\mu}}_{|_{K_2}} \bn_{K_2}, & \qquad 
\mvl{\mu}_{\mvlparamF}  & \coloneqq   (1 - \mvlparamF) \mu_{|_{K_1}} + \mvlparamF \mu_{|_{K_2}},\\
\jump{\boldsymbol{\mu}}_{\sf N} & \coloneqq  {\boldsymbol{\mu}}_{|_{K_1}} \cdot \bn_{K_1} + 
{\boldsymbol{\mu}}_{|_{K_2}} \cdot \bn_{K_2},
& \qquad
\mvl{\boldsymbol{\mu}}_{1-\mvlparamF} & \coloneqq   \mvlparamF \boldsymbol{\mu}_{|_{K_1}} + (1-\mvlparamF) \boldsymbol{\mu}_{|_{K_2}},
\end{alignat*}
with~$\mvlparamF\in [0,1]$. %
\paragraph{Discrete gradient and divergence operators in space.}
The discrete gradient operator~$\nablaLDG : \Wh \to \Rh$ is defined by 
\begin{equation}\label{EQ::DGGRAD}
\left(\nablaLDG \vh , \phih \right)_\Omega
=
\left(\nabla_h \vh - \mathcal{L}(\vh),  \phih \right)_\Omega
\quad\forall  \phih \in \Rh,
\end{equation}
where~$\nabla_h$ denotes the piecewise gradient operator defined element-by-element, and the jump lifting operator $\mathcal{L}:
\Wh \to \Rh$ is given by
\[
\left(\mathcal{L}(\vh), \phih \right)_\Omega
=\sum_{F\in\Fho} \left(\jump{\vh}_{\sf N}, \mvl{\phih}_{1-\gamma_F}\right)_F
\quad\forall  \phih \in \Rh.
\]
Accordingly, the discrete divergence operator~$\divLDG: \Rh \to \Wh$ is defined by
\begin{equation}\label{EQ::DGdiv}
\left(\divLDG \rh, \psih\right)_\Omega
=
- \left(\rh, \nablaLDG \psih\right)_{\Omega}
\quad\forall  \psih \in \Wh.
\end{equation}
These definitions correspond to choosing the numerical fluxes in the LDG method using weighted averages, with weight~$\gamma_F$ for the scalar unknowns and~$(1-\gamma_F)$ for vector-valued unknowns, thus preserving the symmetry of the discretization of the second-order differential operator.

\paragraph{The backward Euler-LDG method.}
We fix a penalty parameter~$\varepsilon > 0$, whose role in the method is described in Remark~\ref{RMK:PENALTY-TERM} below. The method is defined as follows: for~$n = 0, \ldots, N_t - 1$, find~$\whp^{\npo}, \whq^{\npo} \in \Wh$ and~$\zhp^{\npo}$, $\zhq^{\npo}$, $\sigmahp^{\npo}$, $\sigmahq^{\npo}$, $\rhp^{\npo}$, $\rhq^{\npo} \in \Rh$ (where the dependence on~$\varepsilon$ has been omitted for brevity), such that, for~$\star=p,q$,
\begin{subequations}
\label{EQ::VARIATIONAL}
\begin{alignat}{3}
\zhstar^{\npo} &= -\nablaLDG \whstar^{\npo}, \\
\label{EQ::VARIATIONAL_SIGMA}
\left(\D\sstar''\big(\ustar(\whstar^{\npo})\big)  \sigmahstar^{\npo},\, \phih\right)_{\Omega} & = \left(\D \zhstar^{\npo},\, \phih\right)_{\Omega} & & 
\qquad \forall \phih \in \Rh, \\
\rhstar^{\npo} &= \PiR \big(\D \sigmahstar^{\npo}\big),\\
\nonumber
\varepsilon \left(\whstar^{\npo},\, \psih \right)_{\LDG} + \frac{1}{\tau_{n + 1}}\big(\ustar(\whstar^{\npo}) & - \ustarh^{\n},\, \psih \big)_{\Omega}  &  \\
\nonumber
+ \left(\divLDG \rhstar^{\npo},\, \psih\right)_{\Omega} 
+  \sum_{F \in \Fho} \big(& \eta_F \mathsf{h}^{-1} \jump{\whstar^{(n + 1)}}_{\sf N}, \jump{\psih}_{\sf N} \big)_{F} \\
\label{EQ::VARIATIONAL_EQ}
& = \left(\fstar\big(\up(\whp^{\npo}), \uq(\whq^{\npo})\big),\, \psih \right)_{\Omega} & & 
\qquad \forall \psih \in \Wh.
\end{alignat}
\end{subequations}

In~\eqref{EQ::VARIATIONAL_EQ}, the quantities~$\ustarh^{\n}$ transmitted from the previous time step are defined as follows: 
\begin{alignat*}{6}
\uph^{(n)} & \coloneqq  \begin{cases}
\PiW p_0 & \text{ if } n = 0, \\
\PiW \up(\whp^{\n}) & \text{ otherwise},
\end{cases}
\quad \text{ and } \quad
\uqh^{(n)} & \coloneqq  \begin{cases}
\PiW q_0 & \text{ if } n = 0, \\
\PiW \uq(\whq^{\n}) & \text{ otherwise},
\end{cases}
\end{alignat*}
namely, at the first time step, they are the~$L^2(\Omega)$ projections of the initial conditions for the \emph{original} variables~$p$ and~$q$; at the subsequent time steps, they are the~$L^2(\Omega)$ projections of the \emph{transformed} variables computed at the previous time step.
Additionally, $(\cdot, \cdot)_{\LDG}$ is a bilinear form in the space~$\Wh$, which is coercive with respect to a norm~$\Norm{\cdot}{\DG}$. While the existence of discrete solutions is guaranteed, for instance, with the choice~$\Norm{\cdot}{\DG}=\Norm{\cdot}{L^2(\Omega)}$ and~$(\cdot, \cdot)_{\LDG}=(\cdot, \cdot)_{L^2(\Omega)}$, specific assumptions on~$\Norm{\cdot}{\DG}$ are required for the convergence analysis; see Section~\ref{sec:convergence} below.

\begin{remark}[Role of the penalty term]
\label{RMK:PENALTY-TERM}
The penalty term with parameter~$\varepsilon > 0$ in~\eqref{EQ::VARIATIONAL_EQ} is introduced to prevent~$u_p(w_p)$ from approaching the extreme values~$0$ and~$\equilp$, and~$u_q(w_q)$ from approaching~$0$. This is necessary because 
$\sp''(\cdot)$ and $\sq''(\cdot)$ in~\eqref{EQ::VARIATIONAL_SIGMA} become singular at these limits. %
The penalty term thus plays a crucial role in the analysis of the method: first, in proving the existence of discrete solutions in Theorem~\ref{thm:existence_discrete}, and subsequently in the compactness argument 
for the~$h$-convergence result in Theorem~\ref{thm:h-convergence}. From a numerical perspective, it 
enhances the stability and convergence of nonlinear solvers, such as the Newton method described in Appendix~\ref{SEC::NEWTON} below.
\eremk
\end{remark}

\begin{remark}[Equation~\eqref{EQ::VARIATIONAL_SIGMA} uniquely determines~$\sigmahstar^{\npo}$]\label{RMK:unique_sigma}
As~$s_p''(p)\D$ is uniformly positive definite, see~\eqref{EQ::HETER_UPD1}, given~$\whp^{\npo}$ and~$\zhp^{\npo}$, equation~\eqref{EQ::VARIATIONAL_SIGMA} with~$\star=p$ determines~$\sigmahp^{\npo}$ in a unique way; see also~Remark~\ref{RMK::SIGMA}.
Given~$\whq^{\npo}$, and choosing~$\phiqh = \sigmahq^{\npo}$ in the term on the left-hand side of~\eqref{EQ::VARIATIONAL_SIGMA} with~$\star=q$, we get
\begin{equation}\label{EQ::COER_SIGMA}
\left(\D \sq''\big(\uq(\whq^{\npo})\big) \sigmahq^{\npo},\, \sigmahq^{\npo} \right)_{\Omega} = \int_{\Omega} \frac{1}{\uq(\whq^{\npo})} \D \sigmahq^{\npo} \cdot \sigmahq^{\npo} \dx \geq \dext \int_{\Omega} \frac{|\sigmahq^{\npo}|^2}{\uq(\whq^{\npo})} \dx.
\end{equation}
Therefore, since~$\dext > 0$ and~$\uq(\whq^{\npo}) > 0$, if the right-hand side of~\eqref{EQ::VARIATIONAL_SIGMA} with~$\star=q$ is equal to zero, then~$\sigmahq = 0$, which implies that~\eqref{EQ::VARIATIONAL_SIGMA} with~$\star=q$ determines~$\sigmahq^{\npo}$ in a unique way. However, as observed in Remark~\ref{RMK::SIGMA}, for large values of~$\uq(\whq^{\npo})$, the linear system may degenerate. Such a degeneracy must be prevented using the penalty term~$\varepsilon (\cdot,\, \cdot)_{\LDG}$; see Remark~\ref{RMK:PENALTY-TERM}.
\eremk
\end{remark}

\begin{remark}[Alternative changes of variable]\label{rem:virtual_entropy}
For the design of method~\eqref{EQ::VARIATIONAL}, the knowledge of an explicit expression of the entropy densities~$s_p$ and~$s_q$ is not actually required, and only the transformations~$u_p$ and~$u_q$ need to be known in order to implement the method.
Indeed, the construction of the scheme only involves~$s_{\star}''(\ustar(w_{\star}))$, $\star=p,q$, which can be computed in terms of the change-of-variable functions~$u_\star$ as~$s_{\star}''(\ustar(w_{\star}))=1/\ustar'(w_{\star})$, since
\[
\ustar'(w_{\star}) = ((\sstar')^{-1})'(w_{\star}) = \frac{1}{s_{\star}''((s_{\star}')^{-1}(w_\star))} = \frac{1}{s_{\star}''(\ustar(w_{\star}))}.
\]
This opens the possibility of using alternative changes of variable, whose corresponding entropy densities do not have a closed-form expression.
For instance, for the unbounded unknown~$q$, one could use the softplus transformation~$u_q(w_q) = \log(1+e^{\wq})$, or the Kaniadakis~$\kappa$-exponential transformation~$u_q(w_q) = (\sqrt{1+\kappa^2 x^2} + \kappa x)^{1/{\kappa}}$ with~$k \in (0, 1]$, see~\cite{Kaniadakis:2013}. %
Using the softplus transformation would lead to the advantage that~$\sq''(q)$ becomes uniformly bounded from below with
\[
s''(u_q(w_q))=1+e^{-{w_q}}>1,
\]
thus overcoming the degeneracy in~\eqref{EQ::HETER_SQRTQ} of the diffusive term, and providing a direct estimate on~$\Norm{\nabla q}{L^2(\Omega)^d}^2$ rather than~$\Norm{\nabla \sqrt{q}}{L^2(\Omega)^d}^2$.

Nevertheless, in the analysis of the method, the explicit forms of~$s_p$ and~$s_q$ appear to play an important role due to the presence of nonlinear reaction terms. 
As summarized in Remark~\ref{rem:math_entropy}, although most of the required properties only involve derivatives of~$s_p$ and~$s_q$, 
property~\ref{prop:v} depends on the specific structure of the underlying entropy functions~$s_p$ and~$s_q$. For this reason, this paper considers the transformations in~\eqref{EQ::HETER_UPUQ}, whose corresponding entropy densities~\eqref{EQ::HETER_SPSQ} are shown to satisfy this property in Proposition~\ref{PROP::HETER_FBOUND}.
\eremk
\end{remark}

\subsection{Matrix form}\label{SEC::MATRIX}

For~$\star=p,q$, we define the following forms and functionals: for all~$\wh,\psih\in\Wh$ and~$\zh,\sigmah,\phih\in\Rh$,
\begin{subequations}
\label{EQ::FORMS}
\begin{alignat}{2}
{\sf m}_h(\zh,\phih)&\coloneqq \left(\zh,\phih\right)_\Omega,\\
{\sf b}_h(\wh,\phih)&\coloneqq \left(\nabla_h \wh,  \phih \right)_\Omega-\sum_{F\in\Fho} \left(\jump{\wh}_{\sf N}, \mvl{\phih}_{1-\gamma_F}\right)_F,\\
{\sf j}_h(\wh,\psih)&\coloneqq \sum_{F \in \Fho} \big(\eta_F \mathsf{h}^{-1} \jump{\wh}_{\sf N}, \jump{\psih}_{\sf N} \big)_{F},\\
{\sf d}_h(\wh,\phih)&\coloneqq \sum_{K\in\Th}\left(\D\zh,\phih\right)_K,\\
{\sf n}_{\star,h}(\wh;\sigmah,\phih)&\coloneqq \sum_{K\in\Th}\left(\D s_\star''\left(\ustar(\wh)\right)\sigmah, \phih\right)_K,\\
{\sf u}_{\star,h}(\wh,\psih)&\coloneqq \left(\ustar(\wh),\psih\right)_\Omega,\\
{\sf f}_{\star,h}(\whp,\whq;\psih)&\coloneqq \left(f_\star\left(u_p(\whp),u_q(\whq)\right),\psih\right)_\Omega.
\end{alignat}
\end{subequations}
We fix the bases of~$\Wh$ and~$\Rh$. For~$\star=p,q$, we denote by~$\cvWstar$, $\cvZstar$, $\cvSigmastar$, and~$\cvRstar$ the vectors of expansion coefficients expressing~$\whstar$, $\zhstar$, $\sigmahstar$, $\rhstar$, respectively, in the chosen bases. 
Additionally, we denote by~$M_I$, $B$, $J$, $M_D$, and~$A_{\LDG}$ the matrices associated with the bilinear forms~${\sf m}_h(\cdot,\cdot)$, ${\sf b}_h(\cdot,\cdot)$, ${\sf j}_h(\cdot,\cdot)$, ${\sf d}_h(\cdot,\cdot)$, and~$(\cdot,\cdot)_\LDG$, respectively, and by~${\mathcal N}_\star$,\, ${\mathcal U}_\star$, and~${\mathcal F}_\star$ the operators associated with the nonlinear functionals~${\sf n}_{\star,h}(\cdot\,;\cdot,\cdot)$, ${\sf u}_{\star,h}(\cdot,\cdot)$, and~${\sf f}_{\star,h}(\cdot,\cdot\,;\cdot)$, respectively. Since the nonlinear functional~${\sf n}_{\star,h}(\cdot\,;\cdot,\cdot)$ is linear with respect to its second argument, we write~$\mathcal{N}_\star(\cvWstar^{\npo}) \cvSigmastar^{\npo}$ instead of~$\mathcal{N}_\star\big(\cvWstar^{\npo}; \cvSigmastar^{\npo}\big)$.

With the above notation, the scheme in~\eqref{EQ::VARIATIONAL} can be written in matrix form as follows:
\begin{subequations}
\label{EQ::MATRICIAL}
\begin{alignat}{3}
M_I \cvZstar^{\npo} &= -B \cvWstar^{\npo}, \\
\mathcal{N}_{\star}\big(\cvWstar^{\npo}\big) \cvSigmastar^{\npo}  & = M_D \cvZstar^{\npo},\\
M_I \cvRstar^{\npo} & = M_D \cvSigmastar^{\npo},\\
\label{EQ::MATRICIAL-WH}
\varepsilon A_{\LDG} \cvWstar^{\npo} + \frac{1}{\tau_{n + 1}}\mathcal{U}_{\star}(\cvWstar^{\npo})  - B^T \cvRstar^{\npo} +  J \cvWstar^{(n + 1)} & = \mathcal{F}_{\star} \big(\cvWp^{\npo}, \cvWq^{\npo}) +\frac{1}{\tau_{n + 1}}\mathcal{U}_{\star, h}^{(n)},
\end{alignat}
\end{subequations}
where, in~\eqref{EQ::MATRICIAL-WH}, for~$\star=p,q$, 
\[
\text{$\mathcal{U}_{\star, h}^{(n)}$
is the coefficient vector of~$\PiW (\star)_0$ if~$n = 0$, or the coefficient vector of
$\PiW \ustar(\whstar^{\n})$ if~$n > 0$.}
\]

Since the mass matrix~$M_I$ is symmetric, positive definite, and block diagonal, it is easy to invert. Therefore, representing~$\cvZstar^{\npo}$ and~$\cvRstar^{\npo}$ as
\[
\cvZstar^{\npo}=-M_I^{-1}B\cvWstar^{\npo}
\quad\text{and}\quad
\cvRstar^{\npo} = M_I^{-1}M_D \cvSigmastar^{\npo},
\]
system~\eqref{EQ::MATRICIAL} becomes
\begin{subequations}
\label{EQ::MATRICIAL_COMPACT}
\begin{alignat}{3}
\label{EQ::MATRICIAL_COMPACT_SIGMA}
\mathcal{N}_{\star}\big(\cvWstar^{\npo}\big) \cvSigmastar^{\npo}  & = -M_D M_I^{-1} B \cvWstar^{\npo},\\
\nonumber
\varepsilon A_{\LDG} \cvWstar^{\npo} + \frac{1}{\tau_{n + 1}}\mathcal{U}_{\star}(\cvWstar^{\npo})   - B^T M_I^{-1} M_D \cvSigmastar^{\npo} & +  J \cvWstar^{(n + 1)} \\
\label{EQ::MATRICIAL_COMPACT_W}
& = \mathcal{F}_{\star} \big(\cvWp^{\npo}, \cvWq^{\npo}) +\frac{1}{\tau_{n + 1}}\mathcal{U}_{\star, h}^{(n)}.
\end{alignat}
\end{subequations}

\begin{remark}[Formulation in terms of the~$\cvW$ unknowns only]\label{RMK::UNIQUESIGMA}
Given~$\cvWstar^{\npo}$, equation~\eqref{EQ::MATRICIAL_COMPACT_SIGMA} determines~$\cvSigmastar^{\npo}$ in a unique way; see Remark~\ref{RMK:unique_sigma}. Then, using
$\cvSigmastar^{\npo}=-\big(\mathcal{N}_{\star}\big(\cvWstar^{\npo}\big)\big)^{-1} M_D M_I^{-1}B \cvWstar^{\npo}$,
system~\eqref{EQ::MATRICIAL_COMPACT} can be reformulated in terms of the~$\cvWstar^{\npo}$ unknowns only as
\begin{equation}\label{EQ::MATRICIAL_W}
\begin{split}
\varepsilon A_{\LDG} \cvWstar^{\npo} + \frac{1}{\tau_{n + 1}}\mathcal{U}_{\star}(\cvWstar^{\npo})+B^TM_I^{-1}M_D \big(\mathcal{N}_{\star}\big(\cvWstar^{\npo}\big)\big)^{-1} M_D M_I^{-1}B \cvWstar^{\npo}+J\cvWstar^{\npo}
\\
=\mathcal{F}_{\star} \big(\cvWp^{\npo}, \cvWq^{\npo}) +\frac{1}{\tau_{n + 1}}\mathcal{U}_{\star, h}^{(n)},
\end{split}
\end{equation}
thus reducing the size of the global nonlinear system.
\eremk
\end{remark}

\section{Stability and existence of discrete solutions}\label{SEC::ANALYSIS}

In this section, we start by establishing a stability property of the system in~\eqref{EQN::HETERODIMER} (Section~\ref{SEC::ENTROPY_CONT}). Then, we derive a discrete analogue of this stability property for the backward Euler-LDG method (Section~\ref{SEC::ENTROPY_DISCR}), which we use to prove the existence of discrete solutions (Section~\ref{sec:EXISTENCE-DISCRETE}). 

\subsection{Entropy stability of the continuous problem}\label{SEC::ENTROPY_CONT}

We now derive an entropy stability estimate for problem~\eqref{EQN::HETERODIMER}. Before doing so, we prove the following bound for the reaction terms, which we write using the notation introduced in~\eqref{EQ::HETER_FPFQ}. 

\begin{proposition}[A bound for the reaction term]\label{PROP::HETER_FBOUND}
For all~$p\in \mathrm{int}\, L_{(0,\equilp)}^\infty(\Omega)$ and~$q\in \mathrm{int}\, L_+^{\infty}(\Omega)$,
we have
\begin{equation}\label{EQ::HETER_BOUNDF}
\fp(p,q)s_p'(p)+\fq(p,q)s_q'(q)\le
C_f\left(\equilq+s_q(q)\right),
\end{equation}
with~$C_f \coloneqq  \frac{1+\log 2}{\log 2}\,\equilp\left(\frac{\lambda_p}{\equilq}+2\mu_{pq}\right) =\frac{1+\log 2}{\log 2}\,\frac{\kappa_p\mu_{pq}}{\lambda_p\Upsilon_{pq}}\left(\kappa_p\mu_{pq}+\Upsilon_{pq}\right)$. \end{proposition}
\begin{proof}
We compute 
\[
\begin{split}
\fp(p,q)s_p'(p)+\fq(p,q)s_q'(q)=&
(-\lambda_p p+\kappa_p)\log\left(\frac{p}{\equilp- p}\right)
-\mu_{pq} pq\, \log\left(\frac{p}{\equilp- p}\right)
-q(\lambda_q-\mu_{pq} p)\log \frac{q}{\equilq}\\
=&: J_1+J_2+J_3,
\end{split}
\]
and proceed to estimate~$J_1$, $J_2$, and~$J_3$.
Recalling that~$0 < p <\equilp=\kappa_p/\lambda_p$, we have
\[
J_1=\lambda_p p\left(1-\frac{\kappa_p}{\lambda_p p}\right)\log\left(\frac{\equilp-p}{p}\right) =-\lambda_p p\left(\frac{\equilp-p}{p}\right)\log\left(\frac{\equilp-p}{p}\right)
\le \frac{\lambda_p p}{e}
< \frac{\kappa_p}{e},
\]
where we have used the inequality~$-\rho\log\rho\le 1/e$ for all~$\rho> 0$.

For~$J_2$, using again that $0 < p <\equilp$ and~$-\rho\log\rho \le 1/e$ for all~$\rho > 0$, and the inequality~$\rho \log(1 - \rho) \le 0$ for all~$0 < \rho < 1$,
we get
\begin{alignat*}{3}
J_2 
= -\mu_{pq} pq \log\Big(\frac{\equilp^{-1} p}{1 - \equilp^{-1} p} \Big) 
&   = \mu_{pq} q\equilp  \Big[\equilp^{-1} p \log(1-\equilp^{-1} p) - \equilp^{-1} p \log(\equilp^{-1} p) \Big] \\
& \le \frac{\equilp\mu_{pq}}{e}q \le \frac{\equilp\mu_{pq}}{e\log 2}
\left(\equilq+s_q(q)\right),
\end{alignat*}
where, in the last inequality, we have used the bound~$q\le \left(\equilq+s_q(q)\right)/\log 2$, %
which follows from~$(x-1)\ge \log x$ for all~$x>0$.

Finally, for~$J_3$, recalling the definition of~$s_q(q)$ in~\eqref{EQ::HETER_SPSQ}, we have
\[
\begin{split}
J_3 %
=(\mu_{pq}p-\lambda_q)\big(s_q(q)+q-\equilq\big)
\le \frac{\Upsilon_{pq}}{\lambda_p}\,\big(s_q(q)+q\big)
+\equilq\lambda_q-\equilq\mu_{pq} p\\
\le \frac{1+\log 2}{\log 2}\,\frac{\Upsilon_{pq}}{\lambda_p}\left(\equilq+s_q(q)\right)+\equilq\lambda_q,
\end{split}
\]
where, in the last inequality, we have used again~$q\le \left(\equilq+s_q(q)\right)/\log 2$.
Collecting the previous estimates of~$J_1$, $J_2$, and~$J_3$ gives
\[
\begin{split}
\fp(p,q)s_p'(p)+\fq(p,q)s_q'(q) &\le
\equilq\left(\frac{\kappa_p}{\equilq e}+\lambda_q\right)+
\left(\frac{\equilp\mu_{pq}}{e\log 2}+\frac{1+\log 2}{\log 2}\,\frac{\Upsilon_{pq}}{\lambda_p}
\right)
\left(\equilq+s_q(q)\right)\\
&\le \left(\frac{\kappa_p}{\equilq e}+\lambda_q+\frac{\equilp\mu_{pq}}{e\log 2}+\frac{1+\log 2}{\log 2}\,\frac{\Upsilon_{pq}}{\lambda_p}\right)\left(\equilq+s_q(q)\right)\\
&\le \frac{1+\log 2}{\log 2}\left(\frac{\kappa_p}{\equilq}+\lambda_q+\equilp\mu_{pq}+\frac{\Upsilon_{pq}}{\lambda_p}\right)\left(\equilq+s_q(q)\right)\\
&=\frac{1+\log 2}{\log 2}\,\equilp\left(\frac{\lambda_p}{\equilq}+2\mu_{pq}\right)\left(\equilq+s_q(q)\right)
=:C_f\left(\equilq+s_q(q)\right),
\end{split}
\]
and the proof is complete.
\end{proof}

We prove the following \emph{local} entropy-stability 
estimate. 

\begin{theorem}[Local entropy stability]\label{TH::HETER_STAB}
Let~$t\in [0,T]$ be such that~$t<1/C_f$, with~$C_f$ as in Proposition~\ref{PROP::HETER_FBOUND}. Then, every solution~$(p,q)$ to~\eqref{EQN::HETERODIMER} satisfies the following stability estimate: 
\begin{equation}\label{EQ::HETER_STAB}
\begin{split}
\int_\Omega 
\left(s_p\left(p(\cdot, t)\right)+
s_q\left(q(\cdot,t)\right)\right)\dx
+4 \dext
\left(
\equilp^{-1}\int_0^{t}\Norm{\nabla p (\cdot, \tau)}{L^2(\Omega)^d}^2 \dtau 
+\int_0^{t}\Norm{\nabla \sqrt{q(\cdot, \tau)}}{L^2(\Omega)^d}^2 \dtau
\right)\\
\le\left(\int_\Omega \left(s_p\left(p_0\right)+s_q\left(q_0\right)\right)\dx + t \equilq C_f|\Omega| \right)e.
\end{split}
\end{equation}
\end{theorem}

\begin{proof}
By testing the first equation in~\eqref{EQN::HETERODIMER} with~$w_p\,e^{-\tau/t}=s_p'(p)\,e^{-\tau/t}$, and using the chain rule~\eqref{EQ::HETER_CHAINRULE}, we obtain
\begin{equation}\label{EQ::HETER_WEAK1}
\begin{split}
\int_0^{t}\int_\Omega f_p(p,q)\,s_p'(p)\,e^{-\tau/t}\dx\dtau & =
\int_0^{t}\int_\Omega f_p(p,q)\,w_p\,e^{-\tau/t}\dx\dtau \\
&=
\int_0^{t}\int_\Omega \dptau p\, w_p\,e^{-\tau/t}\dx\dtau- \int_0^{t}\int_\Omega\nabla \cdot (\D \nabla p)\,w_p\,e^{-\tau/t}\dx\dtau\\
&=\int_0^{t}\int_\Omega\dptau p\,s_p'(p)\,e^{-\tau/t}\dx\dtau+\int_0^{t}\int_\Omega (\D \nabla p)\cdot\nabla w_p\,e^{-\tau/t}\dx\dtau\\
& = \int_0^{t}\int_\Omega\dptau \left(s_p(p)\,e^{-\tau/t}\right)\dx\dtau
+\frac{1}{t}\int_0^{t}\int_\Omega s_p(p)\,e^{-\tau/t}\dx\dtau\\
&\qquad +\int_0^{t}\int_\Omega \left(s_p''(p)\D \nabla p\right)\cdot\nabla p\,e^{-\tau/t}\dx\dtau\\
& = \frac{1}{e}\int_\Omega s_p\left(p(\cdot,t)\right) \dx -\int_\Omega s_p\left(p_0 \right) \dx +\frac{1}{t}\int_0^{t}\int_\Omega s_p(p)\,e^{-\tau/t}\dx\dtau\\
&\qquad + \int_0^{t}\int_\Omega \left(s_p''(p)\D \nabla p\right)\cdot\nabla p \,e^{-\tau/t}\dx \dtau.
\end{split}
\end{equation}
Similarly, from the second equation in~\eqref{EQN::HETERODIMER} tested with~$w_q\,e^{-\tau/t}=s_q'(q)\,e^{-\tau/t}$, we get
\begin{equation}\label{EQ::HETER_WEAK2}
\begin{split}
\int_0^{t}\int_\Omega f_q(p,q)\,s_q'(q)\,e^{-\tau/t} \dx \dtau &=
\frac{1}{e}\int_\Omega s_q\left(q(\cdot,t)\right) \dx -\int_\Omega s_q\left(q_0\right) \dx +\frac{1}{t}\int_0^{t}\int_\Omega s_q(q)\,e^{-\tau/t}\dx\dtau\\
& \qquad + \int_0^{t}\int_\Omega \left(s_q''(q)\D \nabla q\right)\cdot\nabla q \,e^{-\tau/t}\dx \dtau.
\end{split}
\end{equation}
Summing~\eqref{EQ::HETER_WEAK1} and~\eqref{EQ::HETER_WEAK2}, and using the bound in~\eqref{EQ::HETER_BOUNDF}, we obtain
\begin{equation}\label{EQ::HETER_EST1}
\begin{split}
&\frac{1}{e}\int_\Omega 
\left(s_p\left(p(\cdot,t)\right)+s_q\left(q(\cdot,t)\right)\right)\dx
 +\int_0^{t}\int_\Omega
\left[\left(s_p''(p)\D \nabla p\right)\cdot\nabla p+
\left(s_q''(q)\D \nabla q\right)\cdot\nabla q
\right]\,e^{-\tau/t}\dx\dtau\\
&\qquad +\frac{1}{t}\int_0^{t}\int_\Omega \left(s_p(p)+s_q(q)\right)\,e^{-\tau/t}\dx\dtau\\
&\qquad\qquad \le\int_\Omega \left(s_p\left(p_0\right)  + s_q\left(q_0\right)\right) \dx 
+ C_f\int_0^{t}\int_\Omega \left(\equilq+s_q\left(q\right)\right) \,e^{-\tau/t}\dx \dtau \\
& \qquad\qquad\le\int_\Omega \left(s_p\left(p_0\right)  + s_q\left(q_0\right)\right) \dx 
+ t\equilq C_f |\Omega|
+ C_f\int_0^{t}\int_\Omega \left(s_p\left(p\right)+s_q\left(q\right)\right) \,e^{-\tau/t}\dx \dtau.
\end{split}
\end{equation}
Thus, estimate~\eqref{EQ::HETER_STAB} is obtained by inserting~\eqref{EQ::HETER_UPD1} and~\eqref{EQ::HETER_SQRTQ} into~\eqref{EQ::HETER_EST1}, and using the inequalities~$\frac{1}{t}-C_f>0$
and~$e^{-\tau/t}\ge 1/e$ for all~$\tau\in[0,t]$.
\end{proof}

\subsection{Discrete entropy stability}\label{SEC::ENTROPY_DISCR}

For~$\star=p,q$, recall that~$\cvWstar$, $\cvZstar$, $\cvSigmastar$, and~$\cvRstar$ denote the coefficient vectors expressing~$\whstar$, $\zhstar$, $\sigmahstar$, $\rhstar$, respectively, in terms of fixed, given bases. 
Similar to~\cite [Thm.~3.1]{Gomez-Jungel-Perugia:2024}, we prove a discrete entropy stability estimate and existence of solutions to the discrete problem~\eqref{EQ::MATRICIAL_COMPACT}. Before doing so, we establish the following discrete Gr\"onwall-type inequality.

\begin{lemma}[A discrete Gr\"onwall inequality]
\label{lemma:gronwall}
Assume that the time mesh~$\Tt$ is quasi-uniform, i.e., there exists a constant~$C_{\mathrm{qu}} > 0 $ such that~$\tau \le C_{\mathrm{qu}} \tau_{\min}$, and that its mesh size satisfies~$\tau < 1/(2C_f)$ for some positive constant~$C_f$. 
Let~$\{\alpha_n\}_{n = 0}^{N_t} \subset \R^+$ and~$\beta\in \R^+$ be such that
\begin{equation}
\label{eq:recursive-ineq}
(1 - C_f \tau_{n+1}) \alpha_{n+1} \le \alpha_n + \beta C_f \tau_{n+1}  \qquad \text{ for } n = 0, \ldots, N_t - 1.
\end{equation}
Then, the following bound holds for~$n = 0, \ldots, N_t - 1$:
\begin{equation}
\label{eq:gronwall}
\alpha_{n+1} \le \exp\Big(\frac{3}{2} C_f C_{\mathrm{qu}} \tnpo\Big) \big(\alpha_0 + 2\beta C_f \tnpo\big).
\end{equation}
\end{lemma}
\begin{proof}
Let~$n \in \{0, \ldots, N_t - 1\}$. Since~$1 - C_f \tau_{n + 1}  \geq 1 - C_f \tau > 1/2$, we have
\begin{equation*}
\alpha_{n + 1} 
\le \frac{1}{1-C_f\tau}\alpha_n+\frac{1}{1-C_f\tau}\beta C_f\tau_{n+1}
\le \frac{1}{1 - C_f \tau} \alpha_n +  2 \beta C_f \tau_{n+1}.
\end{equation*}
This, together with~$1/(1 - C_f \tau)^{\ell} < 1/(1 - C_f \tau)^{n + 1}$ for all~$\ell\le n$, implies
\begin{equation}
\label{eq:aux-gronwall}
\alpha_{n + 1} \le \Big(\frac{1}{1 - C_f \tau} \Big)^{n + 1} \alpha_0 + 2\beta C_f \sum_{k = 0}^{n} \Big(\frac{1}{1 - C_f \tau} \Big)^k \tau_{n + 1 - k} \le \Big( \frac{1}{1 - C_f \tau}\Big)^{n + 1} \big( \alpha_0 + 2 \beta C_f \tnpo).
\end{equation}
Moreover, the following inequality can be easily proven using Taylor expansions:
\begin{equation*}
\label{eq:taylor-inequality}
0 < - \log( 1 - \delta) \le \delta + \delta^2 \quad \text{ for all } 0 < \delta < \frac{1}{2},
\end{equation*}
which implies
\begin{equation}
\label{eq:delta-formula}
1 < \frac{1}{1-\delta} \le \exp\big( \delta + \delta^2\big) \quad \text{ for all } 0 < \delta < \frac12. 
\end{equation}
Then, combining~\eqref{eq:delta-formula} for~$\delta = C_f \tau$ with the inequalities~$(n + 1) \tau \le C_{\mathrm{qu}} \tnpo$ and~$C_f \tau < 1/2$, 
we get
\begin{equation*}
\Big(\frac{1}{1 - C_f \tau}\Big)^{n+1} \le \exp\Big( C_f (n + 1) \tau\big( 1  + C_f \tau) \Big) \le \exp \Big(\frac{3}{2} C_f C_{\mathrm{qu}} \tnpo \Big),
\end{equation*}
which, together with~\eqref{eq:aux-gronwall}, leads to the desired result. 
\end{proof}

The quasi-uniformity assumption on~$\Tt$ 
\begin{equation} \label{eq:qu_t}
\text{there exists a constant~$C_{\mathrm{qu}} > 0 $ such that~$\tau \le C_{\mathrm{qu}} \tau_{n}$, \quad  $n=1,\ldots,N_t$}.
\end{equation}
is used in the last step of the proof of Lemma~\ref{lemma:gronwall}, where it simplifies the argument and the expression of the final estimate. 
While not essential, such an assumption simplifies the notation and presentation, so we adopt it from here on.

Recalling from Section~\ref{SEC::METHOD} that~$\Norm{\cdot}{\DG}$ denotes a norm in~$\Wh$, in which the bilinear form~$(\cdot, \cdot)_{\LDG}$ is coercive (see Section~\ref{SEC::h-convergence} below for details), we state the following theorem.

\begin{theorem}[Discrete entropy stability]\label{TH::DISCRETE_ENTROPY}
Assume that the time mesh~$\Tt$ satisfies the quasi-uniformity assumption~\eqref{eq:qu_t}, and that its mesh size satisfies~$\tau < 1/(2C_f)$,
where~$C_f$ is the constant in Proposition~\ref{PROP::HETER_FBOUND}. 
Then, any solution
$\big\{(\whp^{\n}, \whq^{\n}, \sigmahp^{\n}, \sigmahq^{\n})\big\}_{n = 1}^{N_t}$ to~\eqref{EQ::VARIATIONAL} (recall that~$\zhstar^{\n}$ and~$\rhstar^{\n}$ are defined in terms of~$\whstar^{\n}$ and~$\sigmahstar^{\n}$, respectively) satisfies
\begin{align}
\nonumber
\varepsilon\tau_{n+1} \sum_{\star=p,q}\Norm{\whstar^{\npo}}{\DG}^2 + \sum_{\star=p,q} \int_\Omega s_\star(u_\star(\whstar^{\npo}))\dx 
+ 4\,\equilp^{-1}\dext\tau_{n+1}\Norm{\sigmahp^{\npo}}{L^2(\Omega)^d}^2 \\
\label{eq:entropy1_nosigmaq}
+ \equilq^{-1}\dext \tau_{n+1}\int_\Omega \frac{|\sigmahq^{\npo}|^2}{e^{\whq^{\npo}}}\dx
+\tau_{n+1}\sum_{\star=p,q}\Norm{\eta_F^{\frac12} \mathsf{h}^{-1/2} \jump{\whstar^{\npo}}_{\sf N}}{L^2(\Fho)^d}^2\\
\nonumber
\le 
\sum_{\star=p,q}\int_\Omega s_\star(\ustar(\whstar^{\n}))\dx +
C_f\tau_{n+1}\left(\equilq|\Omega|+C^\ast\right),
\end{align}
and
\begin{alignat}{2}
\nonumber
\sum_{n=0}^{N_t-1}\tau_{n+1} & 
\Bigg(\varepsilon  \sum_{\star=p,q} \Norm{\whstar^{\npo}}{\DG}^2  + 4\,\equilp^{-1}\dext \Norm{\sigmahp^{\npo}}{L^2(\Omega)^d}^2  \\
\nonumber
& + \equilq^{-1}\dext \int_\Omega \frac{|\sigmahq^{\npo}|^2}{e^{\whq^{\npo}}}\dx
+\sum_{\star=p,q}\Norm{\eta_F^{\frac12}\mathsf{h}^{-\frac12} \jump{\whstar^{\npo}}_{\sf N}}{L^2(\Fho)^d}^2\Bigg)
 \\
\label{eq:entropy2_nosigmaq}
& 
+ \sum_{\star = p, q} \int_{\Omega} s_{\star}(u_{\star}(\whstar^{N_t})) \dx  \le 
\int_{\Omega} s_p(p_0) \dx + \int_{\Omega} s_q(q_0) \dx + C_f T\left(\equilq|\Omega|+C^\ast\right),
\end{alignat}
where %
\begin{equation}
\label{def:C-ast}
C^\ast\coloneqq \exp \Big(\frac{3}{2} C_f C_{\mathrm{qu}} T \Big) \Big( \int_{\Omega} s_q(q_0) \dx + 2 \equilq |\Omega| C_f T\Big).
\end{equation}
\end{theorem}

\begin{proof}
Multiplying~\eqref{EQ::MATRICIAL_COMPACT_W} by~$\cvWstar^{\npo}$, we obtain
\begin{equation}\label{EQ::START}
\begin{split}
\tau_{n+1}\sum_{\star=p,q}\Big[\varepsilon \langle A_{\LDG} \cvWstar^{\npo}, \cvWstar^{\npo}\rangle & + \frac{1}{\tau_{n + 1}}\langle \mathcal{U}_{\star}(\cvWstar^{\npo})-\mathcal{U}_{\star, h}^{(n)},\cvWstar^{\npo}\rangle \\
- \langle B^T M_I^{-1} M_D \cvSigmastar^{\npo},\cvWstar^{(n + 1)}\rangle 
& +  \langle J \cvWstar^{(n + 1)},\cvWstar^{\npo}\rangle\Big] \\
& = \tau_{n+1}\sum_{\star=p,q}\langle \mathcal{F}_{\star} \big(\cvWp^{\npo}, \cvWq^{\npo}),\cvWstar^{\npo}\rangle,
\end{split}
\end{equation}
\paragraph{Step 1.} We start by establishing bounds for all terms in~\eqref{EQ::START}.
For the first two terms in the sum on the left-hand side, using that~$\ustar=(s_\star')^{-1}$, \text{for~$n = 1, \ldots, N_t-1$,} we obtain 
\begin{alignat}{3}
\nonumber
\varepsilon \tau_{n+1}\langle A_{\LDG} \cvWstar^{\npo}, \cvWstar^{\npo}\rangle + & \langle \mathcal{U}_{\star}(\cvWstar^{\npo})-\mathcal{U}_{\star, h}^{(n)},\cvWstar^{\npo}\rangle\\
\nonumber
&\ge \varepsilon\tau_{n+1}\Norm{\whstar^{\npo}}{\DG}^2+\int_\Omega\left(\ustar(\whstar^{\npo})  - \ustarh^{\n}\right)\whstar^{\npo}\dx\\
\nonumber
& = \varepsilon\tau_{n+1}\Norm{\whstar^{\npo}}{\DG}^2+\int_\Omega\left(\ustar(\whstar^{\npo})  - \ustar(\whstar^{\n})\right)\whstar^{\npo}\dx \\
\label{EQ::FIRSTSECOND}
&\ge \varepsilon\tau_{n+1}\Norm{\whstar^{\npo}}{\DG}^2+
\int_\Omega\left(s_\star(\ustar(\whstar^{\npo})) -s_\star(\ustar(\whstar^{\n})) \right)\dx,
\end{alignat}
where, in the last step, we used the convexity of~$s_\star$. For~$n = 0$, the term~$\ustar(\whstar^{\n})$ in~\eqref{EQ::FIRSTSECOND} is replaced by~$p_0$ (resp. $q_0$) for~$\star = p$ (resp. $\star = q$). In what follows, we focus on the case~$n > 0$, while the case~$n = 0$ can be treated similarly.

\noindent
For the third term, we have
\begin{align*}
- \tau_{n+1}\langle B^T M_I^{-1} M_D \cvSigmastar^{\npo},\cvWstar^{(n + 1)}\rangle &= -\tau_{n+1} \langle M_I^{-1} M_D \cvSigmastar^{\npo},B\cvWstar^{(n + 1)}\rangle\\
&=-\tau_{n+1}\int_\Omega \D\sigmahstar^{\npo}\cdot\nablaLDG \whstar^{\npo}\dx\\
&=\tau_{n+1}\int_\Omega \D\sigmahstar^{\npo}\cdot\zhstar^{\npo}\dx =\tau_{n+1}\int_\Omega \sigmahstar^{\npo}\cdot\D\zhstar^{\npo}\dx\\
&=\tau_{n+1}\int_\Omega \sigmahstar^{\npo}\cdot\D s_\star''(\ustar(\whstar^{\npo}))\sigmahstar^{\npo}\dx.
\end{align*}
Then, for~$\star=p$, using~\eqref{EQ::HETER_UPD1}, we obtain
\begin{equation}\label{EQ::THIRDp}
- \tau_{n+1}\langle B^T M_I^{-1} M_D \cvSigmap^{\npo},\cvWp^{(n + 1)}\rangle\ge 4\,\equilp^{-1}\dext\tau_{n+1}\Norm{\sigmahp^{\npo}}{L^2(\Omega)^d}^2.
\end{equation}
For~$\star=q$, %
we derive
\begin{equation}\label{EQ::THIRDq}
\begin{split}
- \tau_{n+1}\langle B^T M_I^{-1} M_D \cvSigmaq^{\npo},\cvWq^{(n + 1)}\rangle&\ge \dext \tau_{n+1}\int_\Omega \frac{|\sigmahq^{\npo}|^2}{u_q(\whq^{\npo})}\dx 
= \equilq^{-1}\dext \tau_{n+1}\int_\Omega \frac{|\sigmahq^{\npo}|^2}{e^{\whq^{\npo}}}\dx.
\end{split}
\end{equation}
For the fourth term, by definition, we have the identity
\begin{equation}\label{EQ::FOURTH}
\tau_{n+1}\langle J \cvWstar^{(n + 1)},\cvWstar^{\npo}\rangle = \tau_{n+1}\Norm{\eta_F^{\frac12}\mathsf{h}^{-\frac12} \jump{\whstar^{\npo}}_{\sf N}}{L^2(\Fho)^d}^2.
\end{equation}
Finally, for the term on the right-hand side, we have
\begin{equation}\label{EQ::RHS}
\begin{split}
\tau_{n+1}\sum_{\star=p,q} & \langle \mathcal{F}_{\star} \big(\cvWp^{\npo}, \cvWq^{\npo}),  \cvWstar^{\npo}\rangle \\
&= \tau_{n+1}\sum_{\star=p,q}\int_\Omega f_\star\left(u_p(\whp^{\npo}),u_q(\whq^{\npo})\right)\whstar^{\npo}\dx\\
&=\tau_{n+1}\sum_{\star=p,q}\int_\Omega f_\star\left(u_p(\whp^{\npo}),u_q(\whq^{\npo})\right)s_\star'(\ustar(\whstar^{\npo}))\dx\\
&\le C_f\equilq\tau_{n+1}|\Omega|  + C_f\tau_{n+1}\int_\Omega s_q(u_q(\whq^{\npo}))dx,
\end{split}
\end{equation}
where, in the last step, we used Proposition~\ref{PROP::HETER_FBOUND}.

\paragraph{Step 2.} We proceed by proving the following bound: for~$n=1,\ldots,N_t$,
\begin{equation}\label{eq:bound_s}
\begin{split}
\int_{\Omega} s_q(u_q(\whq^{(n)})) \dx &\le \exp \Big(\frac{3}{2} C_f C_{\mathrm{qu}} T \Big) \Big( \int_{\Omega} s_q(q_0) \dx + 2 \equilq |\Omega| C_f T\Big) =: C^\ast.
\end{split}
\end{equation}
We emphasize that the constant~$C^\ast$ %
is independent of~$h$, $\tau$, and~$n$.
Using~\eqref{EQ::FIRSTSECOND} with~$\star=q$, \eqref{EQ::RHS}, and the fact that the right-hand sides of~\eqref{EQ::THIRDq} and~\eqref{EQ::FOURTH} are nonnegative,
from~\eqref{EQ::MATRICIAL_COMPACT_W} with~$\star=q$ multiplied by~$\cvWq^{\npo}$, we get
\begin{equation*}\label{eq:bound_w}
\begin{split}
\varepsilon\tau_{n+1}\Norm{\whq^{\npo}}{\DG}^2+
\left(1-C_f\tau_{n+1}\right)\int_\Omega s_q(u_q(\whq^{\npo}))\dx
\le
\int_\Omega s_q(\uq(\whq^{\n}))\dx+C_f\equilq \tau_{n+1} |\Omega|,
\end{split}
\end{equation*}
which implies
\[
\left(1-C_f\tau_{n+1}\right)\int_\Omega s_q(u_q(\whq^{\npo}))\dx
\le
\int_\Omega s_q(\uq(\whq^{\n})) \dx+C_f\equilq \tau_{n+1} |\Omega|.
\]
We apply Lemma~\ref{lemma:gronwall} with~$\beta = \equilq|\Omega|$ and
\begin{equation*}
\alpha_n \coloneqq  \begin{cases}
\displaystyle \int_{\Omega} s_q(q_0) \dx & \text{ if } n = 0, \\\\
\displaystyle  \int_{\Omega} s_q(u_q(\whq^{(n)})) \dx & \text{ if } n > 0
\end{cases}
\end{equation*}
to obtain, for~$n=1,\ldots,N_t$,
\begin{equation*}
\begin{split}
\int_{\Omega} s_q(u_q(\whq^{(n)})) \dx & \le \exp \Big(\frac{3}{2} C_f C_{\mathrm{qu}} \tn \Big) \Big( \int_{\Omega} s_q(q_0) \dx + 2 \equilq |\Omega| C_f \tn\Big),
\end{split}
\end{equation*}
which directly yields~\eqref{eq:bound_s}.
\medskip

\paragraph{Step 3.} %
Inserting~\eqref{EQ::FIRSTSECOND}--\eqref{EQ::RHS}
into~\eqref{EQ::START} gives~\eqref{eq:entropy1_nosigmaq}.
This, together with~\eqref{eq:bound_s}, implies~\eqref{eq:entropy2_nosigmaq}.

\end{proof}

\begin{remark}[Control of~$\Norm{\sigmahq^{\npo}}{L^2(\Omega)^d}$] \label{rem:controlsigma}
Since, for our choice of the entropy densities, $s_q''$ is not bounded away from zero, the estimates of Theorem~\ref{TH::DISCRETE_ENTROPY} provide no control of~$\Norm{\sigmahq^{\npo}}{L^2(\Omega)^d}$. 
This is of no consequence for the proof of the existence of discrete solutions in Section~\ref{sec:EXISTENCE-DISCRETE} below. 
However, for the discrete compactness argument used in the convergence analysis, we require an $h$-uniform bound on~$\Norm{\sigmahq^{\npo}}{L^2(\Omega)^d}$. This bound can be obtained from the estimates of Theorem~\ref{TH::DISCRETE_ENTROPY} under suitable assumptions on the norm~$\Norm{\cdot}{\DG}$; see Section~\ref{sec:convergence} below.
\eremk
\end{remark}

\subsection{Existence of discrete solutions \label{sec:EXISTENCE-DISCRETE}}
In the following theorem, we establish the existence of discrete solutions. The proof follows the lines of \cite[Thm.~3.2]{Gomez-Jungel-Perugia:2024}, adapted to our setting. The main modifications concern the terms related to the~$q$ variable.

\begin{theorem}[Existence of discrete solutions]\label{thm:existence_discrete}
For each time step~$n = 0, \ldots, N_t - 1$, problem~\eqref{EQ::MATRICIAL_COMPACT} admits a solution $(\cvWp^{\npo}, \cvWq^{\npo},\cvSigmap^{\npo}, \cvSigmaq^{\npo})$. 
\end{theorem}

\begin{proof}
We apply the Leray-Schauder theorem.
Set~$N_h\coloneqq \dim(\Wh)$ and begin with the case~$n=0$. 

\paragraph{Case~$n=0$.}
Define the operator $\Phi: \R^{2 N_h}\to \R^{2 N_h}$ that maps $(\cvVp,\cvVq)\in \R^{2 N_h}$ to the unique solution $(\cvWp,\cvWq)\in \R^{2 N_h}$ to the following problem: for~$\star=p,q$,
\begin{equation}\label{eq:linearized}
\varepsilon A_{\LDG} \cvWstar = - \frac{1}{\tau_{1}}\mathcal{U}_{\star}(\cvVstar) +\frac{1}{\tau_{1}}\mathcal{U}_{\star, h}^{(0)} + B^T M_I^{-1} M_D \cvSigmastar -  J \cvVstar + \mathcal{F}_{\star} \big(\cvVp, \cvVq),
\end{equation}
where~$\cvSigmastar = \cvSigmastar(\cvVstar)$ is the unique solution to
\begin{equation}\label{eq:linearized2}
\mathcal{N}_{\star}\big(\cvVstar\big) \cvSigmastar  = -M_D M_I^{-1} B \cvVstar,
\end{equation}
see Remarks~\ref{RMK::UNIQUESIGMA} and~\ref{RMK:unique_sigma},
and $\mathcal{U}_{\star, h}^{(0)}$ is the coefficient vector of~$\PiW p_0$ or~$\PiW q_0$ accordingly. Denoting by $\whstar\in \Wh$, $\sigmahstar\in \Rh$, and~$\vhstar\in \Wh$ the functions whose coefficient vectors are, respectively, $\cvWstar$, $\cvSigmastar$, and~$\cvVstar$, the map~$\Phi$ can be defined equivalently as the operator
\[
\Phi: (\Wh)^2 \to (\Wh)^2, \qquad (\vhp,\vhq)\mapsto (\whp,\whq).
\]

\paragraph{Well-definedness of~$\Phi$.} The well-posedness of the linear system~\eqref{eq:linearized} follows from the positive-definiteness of~$A_{\LDG}$, while that of~\eqref{eq:linearized2} is discussed in Remark~\ref{RMK:unique_sigma}. 

\paragraph{Continuity of~$\Phi$.} As for the continuity of~$\Phi$, we aim to prove that, for any sequence~$\{(\vhp^r, \vhq^r)\}_{r \in \N} \subset (\Wh)^2$ converging to some~$(\vhp^\diamond, \vhq^\diamond) \in (\Wh)^2$, the sequence~$\{\Phi(\vhp^r, \vhq^r)\}_{r \in \N}$ converges to~$\Phi(\vhp^\diamond, \vhq^\diamond)$. 
Since the space~$\Wh$ is finite dimensional, the sequence~$(\vhp^r, \vhq^r)$ converges pointwise to~$(\vhp^\diamond, \vhq^\diamond)$, and it is bounded uniformly in~$\Omega$ (recall that~$h$, $\tau$, and~$\varepsilon$ are fixed).
We show that the relation~\eqref{eq:linearized2}
leads to
\begin{equation}\label{eq:conv_nonlin}
\cvSigmastar(\cvVstar^r) \to \cvSigmastar(\cvVstar^\diamond).
\end{equation}

Since the sequence~$(\vhp^r, \vhq^r)$ is bounded uniformly in~$\Omega$, and both~$\up(\cdot)$ and~$\uq(\cdot)$ are continuous in~$\R^d$, then the sequences~$(\up(\vhp^r))$ and~$(\uq(\vhq^r))$ take values in compact sets~$\mathcal{K}_p \subset (0, \equilp)$ and~$\mathcal{K}_q \subset (0, \infty)$, respectively. Moreover, $\sstar''(\cdot)$ is continuous on~$\mathcal{K}_{\star}$, which implies the uniform boundedness in~$\Omega$ of the sequence~$\sstar''(\ustar(\vhstar^r))$ for~$\star = p, q$, i.e., there exist positive constants $C_{s, \star}^{\min}$ and~$C_{s, \star}^{\max}$ independent of~$r$ such that
\begin{equation}
\label{eq:bounds-spp}
0 < C_{s,\star}^{\min} \le \sstar''(\ustar(\vhstar^r)) \le C_{s,\star}^{\max}.
\end{equation}
We emphasize that, for~$s''_p$, we actually have $s''_p\ge 4\equilp^{-1}$ (see the text above~\eqref{EQ::HETER_UPD1}), so that~$C_{s,p}^{\min}=4\equilp^{-1}$.

Additionally, since~$\sstar''(\ustar(\cdot))$ is locally Lipschitz for both~$\star=p,q$\,\footnote{For~$\star=p$, $s_p''(u_p(\cdot))=(1+\exp(\cdot))^2/\equilp\exp(\cdot)$ and, for~$\star=q$, $s_q''(u_q(\cdot))=1/\equilq \exp(\cdot)$.}, and both~$\vhstar^r$ and~$\vhstar^\diamond$ are bounded in~$L^\infty(\Omega)$, we have
\begin{equation} \label{eq:convergence-spp}
\lim_{r\to\infty}\Norm{\sstar''(\ustar(\vhstar^r))-\sstar''(\ustar(\vhstar^\diamond))}{L^\infty(\Omega)} \le \lim_{r\to\infty}C\Norm{\vhstar^r-\vhstar^\diamond}{L^{\infty}(\Omega)}=0,\qquad \star=p,q.
\end{equation}
From the definition of~$\sigmahstar(\vhstar^r)$ and~$\sigmahstar(\vhstar^{\diamond})$ in~\eqref{eq:conv_nonlin}, we deduce the following identity:
\begin{equation}
\label{eq:sigmas-identity}
\begin{split}
\big(\D \sstar''(\ustar(\vhstar^r)) (\sigmahstar(\vhstar^r) - \sigmahstar(\vhstar^{\diamond})), \phih \big)_{\Omega} 
& = \Big(\D \big(\sstar''(\ustar(\vhstar^r)) - \sstar''(\ustar(\vhstar^{\diamond})) \big) \sigmahstar(\vhstar^{\diamond}), \phih\Big)_{\Omega} \\
& \quad - \Big(\D \nablaLDG \big(\vhstar^r - \vhstar^{\diamond} \big), \phih \Big)_{\Omega}
\end{split}
\end{equation}
for all~$\phih \in \Rh$. Taking~$\phih = \sigmahstar(\vhstar^r) - \sigmahstar(\vhstar^{\diamond})$  in~\eqref{eq:sigmas-identity} and using the properties of~$\D$ in~\eqref{eq:D}, the bound in~\eqref{eq:bounds-spp}, and the Cauchy--Schwarz inequality, we obtain
\begin{equation*}
\begin{split}
\dext & C_{s, \star}^{\min}  \Norm{\sigmahstar(\vhstar^r) - \sigmahstar(\vhstar^{\diamond})}{L^2(\Omega)^d}^2 \\
& \le D_{\max} \Big(\Norm{\sstar''(\ustar(\vhstar^r)) - \sstar''(\ustar(\vhstar^{\diamond}))}{L^{\infty}(\Omega)^d}\Norm{\sigmahstar(\vhstar^{\diamond})}{L^2(\Omega)^d} + \Norm{\nablaLDG (\vhstar^r - \vhstar^{\diamond})}{L^2(\Omega)^d} \Big) \\
& \quad \times \Norm{\sigmahstar(\vhstar^r) - \sigmahstar(\vhstar^{\diamond})}{L^2(\Omega)^d}.
\end{split}
\end{equation*}
This bound, together with~\eqref{eq:convergence-spp} and the convergence of~$\Norm{\nablaLDG (\vhstar^r- \vhstar^{\diamond})}{L^2(\Omega)^d}$ due to the norm equivalence in finite dimensions, implies
\begin{equation*}
\lim_{r \to \infty}    \Norm{\sigmahstar(\vhstar^r) - \sigmahstar(\vhstar^{\diamond})}{L^2(\Omega)^d} =  0,
\end{equation*}
which completes the proof of~\eqref{eq:conv_nonlin}.

The continuity property in~\eqref{eq:conv_nonlin}, together with the continuity of~$\ustar(\cdot)$ and~$\fstar(\cdot, \cdot)$, and the linearity and boundedness of the remaining terms, implies the continuity of the right-hand side of~\eqref{eq:linearized} with respect to~$\cvVstar$. Passing to the limit in~\eqref{eq:linearized} for the sequence~$\{(\cvVp^r, \cvVq^r)\}_{r \in \N}$, the following equations are obtained: for~$\star = p, q$,
\begin{equation*}
\varepsilon A_{\LDG} \cvWstar^{\diamond} = - \frac{1}{\tau_{1}}\mathcal{U}_{\star}(\cvVstar^{\diamond}) +\frac{1}{\tau_{1}}\mathcal{U}_{\star, h}^{(0)} + B^T M_I^{-1} M_D \cvSigmastar(\cvVstar^{\diamond}) -  J \cvVstar^{\diamond}  + \mathcal{F}_{\star} \big(\cvVp^{\diamond}, \cvVq^{\diamond}).
\end{equation*}
The uniqueness of the solution to this linear system, guaranteed by the positive-definiteness of~$A_{\LDG}$, implies that~$(\whp^{\diamond}, \whq^{\diamond}) = \Phi(\vhp^{\diamond}, \vhq^{\diamond})$, which completes the proof of the continuity of the map~$\Phi$.

\paragraph{Compactness of~$\Phi$.} The compactness is automatic in finite dimensions. 

\paragraph{Boundedness of scaled fixed points.} It only remains to prove that 
\begin{equation}\label{eq:LS}
\text{all solutions to $(\cvWp^\lambda,\cvWq^\lambda)=\lambda \Phi((\cvWp^\lambda,\cvWq^\lambda))$ with~$\lambda\in [0,1]$ satisfy~$\Norm{(\cvWp^\lambda,\cvWq^\lambda)}{\R^{2 N_h}}\le R$},
\end{equation}
with a constant~$R>0$ independent of~$\lambda$. Then, the Leray-Schauder theorem implies that~$\Phi$ has a fixed point, which gives the existence of a solution to~\eqref{EQ::MATRICIAL_COMPACT} for~$n=0$.  

Let~$(\cvWp^\lambda,\cvWq^\lambda)\ne (\mathbf{0},\mathbf{0})$ be such that~$(\cvWp^\lambda,\cvWq^\lambda)=\lambda \Phi((\cvWp^\lambda,\cvWq^\lambda))$ with~$\lambda\in [0,1]$. Then, necessarily, $\lambda\in (0,1]$ and
\[
\frac{\varepsilon}{\lambda} A_{\LDG} \cvWstar^\lambda + \frac{1}{\tau_{1}}\mathcal{U}_{\star}(\cvWstar^\lambda) -\frac{1}{\tau_{1}}\mathcal{U}_{\star, h}^{(0)} - B^T M_I^{-1} M_D \cvSigmastar^\lambda +  J \cvWstar^\lambda  = \mathcal{F}_{\star} \big(\cvWp^\lambda, \cvWq^\lambda).
\]
Multiplying by~$\cvWstar^\lambda$ and proceeding as in the proof of Theorem~\ref{TH::DISCRETE_ENTROPY}, we obtain, in terms of the finite element functions, 
\begin{align}
\frac{\varepsilon}{\lambda}\tau_{1} \sum_{\star=p,q}\Norm{\whstar^\lambda}{\DG}^2 + \sum_{\star=p,q} \int_\Omega s_\star(u_\star(\whstar^\lambda))\dx 
\label{eq:LSauxiliary}
\le 
\int_\Omega s_p(p_0)\dx +\int_\Omega s_q(q_0)\dx +
C_f\tau_{1}\left(\equilq|\Omega|+C^\ast\right),
\end{align}
with~$C^\ast$ as in Theorem~\ref{TH::DISCRETE_ENTROPY}. Using~\eqref{eq:initial_entropy}, we have that, for~$\star=p,q$, $\Norm{\whstar^\lambda}{\DG}$ is uniformly bounded with respect to~$\lambda$, which implies~\eqref{eq:LS}. This concludes the proof of the existence of solutions to the system at the first time step. 
Additionally, from~\eqref{eq:LSauxiliary}, for any of such solutions, we also have that
\[
\sum_{\star=p,q} \int_\Omega s_\star(u_\star(\whstar^{(1)}))\dx\le \int_\Omega s_p(p_0)\dx +\int_\Omega s_q(q_0)\dx +
C_f\tau_{1}\left(\equilq|\Omega|+C^\ast\right).
\]

\paragraph{Case~$n\ge 1$.} For~$n\ge 1$, we proceed by induction and assume existence of solutions at time step~$n-1$ as well as, for any of such solutions, the boundedness of~$\sum_{\star=p,q} \int_\Omega s_\star(u_\star(\whstar^{(n)}))\dx$. Proceeding as above for the linearized problem
\[
\varepsilon A_{\LDG} \cvWstar = - \frac{1}{\tau_{n + 1}}\mathcal{U}_{\star}(\cvVstar) +\frac{1}{\tau_{n + 1}}\mathcal{U}_{\star, h}^{(n)} + B^T M_I^{-1} M_D \cvSigmastar -  J \cvWstar  + \mathcal{F}_{\star} \big(\cvVp, \cvVq),
\]
where $\mathcal{U}_{\star, h}^{(n)} = \mathcal{U}_{\star}\big(\cvWstar^{(n)}\big)$,
we obtain 
\begin{align}
\frac{\varepsilon}{\lambda}\tau_{n+1} \sum_{\star=p,q}\Norm{\whstar^\lambda}{\DG}^2 + \sum_{\star=p,q} \int_\Omega s_\star(u_\star(\whstar^\lambda))\dx 
\label{eq:LSauxiliary2}
\le 
\sum_{\star=p,q}\int_\Omega s_\star(u_\star(\whstar^{(n)}))\dx +
C_f\tau_{n+1}\left(\equilq|\Omega|+C^\ast\right).
\end{align}
The boundedness of~$\sum_{\star=p,q} \int_\Omega s_\star(u_\star(\whstar^{(n)}))\dx$ implies that, for~$\star=p,q$, $\Norm{\whstar^\lambda}{\DG}$ is uniformly bounded with respect to~$\lambda$, and thus existence of solutions to~\eqref{EQ::MATRICIAL_COMPACT} at the time step~$n$. Additionally, from~\eqref{eq:LSauxiliary2}, for any of such solutions, we also have that
\[
\sum_{\star=p,q} \int_\Omega s_\star(u_\star(\whstar^{(n+1)}))\dx\le \sum_{\star=p,q}\int_\Omega s_\star(u_\star(\whstar^{(n)}))\dx +
C_f\tau_{1}\left(\equilq|\Omega|+C^\ast\right),
\]
which completes the proof.
\end{proof}

\section{Convergence of the scheme}
\label{sec:convergence}
We now study the convergence of the scheme to a weak solution to the system~\eqref{EQN::HETERODIMER} that satisfies the physical bounds in~\eqref{EQ:HETER_BOUNDSCQ}. To do so, after specifying the choice of the bilinear form~$(\cdot, \cdot)_{\LDG}$ and the norm~$\Norm{\cdot}{\DG}$ (Section~\ref{sec:new}), we proceed in two steps: first, we prove that, as~$h \to 0$, the scheme converges to a semidiscrete-in-time formulation (Section~\ref{SEC::h-convergence}); then, we analyze the convergence of this formulation as~$(\varepsilon, \tau) \to (0, 0)$ (Section~\ref{SEC::tau-varepsilon}). 

\subsection{Choice of the bilinear form~\texorpdfstring{$(\cdot, \cdot)_{\LDG}$}{(.,.)_{LDG}} and the norm~\texorpdfstring{$\Norm{\cdot}{\DG}$}{||.||_{DG}}}\label{sec:new}
Henceforth, we assume that the degree of approximation in space satisfies~$\ell \geq 2$, and set the weight parameter~$\gamma_F = 1/2$ for all facets~$F \in \Fho$.
We choose the bilinear form~$(\cdot, \cdot)_{\LDG}$ as a discrete version of the $H^2(\Omega)$-regularization term used in the boundedness-by-entropy framework~\cite[\S3]{Jungel:2015} for two- and three-dimensional domains, and the norm~$\Norm{\cdot}{\DG}$ as a discrete version of the~$H^2(\Omega)$ norm. 
The specific choice we make admits the following discrete Sobolev inequality (see Lemma~\ref{lemma:LDG-product} below): there exists a positive constant~$\Csob$ independent of the mesh size~$h$ such that
\begin{equation}\label{EQ::LINFTY}
\Norm{\vh}{L^{\infty}(\Omega)} \le \Csob \Norm{\vh}{\DG}  \qquad \forall \vh \in \Wh.
\end{equation}
This inequality allows us to treat the term involving~$\sigmahq^{(n+1)}$ in the stability estimates of Theorem~\ref{TH::DISCRETE_ENTROPY} (see also Remark~\ref{rem:controlsigma}) as follows:
\begin{equation}
\label{eq:bound-sigmahq}
\equilq^{-1}\dext \tau_{n+1}\int_\Omega \frac{|\sigmahq^{\npo}|^2}{e^{\whq^{\npo}}}\dx \geq \equilq^{-1} \dext \tau_{n+1} \exp(-C/{\sqrt{\varepsilon \tau_{n+1}}}) \Norm{\sigmahq^{(n+1)}}{L^2(\Omega)^d}^2,
\end{equation}
for a positive constant~$C$ independent of~$h$, $\tau$, and~$\varepsilon$.
Indeed, estimate~\eqref{eq:entropy1_nosigmaq} implies that the term~${\varepsilon}\tau_{n+1}\Norm{\whq^{\npo}}{\DG}^2$ is uniformly bounded with respect to~$h$. Consequently, thanks to~\eqref{EQ::LINFTY}, the same holds for~$\varepsilon \tau_{n+1}\Norm{\whq^{\npo}}{L^\infty(\Omega)}^2$.
This uniform bound yields the lower bound~$\exp({-\Norm{\whq^{\npo}}{L^\infty(\Omega)}}) \ge \exp(-C/{\sqrt{\varepsilon\tau_{n+1}}})$, from which~\eqref{eq:bound-sigmahq} follows.

In order to define~$(\cdot, \cdot)_{\LDG}$ and~$\Norm{\cdot}{\DG}$, we first introduce the auxiliary piecewise polynomial space
\begin{equation*}
    \Zh \coloneqq  \prod_{K \in \Th} \Pp{\ell}{K}^{d \times d},
\end{equation*}
and the discrete LDG Hessian operator~$\mathcal{H}_{\LDG} : \Wh \to \Zh$ defined by
\begin{equation*}
\mathcal{H}_{\LDG} \lambda_h \coloneqq  \mathcal{H}_h \lambda_h - \mathcal{G}_h (\lambda_h) + \mathcal{B}_h (\lambda_h),
\end{equation*}
where~$\mathcal{H}_h$ is the piecewise Hessian operator defined element-by-element, and~$\mathcal{G}_h : \Wh \to \Zh$ and~$\mathcal{B}_h : \Wh \to \Zh$ are the lifting operators, which are given, respectively, by
\begin{subequations}
\begin{alignat}{3}
(\mathcal{G}_h \lambda_h, \Theta_h)_{\Omega} & = \sum_{K \in \Th} (\nabla \lambda_h, \mvl{\Theta_h}_{\!\frac12} \bnK)_{(\partial K)^{\circ}} & & \qquad \forall \Theta_h \in \Zh, \\
(\mathcal{B}_h \lambda_h, \Theta_h)_{\Omega} & = (\jump{\lambda_h}_{\sf N}, \mvl{\nabla_h \cdot \Theta_h}_{\!\frac12})_{\Fho} & & \qquad \forall \Theta_h \in \Zh.
\end{alignat}
\end{subequations}

Then, we define the bilinear form~$(\cdot, \cdot)_{\LDG}$ as
\begin{equation}
\label{def:LDG}
\begin{split}
    (\wh, \lambdah)_{\LDG} \coloneqq  & (\wh, \lambdah)_{\Omega} + (\nablaLDG \wh, \nablaLDG \lambdah)_{\Omega} + (\mathcal{H}_{\LDG} \wh,\mathcal{H}_{\LDG} \lambdah)_{\Omega} \\
    & + (\mathsf{h}^{-1} \jump{\nabla_h \wh}, \jump{\nabla_h \lambdah})_{\Fho} + (\mathsf{h}^{-3} \jump{\wh}_{\sf N}, \jump{\lambdah}_{\sf N})_{\Fho},
\end{split}
\end{equation}
where~$\jump{\cdot}$ denotes the standard (vector-valued) total jump on each facet~$F \in \Fho$.
Finally, we define the discrete norm:
\begin{equation}
\label{def:DG-norm}
    \Norm{\wh}{\DG}^2 \coloneqq  \Norm{\wh}{L^2(\Omega)}^2 + \Norm{\nabla_h \wh}{L^2(\Omega)^d}^2 + \Norm{\mathcal{H}_h \wh}{L^2(\Omega)^{d \times d}}^2 + \Norm{\mathsf{h}^{-\frac12} \jump{\nabla_h \wh}}{L^2(\Fho)^d}^2 + \Norm{h^{-\frac{3}{2}} \jump{\wh}_{\sf N}}{L^(\Fho)^d}^2.
\end{equation}
We emphasize that~$\Norm{\cdot}{\DG}$ is \emph{not} the norm  associated with~$(\cdot,\cdot)_{\LDG}$.

\subsection{Convergence as~\texorpdfstring{$h \to 0$}{h to 0} \label{SEC::h-convergence}}
In this section, we fix the penalty parameter~$\varepsilon > 0$ and the partition~$\Tt$ of the time interval~$(0, T)$ with~$\tau < 1/(2C_f)$. We consider a sequence of space meshes~$\{\Thm\}_m$ with decreasing  maximum mesh sizes~$\{h_m\}_m$ such that~$h_m \to 0$  as~$m \to \infty$. For the sake of clarity, we henceforth write the superscript~$\varepsilon$ to highlight the dependence of the discrete solution to~\eqref{EQ::VARIATIONAL} on the penalty parameter~$\varepsilon$.

Next lemma summarizes some properties of the bilinear form~$(\cdot, \cdot)_{\LDG}$ that are instrumental in our convergence analysis; see~\cite[\S4.2]{Gomez-Jungel-Perugia:2024}. 

\begin{lemma}[{Properties of~$(\cdot, \cdot)_{\LDG}$}]
\label{lemma:LDG-product}
Let the bilinear form~$(\cdot, \cdot)_{\LDG}$ be defined as in~\eqref{def:LDG}. Then, the following properties are satisfied:
\begin{enumerate}[label = \roman*), ref = \emph{\roman*)}]
    \item\label{P1} The bilinear form~$(\cdot, \cdot)_{\LDG}$ is coercive in~$\Whm$ with respect to the norm~$\Norm{\cdot}{\DG}$ defined in~\eqref{def:DG-norm}, i.e., there exists a positive constant~$\Ccoer$ independent of~$h_m$ such that
    \begin{equation*}
        (\psihm, \psihm)_{\LDG} \geq \Ccoer \Norm{\psihm}{\DG}^2 \qquad \forall \psihm \in \Whm.
    \end{equation*}

    \item\label{P2} If~$d = 2$, the following discrete Sobolev embedding holds: there exists a positive constant~$\Csob$ independent of~$h_m$ such that
    \begin{equation*}
        \Norm{\psihm}{L^{\infty}(\Omega)} \le \Csob \Norm{\psihm}{\DG} \qquad \forall \psihm \in \Whm.
    \end{equation*}

    \item\label{P3} For any sequence~$\{w_m\}_{m \in \N}$ with~$w_m \in \Whm$ that is uniformly bounded in the DG norm~\eqref{def:DG-norm}, there exists a subsequence, that we still denote by~$\{w_m\}_{m \in \N}$, and a function~$w \in H^2(\Omega)$ such that, as~$m \to \infty$, 
    \begin{alignat*}{3}
    \nablaLDG w_m & \rightharpoonup \nabla w & & \qquad \text{ weakly in~$L^2(\Omega)^d$}, \\
    w_m & \to w & & \qquad \text{ strongly in~$L^q(\Omega)$},
    \end{alignat*}
    with~$ 1 \le q < \infty$ (if~$d = 2$) or~$1 \le q < 6$  (if~$d = 3$). Moreover, for any~$\psi \in H^2(\Omega)$, there exists a sequence~$\{\psim\}_{m \in \N}$ with~$\psi \in \mathcal{W}^2(\Thm) \cap H^1(\Omega)$ such that, as~$m \to 0$, such a sequence converges strongly in~$H^1(\Omega)$ to~$\psi$ and
    \begin{equation*}
        (w_m, \psim)_{\LDG} \to (w, \psi)_{\Omega} + (\nabla w, \nabla \psi)_{\Omega} + (\mathcal{H} w, \mathcal{H} \psi)_{\Omega}.
    \end{equation*}
\end{enumerate}
\end{lemma}
Property~\ref{P2} has been proven in~\cite[\S4.2]{Gomez-Jungel-Perugia:2024} for~$d=2$. Although the argument there does not extend to~$d=3$, we believe the property remains valid in that case. However, in the absence of a proof, we proceed by making the following abstract assumption, which is obviously satisfied for~$d=2$, as detailed above.
\begin{assumption}
 There exists a bilinear form~$(\cdot,\cdot)_{\LDG}$ and a norm~$\Norm{\cdot}{\DG}$ such that~\ref{P1}, \ref{P2}, and~\ref{P3} in Lemma~\ref{lemma:LDG-product} are satisfied.
\end{assumption}

\begin{theorem}[$h$-convergence]
\label{thm:h-convergence}
\begin{subequations} \label{eq:strong-u}
Let the penalty parameter~$\varepsilon > 0$ and the time partition~$\Tt$ be fixed, and assume that~$\Tt$ satisfies the quasi-uniformity condition~\eqref{eq:qu_t} with~$\tau < 1/(2C_f)$. Let also the initial conditions~$p_0$ and~$q_0$ satisfy~\eqref{eq:ic}, and the degree of approximation in space~$\ell \geq 2$.
For~$n = 1, \ldots, N_t$, there exist~$w_p^{\varepsilon, \n}, w_q^{\varepsilon, \n} \in H^2(\Omega)$ and a subsequence of~$\{\Thm\}_{m \in \N}$ still denoted by~$\{\Thm\}_{m \in \N}$ 
such that, as~$m \to \infty$, 
\begin{alignat}{3}
\label{eq:strong-u_p}
 p_m^{\varepsilon, \n} \coloneqq  u_p(w_{p, m}^{\varepsilon, \n}) & \to p^{\varepsilon, \n} \coloneqq  u_p(w_p^{\varepsilon, \n}) & & \qquad \text{ strongly in~$L^r$ for all~$r \in [1, \infty)$}, \\
 \label{eq:strong-u_q}
 q_m^{\varepsilon, \n} \coloneqq  u_q(w_{q, m}^{\varepsilon, \n}) & \to q^{\varepsilon, \n} \coloneqq  u_q(w_q^{\varepsilon, \n}) & & \qquad \text{ strongly in~$L^r$ for all~$r \in [1, \infty)$}.
\end{alignat}
\end{subequations}
Moreover, for~$n = 0, \ldots, N_t - 1$, the pair~$(w_p^{\varepsilon, \npo}, w_q^{\varepsilon, \npo})$ solves the following semidiscrete-in-time variational formulation: for~$\star = p, q$, 
\begin{alignat}{3}
\nonumber
\varepsilon \Big[ \left(w_{\star}^{\varepsilon, \npo},  \psi\right)_{\Omega} & + \left(\nabla w_{\star}^{\varepsilon, \npo}, \nabla \psi\right)_{\Omega} + \left(\mathcal{H} w_{\star}^{\varepsilon, \npo},  \mathcal{H} \psi \right)_{\Omega} \Big] + \frac{1}{\tau_{n+1}}\left(u_{\star}(w_{\star}^{\varepsilon, \npo}) - (\star)^{\varepsilon, \n}, \psi \right)_{\Omega} \\
\label{eq:h-limit-formulation}
& + \left(\D \nabla u_{\star}(w_{\star}^{\varepsilon, \npo}), \nabla \psi\right)_{\Omega} = \left(f_{\star}(u_p(w_p^{\varepsilon, \npo}), u_q(w_q^{\varepsilon, \npo})), \psi \right)_{\Omega} \qquad \forall \psi \in H^2(\Omega),
\end{alignat}
where~$(\star)^{\varepsilon, (0)} \coloneqq  (\star)_0$, and~$\mathcal{H} : H^2(\Omega) \to L^2(\Omega)^{d\times d}$ denotes the global Hessian operator.

Finally, the following entropy stability estimate for the limit functions holds for~$n = 1, \ldots, N_t$:
    \begin{align}
\nonumber
\sum_{k=0}^{n-1}\tau_{k+1} \Bigg(\varepsilon  \sum_{\star=p,q} & \Norm{w_{\star}^{\varepsilon, \kpo}}{H^2(\Omega)}^2  + 4\,\equilp^{-1}\dext \Norm{\nabla p^{\varepsilon, \kpo}}{L^2(\Omega)^d}^2 + 4\equilq^{-1}\dext \Norm{\nabla \sqrt{q^{\varepsilon, \kpo}}}{L^2(\Omega)^d}^2 \Bigg)\\
\label{eq:entropy-h-limit}
& 
+ \int_{\Omega} s_p(p^{\varepsilon, n}) \dx + \int_{\Omega} s_q(q^{\varepsilon, n}) \dx \le 
\int_{\Omega} s_p(p_0) \dx + \int_{\Omega} s_q(q_0) \dx + C_f T\left(\equilq|\Omega|+C^\ast\right),
\end{align}
where~$C^*$ is the constant defined in~\eqref{def:C-ast}.
\end{theorem}
\begin{proof}
We prove the result using an induction argument on the time steps.
\paragraph{Approximation of the initial conditions.} Set~$(p_m^{\varepsilon, 0}, q_m^{\varepsilon, 0}) \coloneqq  (\PiW p_0, \PiW q_0) \in \Whm \times \Whm$. Using the orthogonality and approximation properties of~$\PiW$, for all~$\psi \in H^2(\Omega)$, we have
\begin{equation*}
(p_0 - p_m^{\varepsilon, 0}, \psi)_{\Omega} = (p_0, \psi - \PiW \psi)_{\Omega} \le C h_m^2 \Norm{p_0}{L^2(\Omega)} \Seminorm{\psi}{H^2(\Omega)},
\end{equation*}
for some positive constant~$C$ independent of~$h_m$ and~$\psi$. An analogous estimate can be obtained for~$q_0 - q_m^{\varepsilon, 0}$. Therefore, as~$m \to 0$, we get
\begin{alignat*}{3}
    (p_m^{\varepsilon, 0}, \psi) & \to (p_0, \psi) & & \qquad \forall \psi \in H^2(\Omega), \\
    (q_m^{\varepsilon, 0}, \psi) & \to (q_0, \psi) & & \qquad \forall \psi \in H^2(\Omega).
\end{alignat*}

\paragraph{Convergence of~$w_{\star, m}^{\varepsilon, \n}$.} The discrete entropy stability estimate~\eqref{eq:entropy2_nosigmaq} implies that, for~$n = 1, \ldots, N_t$ and~$\star = p, q$, the sequence~$\{w_{\star, m}^{\varepsilon, \n}\}_{m \in \N}$ is bounded in the DG
norm uniformly with respect to~$h_m$. Moreover, due to Lemma~\ref{lemma:LDG-product}, \ref{P2}, it is also bounded in the~$L^{\infty}(\Omega)$ norm uniformly with respect to~$h_m$. Then, we apply Lemma~\ref{lemma:LDG-product}, part~\ref{P3} and obtain the existence of some functions~$w_p^{\varepsilon,\n}$, $w_q^{\varepsilon, \n} \in H^2(\Omega)$ such that, up to subsequences that we still denote by~$\{w_{\star, m}^{\varepsilon, \n}\}_{m \in \N}$, it holds
\begin{subequations}
\begin{alignat}{3}
\label{eq:weak-nablaLDG-wm}
\nablaLDG w_{\star, m}^{\varepsilon, \n} & \rightharpoonup \nabla w_{\star}^{\varepsilon, \n} & & \qquad \text{ weakly in }L^2(\Omega)^d, \\
\label{eq:strong-wm}
w_{\star, m}^{\varepsilon, \n} & \to w_{\star}^{\varepsilon, \n} & & \qquad \text{ strongly in } L^q(\Omega),
\end{alignat}
\end{subequations}
with~$1 \le q < \infty$ (if~$d = 2$) or~$1 \le q < 6$ (if~$d = 3$). Moreover, by~\cite[Thm.~4.9 in \S4.2]{Brezis:2011}, we can extract from~$\{w_{\star, m}^{\varepsilon, \n}\}_{m \in \N}$ another subsequence that converges %
to~$w_{\star}^{\varepsilon, \n}$
almost everywhere in~$\Omega$. 

\paragraph{Convergence of~$u_{\star}(w_{\star, m}^{\varepsilon, \n})$.} 
Let~$\{w_{\star, m}^{\varepsilon, \n}\}_{m \in \N}$ be the subsequence extracted in the previous step. Since~$u_p : \R^d \to (0, \equilp)$ is continuous, %
the sequence~$\{\up(w_{p, m}^{\varepsilon, \n})\}_{m \in \N}$ converges almost everywhere in~$\Omega$ to~$\up(w^{\varepsilon, \n})$, and is uniformly bounded in the~$L^{\infty}(\Omega)$ norm by~$\equilp$. 
Thus, the convergence in~\eqref{eq:strong-u_p} is a consequence of the dominated convergence theorem (see, e.g., \cite[Thm.~4.2 in \S4.1]{Brezis:2011}).

The proof of~\eqref{eq:strong-u_q} is less immediate as~$u_q$ is not uniformly bounded. 
In this case, we first use the discrete entropy stability estimates in Theorem~\ref{TH::DISCRETE_ENTROPY} and the discrete compact embedding in Lemma~\ref{lemma:LDG-product}, part~\ref{P2}, to conclude that~$\{w_{q, m}^{\varepsilon, \n}\}_{m \in \N}$ is uniformly bounded in~$L^{\infty}(\Omega)$ with respect to~$h_m$ (recall that~$\varepsilon$ and~$\Tt$ are fixed). This implies that $w_{q, m}^{\varepsilon, \n}(\bx) \in \mathcal{K}_q$ almost everywhere in~$\Omega$, for some compact set~$\mathcal{K}_q \subset \R$ independent of~$m$. 
Consequently, since~$u_q$ is continuous on~$\R$, $u_q(w_{q, m}^{\varepsilon, \n})(\bx)$ belongs to the compact set~$u_q(\mathcal{K}_q)$ for almost all~$\bx\in\Omega$. 
The convergence in~\eqref{eq:strong-u_q} then follows analogously to that in~\eqref{eq:strong-u_p}.

In a similar way, it can be proven that, for~$\star = p, q$, 
\begin{equation}
\label{eq:strong-f-u}
    f_{\star}(u_p(w_{p, m}^{\varepsilon, \npo}), u_q(w_{q, m}^{\varepsilon, \npo})) \to f_{\star}(u_p(w_{p}^{\varepsilon, \npo}), u_q(w_{q}^{\varepsilon, \npo})) \quad \text{ strongly in }L^r(\Omega) \text{ for all~$r \in [1, \infty)$.}
\end{equation}

\paragraph{Semidiscrete-in-time limit problem.}
In order to prove the last part of the theorem, we first rewrite the fully discrete scheme~\eqref{EQ::MATRICIAL_COMPACT} in the following variational form: for~$\star = p, q$,
\begin{alignat}{3}
\label{eq:weak-form-sigma-hm}
\left( \D s_{\star}''(u_{\star}(w_{\star, m}^{\varepsilon, \npo})) \sigmahmstar^{\varepsilon, \npo}, \bphim \right)_{\Omega} & = - \left(\D \nablaLDG w_{\star, m}^{\varepsilon, \npo},  \bphim \right)_{\Omega} & & \qquad \forall \bphim \in \Rhm, \\
\nonumber
\varepsilon  \left(w_{\star, m}^{\varepsilon, \npo}, \psim\right)_{\LDG} + \frac{1}{\tau_{n+1}} & \Big(u_{\star}(w_{\star, m}^{\varepsilon, \npo})  - u_{\star, m}^{\varepsilon, \n}, \psim \Big)_{\Omega}  \\
\nonumber
\quad - \left(\D \sigmahmstar^{\varepsilon, \npo}, \nablaLDG \psim\right)_{\Omega} & + \sum_{F \in \Fhmo}  \left(\eta_F \mathsf{h}^{-1} \jump{w_{\star, m}^{\varepsilon, \npo}}_{\sf N}, \jump{\psim}_{\sf N} \right)_{F} \\
\label{eq:weak-form-w-hm}
& = \left(f_{\star}(u_p(w_{p, m}^{\varepsilon, \npo}), u_q(w_{q, m}^{\varepsilon, \npo}), \psim \right)_{\Omega} & & \qquad \forall \psim \in \Whm.
\end{alignat}
We aim to pass to the limit in each term in~\eqref{eq:weak-form-w-hm}.

Since~$\boldsymbol{\phi} \mapsto D \boldsymbol{\phi}$ defines a continuous linear operator, using the uniform boundedness of~$\D$ and the weak convergence in~\eqref{eq:weak-nablaLDG-wm} of~$\nablaLDG w_{\star, m}^{\varepsilon, \npo}$,  we can deduce (see~\cite[Thm.~3.10 in~\S3.3]{Brezis:2011})
\begin{equation}
\label{eq:weak-D-nablaLDG-wm}
\D \nablaLDG w_{\star, m}^{\varepsilon, \npo} \rightharpoonup \D \nabla w_{\star}^{\varepsilon, \npo} \quad \text{ weakly in }L^2(\Omega)^d.
\end{equation}
Moreover, the discrete entropy stability estimate~\eqref{eq:entropy2_nosigmaq}, combined with bound~\eqref{eq:bound-sigmahq}, implies that~$\{\sigmahmstar^{\varepsilon, \npo}\}_{m \in \N}$ is uniformly bounded in~$L^2(\Omega)^d$ with respect to~$h_m$ (recall again that~$\varepsilon$ and~$\Tt$ are fixed). Since~$L^2(\Omega)^d$ is (trivially) reflexive, we can extract a subsequence, still denoted by~$\{\sigmahmstar^{\varepsilon, \npo}\}_{m \in \N}$, such that (see \cite[Thm.~3.18 in~\S3.5]{Brezis:2011})
\begin{equation*}
    \sigmahmstar^{\varepsilon, \npo} \rightharpoonup \vsigma_{\star}^{\varepsilon, \npo} \quad \text{ weakly in } L^2(\Omega)^d,
\end{equation*}
for some~$\vsigma_{\star}^{\varepsilon, \npo} \in L^2(\Omega)^d$. Using the continuity and uniform boundedness of~$\D$, $s_p''(\cdot)$, and~$s_q''(\cdot)$ on~$\R$, $u_p(\mathcal{K}_p)$, and~$u_q(\mathcal{K}_q)$, respectively, we have
\begin{subequations}
\begin{alignat}{3}
\label{eq:weak-linear-sigma}
    \D \sigmahmstar^{\varepsilon, \npo} & \rightharpoonup \D \vsigma_{\star}^{\varepsilon, \npo} & & \quad \text{ weakly in }L^2(\Omega)^d, \\
\label{eq:weak-nonlinear-sigma}
    \D s_{\star}''(u_{\star}(w_{\star, m}^{\varepsilon, \npo})) \sigmahmstar^{\varepsilon, \npo} & \rightharpoonup \D s_{\star}''(u_{\star}(w_{\star}^{\varepsilon, \npo})) \vsigma_{\star}^{\varepsilon, \npo} & & \quad \text{ weakly in }L^2(\Omega)^d.
\end{alignat}
\end{subequations}
It only remains to prove that~$\vsigma_{\star}^{\varepsilon, \npo} = - \nabla u_{\star}(w_{\star}^{\varepsilon, \npo})$ almost everywhere in~$\Omega$. 
The density of~$H^1(\Omega)$ in~$L^2(\Omega)$, and the approximation properties of~$\Rhm$ (see~\cite[Thm.~4.4.20 in~\S4.4]{Brenner_Scott:2008}) imply that, for any~$\bphi \in L^2(\Omega)^d$, there is a sequence~$\{\bphim\}_{m \in \N}$ that converges strongly to~$\bphi$ in~$L^2(\Omega)^d$. 
This, combined with equation~\eqref{eq:weak-form-sigma-hm}, the convergences in~\eqref{eq:weak-D-nablaLDG-wm} and~\eqref{eq:weak-nonlinear-sigma}, and the chain rule~$\nabla w = \nabla s_{\star}'(u_{\star}(w)) = s_{\star}''(u_{\star}(w)) \nabla u_{\star}(w)$, implies 
\begin{equation}
\label{eq:sigma-nabla-w-identity}
    \left(\D s_{\star}''(u_{\star}(w_{\star}^{\varepsilon, \npo})) \vsigma_{\star}^{\varepsilon, \npo}, \bphi \right)_{\Omega} = -\left(\D s_{\star}''(u_{\star}(w_{\star}^{\varepsilon, \npo})) \nabla u_{\star}(w_{\star}^{\varepsilon, \npo}), \bphi \right)_{\Omega} \qquad \forall \bphi \in L^2(\Omega)^d. 
\end{equation}
Taking~$\bphi = \vsigma_{\star}^{\varepsilon, \npo} + \nabla u_{\star} (w_{\star}^{\varepsilon, \npo})$ in~\eqref{eq:sigma-nabla-w-identity} and using that~$\D$ is uniformly positive definite with constant larger than or equal to~$\dext$, we obtain
\begin{equation*}
    \dext \essinf_{\bx \in \Omega} s_{\star}''(u_{\star}(w_{\star}^{\varepsilon, \npo}(\bx))) \Norm{\vsigma_{\star}^{\varepsilon, \npo} + \nabla u_{\star} (w_{\star}^{\varepsilon, \npo})}{L^2(\Omega)^d}^2 = 0.
\end{equation*}
Recalling that~$\essinf_{\bx \in \Omega} s_p''(u_p(w_p^{\varepsilon, \npo}(\bx))) \geq 4 \equilp^{-1}$ and~$\essinf_{\bx \in \Omega} s_q''(u_q(w_q^{\varepsilon, \npo}(\bx))) = \essinf_{\by \in \mathcal{K}_q} s_q''(u_q(\by)) > 0$, we deduce that~$\vsigma_{\star}^{\varepsilon, \npo} = - \nabla u_{\star}(w_{\star}^{\varepsilon, \npo})$ almost everywhere in~$\Omega$, as desired.

For any~$\psi \in H^2(\Omega)$, let~$\{\psim\}_{m \in \N}$ be a sequence such that~$\psim \in \mathcal{W}^2(\Thm) \cap H^1(\Omega)$ for~$m \in \N$, and~$\psim \to \psi$ strongly in~$H^1(\Omega)$. Then, combining Lemma~\ref{lemma:LDG-product}, part~\ref{P3} with the convergences in~\eqref{eq:strong-u}, \eqref{eq:weak-linear-sigma}, and~\eqref{eq:strong-f-u}, as well as the fact that~$\jump{\psim}_{\sf N} = 0$, we obtain that the pair~$(w_p^{\varepsilon, \npo}, w_q^{\varepsilon, \npo})$ solves the semidiscrete-in-time formulation~\eqref{eq:h-limit-formulation}. 

\paragraph{Entropy stability.}
The entropy stability estimate in~\eqref{eq:entropy-h-limit} can be proven analogously to that in Theorem~\ref{TH::DISCRETE_ENTROPY}, also using~the bounds in~\eqref{eq:bound-sigmahq} and~\eqref{EQ::HETER_SQRTQ}.
\end{proof}

\subsection{Convergence to a weak solution of the continuous problem\label{SEC::tau-varepsilon}}
We now consider the semidiscrete-in-time limit problem~\eqref{eq:h-limit-formulation}, which admits a solution~$(\wp^{\varepsilon, \npo}, \wq^{\varepsilon, \npo})$; see Theorem~\ref{eq:strong-u}. 
For the sake of simplicity, we assume that~$\Tt$ is a uniform partition of~$(0, T)$ and that~$\varepsilon \le 1$. 
We prove that, as~$(\varepsilon, \tau) \to (0, 0)$, any sequence~$\{(w_p^{\varepsilon, \npo}, w_q^{\varepsilon, \npo})\}_{(\varepsilon, \tau)}$ of such solutions converges (up to a subsequence) to a weak solution of~\eqref{EQN::HETERODIMER} satisfying the physical bounds~\eqref{EQ:HETER_BOUNDSCQ}. The precise notion of weak solution is specified in Theorem~\ref{thm:convergence-eps-tau} below.

Based on the sequence of functions~$\{(w_p^{\varepsilon, \npo}, w_q^{\varepsilon, \npo})\}_{(\varepsilon,\tau)}$ at the discrete times~$\{\tn\}_{n = 0}^{N_t}$,  %
we conveniently introduce the piecewise constant functions in time~$w_p^{\varepsilon, (\tau)} : \QT \to \R$ and~$w_q^{\varepsilon, (\tau)} : \QT \to \R$, defined on the whole space--time domain~$\QT$, as follows:
\begin{equation*}
\begin{split}
    w_p^{\varepsilon, (\tau)}(\cdot, t) & \coloneqq  w_p^{\varepsilon, \n}(\cdot) \qquad \text{for all~$t \in (\tnmo, \tn]$ with~$n \in \{1, \ldots, N_t\}$},\\
    w_q^{\varepsilon, (\tau)}(\cdot, t) & \coloneqq  w_q^{\varepsilon, \n}(\cdot) \qquad \text{for all~$t \in (\tnmo, \tn]$ with~$n \in \{1, \ldots, N_t\}$}
\end{split}
\end{equation*}
We further use the notation~$p^{\varepsilon, (\tau)} \coloneqq  u_p(w_p^{\varepsilon, (\tau)}) : \QT \to (0, \equilp)$ and~$q^{\varepsilon, (\tau)} \coloneqq  u_q(w_q^{\varepsilon, (\tau)}) : \QT \to (0, \infty)$, and define the discrete derivatives
\begin{alignat*}{3}
    \partial_\tau %
    p^{\varepsilon, (\tau)} (\cdot, t) & \coloneqq  \frac{1}{\tau} \big(p^{\varepsilon, \n}(\cdot) -  p^{\varepsilon, \nmo}(\cdot) \big) & & \qquad \text{for all~$t \in (\tnmo, \tn]$ with~$n \in \{1, \ldots, N_t\}$},  \\
    \partial_\tau %
    q^{\varepsilon, (\tau)} (\cdot, t) & \coloneqq  \frac{1}{\tau} \big(q^{\varepsilon, \n}(\cdot) -  q^{\varepsilon, \nmo}(\cdot)\big) & & \qquad \text{for all~$t \in (\tnmo, \tn]$ with~$n \in \{1, \ldots, N_t\}$},
\end{alignat*}
where~$p^{\varepsilon, (0)} \coloneqq  p_0$ and~$q^{\varepsilon, (0)} \coloneqq  q_0$.

\begin{theorem}[Convergence as~$(\varepsilon, \tau) \to (0, 0)$] 
\label{thm:convergence-eps-tau}
Assume that~$\tau < 1/(2C_f)$, and that the initial conditions~$p_0$ and~$q_0$ satisfy~\eqref{eq:ic}.
Let~$r = 2 + 4/d$, $\mu = (d+2)/(d+1)$, and~$\mu' = \mu/(\mu - 1)$. Then, there exists a pair~$(p,q)$, with
\begin{alignat*}{3}
p\in H^1(0,T;H^2(\Omega)')\cap L^2(0,T;H^1(\Omega))\cap L^{\nu}(\QT) & \qquad \text{ for all~$\nu \in [1, \infty)$}, \\
q\in W^{1,\mu}(0,T;W^{2,\mu'}(\Omega)')\cap L^\mu(0,T;W^{1,\mu}(\Omega)) \cap L^{\omega}(\QT) 
& \qquad \text{ for all~$\omega \in [1, r/2)$},
\end{alignat*}
such that, as~$(\varepsilon, \tau) \to (0, 0)$, up to subsequences that are not relabeled, the sequence~$\{(w_p^{\varepsilon, (\tau)}, w_q^{\varepsilon, (\tau)})\}_{(\varepsilon, \tau)}$ of space--time reconstructions of any sequence of solutions~$\{(w_p^{\varepsilon, \npo}, w_q^{\varepsilon, \npo})\}_{(\varepsilon, \tau)}$  to the semidiscrete-in-time formulation~\eqref{eq:h-limit-formulation} satisfies
\begin{subequations}
\begin{alignat}{3}
    \label{eq:weak-w}
    \sqrt{\varepsilon} w_{\star}^{\varepsilon, (\tau)} & \rightharpoonup 0 & & \quad \text{ weakly in $L^2(\QT)$ for~$\star = p, q$}, \\
    \label{eq:weak-nabla-w}
    \sqrt{\varepsilon} \nabla w_{\star}^{\varepsilon, (\tau)} & \rightharpoonup 0 & & \quad \text{ weakly in $L^2(\QT)^d$ for~$\star = p, q$}, \\
    \label{eq:weak-hessian-w}
    \sqrt{\varepsilon} \mathcal{H} w_{\star}^{\varepsilon, (\tau)} & \rightharpoonup 0 & & \quad \text{ weakly in $L^2(\QT)^{d \times d}$ for~$\star = p, q$}, \\
\label{eq:strong-p-tau-epsilon}
    p^{\varepsilon, (\tau)} & \to p & & \quad \text{ strongly in~$L^{\nu}(\QT)$ for all~$\nu \in [1, \infty)$ and~a.e. in~$\QT$}, \\
    \label{eq:weak-D-nabla-p}
    \D \nabla p^{\varepsilon, (\tau)} & \rightharpoonup \D \nabla p & & \quad \text{ weakly in~$L^2(\QT)^d$}, \\
    \label{eq:weak-dpt-p}
    \dptau p^{\varepsilon, (\tau)} & \rightharpoonup \dpt p & & \quad \text{ weakly in~$L^2(0, T; H^2(\Omega)')$}, \\
    \label{eq:strong-q}
    q^{\varepsilon, (\tau)} & \to  q & & \quad \text{ strongly in~$L^{\omega}(\QT)$ for all~$\omega \in [1, r/2)$ and a.e. in~$\QT$}, \\
    \label{eq:weak-nabla-q}
    \D \nabla q^{\varepsilon, (\tau)} & \rightharpoonup \D \nabla q & & \quad \text{ weakly in~$L^{\mu}(\QT)^d$},\\
    \label{eq:weak-dpt-q} 
    \dptau q & \rightharpoonup \dpt q & & \quad \text{ weakly in~$L^{\mu}(0, T; W^{2, \mu'}(\Omega)')$}, \\
    \label{eq:weak-fpq}
    f_{\star}(p^{\varepsilon, (\tau)}, q^{\varepsilon, (\tau)}) & \rightharpoonup f_{\star}(p, q) & & \quad \text{ weakly in~$L^{r/2}(\QT)$ for~$\star = p, q$.}
\end{alignat}
\end{subequations}
Moreover, the pair~$(p, q)$ solves
\begin{subequations}
\label{eq:final-weak-system}
\begin{alignat}{3}
    \nonumber
    \int_0^T \langle \dpt p, \psi_p \rangle_{H^2(\Omega)' \times H^2(\Omega)} \dt & + \int_{0}^T \int_{\Omega} \D \nabla p \cdot \nabla \psi_p \dx \dt \\
    & = \int_0^T \int_{\Omega} f_p(p, q) \psi_p \dx \dt & & \qquad \forall \psi_p \in L^2(0, T; H^2(\Omega)), \\
    \nonumber
    \int_0^T \langle \dpt q, \psi_q \rangle_{W^{2, \mu'}(\Omega)' \times W^{2, \mu'}(\Omega)}\dt & + \int_0^T \int_{\Omega} \D \nabla q \cdot \nabla \psi_q \dx \dt \\
    & = \int_0^T \int_{\Omega} f_q(p, q) \psi_q \dx \dt & & \qquad \forall \psi_q \in L^{\mu'}(0, T; W^{2, \mu'}(\Omega)).
\end{alignat}
\end{subequations}
\end{theorem}
\begin{proof}
We first observe that, from the entropy stability estimate~\eqref{eq:entropy-h-limit} , it follows that, for any~$t \in (0, T]$,
\begin{alignat}{3}
\nonumber
\varepsilon \sum_{\star = p, q} & \Norm{w_{\star}^{\varepsilon, (\tau)}}{L^2(0, t; H^2(\Omega))}^2 + 4 \equilp^{-1} \dext \Norm{\nabla p^{\varepsilon, (\tau)}}{L^2(0, t; L^2(\Omega)^d)}^2 + 4 \equilq^{-1} \dext \Norm{\nabla \sqrt{q^{\varepsilon, (\tau)}}}{L^2(0, t; L^2(\Omega)^d)}^2 \\
\label{eq:xt-entropy-estimate}
& + \int_{\Omega} s_p(p^{\varepsilon, (\tau)}(\cdot, t)) \dx + \int_{\Omega} s_q(q^{\varepsilon, (\tau)}(\cdot, t)) \dx \le 
\int_{\Omega} s_p(p_0) \dx + \int_{\Omega} s_q(q_0) \dx + C_f T\left(\equilq|\Omega|+C^\ast\right).
\end{alignat}
We point out that the norm in~$L^2(0, t; H^2(\Omega))$ appearing in~\eqref{eq:xt-entropy-estimate} is defined as
\[
\Norm{w}{L^2(0, t; H^2(\Omega))}^2 = \Norm{w}{L^2(0, t;L^2(\Omega))}^2 + \Norm{\nabla w}{L^2(0, t; L^2(\Omega)^d)}^2 + \Norm{\mathcal{H} w}{L^2(0, t; L^2(\Omega)^{d\times d})}^2.
\]

For~$\star = p, q$, estimate~\eqref{eq:xt-entropy-estimate} implies that
\begin{equation}\label{eq:unifboundw}
\text{$\sqrt{\varepsilon}\Norm{ w_{\star}^{\varepsilon, (\tau)}}{L^2(0, T; H^2(\Omega))}$ is uniformly bounded 
with respect to~$\tau$ and~$\varepsilon$,}
\end{equation}
which leads to~\eqref{eq:weak-w}, \eqref{eq:weak-nabla-w}, and~\eqref{eq:weak-hessian-w}.

For the proofs of the limits involving~$p^{\varepsilon, (\tau)}$, we proceed along the lines of~\cite[Thm.~2]{Jungel:2015}.  The definition of~$p^{\varepsilon, (\tau)}$ 
as~$p^{\varepsilon, (\tau)} = u_p(w_p^{\varepsilon, (\tau)})$, and the entropy stability estimate~\eqref{eq:xt-entropy-estimate} imply
that
\begin{equation}\label{eq:unifbounds}
\text{$\Norm{p^{\varepsilon, (\tau)}}{L^{\infty}(\QT)}$,\ $\Norm{\nabla p^{\varepsilon, (\tau)}}{L^2(\QT)^d}$, and~$\Norm{\nabla \sqrt{q^{\varepsilon, (\tau)}}}{L^2(\QT)^d}$ are uniformly bounded with respect to~$\tau$ and~$\varepsilon$.}
\end{equation}
As for~$q^{\varepsilon, (\tau)}$, we use the inequality~$s_q(q) \geq q - 2 \equilq$ for all~$q > 0$ to obtain
\begin{equation}
\int_{\Omega} q^{\varepsilon, (\tau)}(\cdot, t) \dx \le \int_{\Omega} s_q(q^{\varepsilon, (\tau)}(\cdot, t)) + 2 \equilq |\Omega| \qquad \forall t \in (0, T],
\end{equation}
which implies that
\begin{equation}\label{eq:Linf}
\text{$\Norm{q^{\varepsilon, (\tau)}}{L^{\infty}(0, T; L^1(\Omega))}$ is uniformly bounded with respect to~$\varepsilon$ and~$\tau$.}
\end{equation}

Denote by~$\Pi^{(\tau)}$ the~$L^2(0, T)$-orthogonal projection into~$\Pp{0}{\Tt}$.
From the semidiscrete-in-time formulation~\eqref{eq:h-limit-formulation} with~$\star=p$ and the H\"older inequality, for all~$\psi \in L^2(0, T; H^2(\Omega))$, we have
\begin{alignat}{3}
\nonumber
\int_0^T \int_{\Omega} \dptau p^{\varepsilon, (\tau)} \psi \dx \dt & = \int_0^T \int_{\Omega} \dptau p^{\varepsilon, (\tau)} \Pi^{(\tau)} \psi \dx \dt & \\
\nonumber
& \le \varepsilon \Norm{w_p^{\varepsilon, (\tau)}}{L^2(0, T; H^2(\Omega))} \Norm{\Pi^{(\tau)} \psi}{L^2(0, T; H^2(\Omega))} \\
\nonumber
& \quad + \Norm{\D}{L^{\infty}(\Omega)^{d\times d}} \Norm{\nabla p^{\varepsilon, (\tau)}}{L^2(\QT)^d}
\Norm{\Pi^{(\tau)} \nabla \psi}{L^2(\QT)^d}
\\
\label{eq:bound-dtp}
& \quad + \int_0^T \int_{\Omega} |f_p(p^{\varepsilon, (\tau)}, q^{\varepsilon, (\tau)})|\, |\Pi^{(\tau)} \psi| \dx \dt.
\end{alignat}
The last term on the right-hand side of~\eqref{eq:bound-dtp} can be estimated as follows:
\begin{alignat}{3}
\nonumber
\int_0^T \int_{\Omega} |f_p(p^{\varepsilon, (\tau)}, q^{\varepsilon, (\tau)})| |\Pi^{(\tau)} \psi| \dx \dt & \le (\equilp\lambda_p  +\kappa_p)\Norm{\Pi^{(\tau)} \psi}{L^1(\QT)} + \equilp \mu_{pq} \Norm{q}{L^{\infty}(0, T; L^1(\Omega))} \Norm{\Pi^{(\tau)} \psi}{L^1(0, T; L^{\infty}(\Omega))} \\
\nonumber
& \lesssim (\equilp\lambda_p  +\kappa_p)  \Norm{\psi}{L^1(\QT)} + \equilp \mu_{pq} \Norm{q}{L^{\infty}(0, T; L^1(\Omega))} \Norm{\psi}{L^1(0, T; L^{\infty}(\Omega))} \\
\label{eq:bound-fp}
& \lesssim \Norm{\psi}{L^2(0, T; H^2(\Omega))},
\end{alignat}
where the hidden constants are independent of~$\varepsilon$, $\tau$, and~$\psi$, and, in the second step, we have used the stability properties of~$\Pi^{(\tau)}$ (see~\cite[Thm. 18.16(ii)]{Ern_Guermond-I:2020}).
Combining~\eqref{eq:bound-fp} with~\eqref{eq:bound-dtp}, taking into account that~$\varepsilon \Norm{w_p^{\varepsilon, (\tau)}}{L^2(0, T; H^2(\Omega))}$ and~$\Norm{\nabla p^{\varepsilon, (\tau)}}{L^2(Q_T)^d}$ are uniformly bounded with respect to~$\tau$ and~$\varepsilon$, see~\eqref{eq:unifboundw} and~\eqref{eq:unifbounds}, we deduce that
\[
\text{$\Norm{\dptau p^{\varepsilon, (\tau)}}{L^2(0, T; H^2(\Omega)')}$ is uniformly bounded with respect to~$\tau$ and~$\varepsilon$.}
\]
We can then apply the Aubin lemma in the version in~\cite[Thm.~1]{Dreher_Jungel:2012} with $p=2$, $r=1$, $X=H^1(\Omega)$, $B=L^2(\Omega)$, and~$Y=H^2(\Omega)'$ (in the notation of~\cite[Thm.~1]{Dreher_Jungel:2012}). %
Consequently, as in the proof of~\cite[Thm.~2]{Jungel:2015}, up to a subsequence that is not relabeled, we obtain~\eqref{eq:strong-p-tau-epsilon}, \eqref{eq:weak-D-nabla-p}, and \eqref{eq:weak-dpt-p}.

We now turn to the limits involving~$q^{\varepsilon, (\tau)}$. Let~$r = 2 + 4/d$ and~$\theta r = 2$. Using the Gagliardo--Nirenberg inequality (see the version in~\cite[Thm.~1.24]{Roubivek:2005}) and the H\"older inequality, we get
\begin{alignat*}{3}
\Norm{\sqrt{q^{\varepsilon, (\tau)}}}{L^r(\QT)}^r & \le C_{\mathrm{GN}} \int_0^T  \Norm{\sqrt{q^{\varepsilon, (\tau)}}}{L^2(\Omega)}^{r(1 - \theta) } \Norm{ \sqrt{q^{\varepsilon, (\tau)}}}{H^1(\Omega)}^{r \theta } \dt  \\
 & \le C_{\mathrm{GN}} \Norm{\sqrt{q^{\varepsilon, (\tau)}}}{L^{\infty}(0, T; L^2(\Omega))}^{4/d} \Norm{\sqrt{q^{\varepsilon, (\tau)}}}{L^2(0, T; H^1(\Omega))}^2.
\end{alignat*}
Therefore, owing to~\eqref{eq:Linf} and~\eqref{eq:unifbounds}, for~$r = 4$ if~$d = 2$, and~$r = 10/3$ if~$d = 3$,
\begin{equation}
\label{eq:unifboundq}
\Norm{\sqrt{q^{\varepsilon, (\tau)}}}{L^r(\QT)} \text{ and } \Norm{q^{\varepsilon, (\tau)}}{L^{r/2}(\QT)} \text{ are uniformly bounded with respect to~$\varepsilon$ and~$\tau$.}
\end{equation}
Moreover, we define $\mu \coloneqq  (d + 2)/(d + 1)$. Using the identity~$\nabla q = 2 \sqrt{q} \,\nabla \sqrt{q}$ and the H\"older inequality\footnote{$\Norm{fg}{1}\le\Norm{f}{p}\Norm{g}{q}$ with~$p=r/\mu$, $q=r/(r-\mu)$, noting that $(r-\mu)=\frac{(d+2)^2}{d(d+1)}>0$.}, we obtain
\begin{equation*}
\begin{split}
\Norm{\nabla q^{\varepsilon, (\tau)}}{L^{\mu}(\QT)^d}^{\mu} = 2^{\mu} \Norm{\sqrt{q^{\varepsilon, (\tau)}}\, \nabla \sqrt{q^{\varepsilon, (\tau)}}}{L^{\mu}(\QT)^{d}}^{\mu} 
& \le 2^{\mu} \int_{\QT} %
\big(q^{\varepsilon, (\tau)}\big)^{\mu/2} |\nabla \sqrt{q^{\varepsilon, (\tau)}}|^{\mu} \dV %
\\
& \le 2^{\mu} \Big(\int_{\QT} \big(q^{\varepsilon, (\tau)}\big)^{r/2} \dV \Big)^{\mu/r} \Big(\int_{\QT} |\nabla \sqrt{q^{\varepsilon, (\tau)}}|^2 \dV  \Big)^{\frac{r-\mu}{r}} \\
& = 2^{\mu} \Norm{q^{\varepsilon, (\tau)}}{L^{r/2}(\QT)}^{\mu/2} \Norm{\nabla \sqrt{q^{\varepsilon, (\tau)}}}{L^2(\QT)^d}^{\frac{2(r-\mu)}{r}},
\end{split}
\end{equation*}
which implies that
\begin{equation}\label{eq:unifboundGradq}
\text{$\Norm{\nabla q^{\varepsilon, (\tau)}}{L^{\mu}(\QT)^d}$ is uniformly bounded with respect to~$\tau$ and~$\varepsilon$,}
\end{equation}
with~$\mu = 4/3$ for~$d = 2$, and~$\mu = 5/4$ for~$d = 3$. 

Proceeding similarly as for~\eqref{eq:bound-dtp}, for all~$\psi \in L^{\mu'}(0, T; W^{2, \mu'}(\Omega))$, we have
\begin{equation}
\label{eq:bound-dtq}
\begin{split}
\int_0^T \int_{\Omega} \dptau q^{\varepsilon, (\tau)} \psi \dx \dt & = \int_0^T \int_{\Omega} \dptau q^{\varepsilon, (\tau)} \Pi^{(\tau)} \psi \dx \dt \\
& \lesssim \varepsilon \Norm{w_q^{\varepsilon, (\tau)}}{L^2(0, T; H^2(\Omega))} \Norm{\psi}{L^2(0, T; H^2(\Omega))} \\
& \quad + \Norm{\D}{L^{\infty}(\Omega)^{d\times d}} \Norm{\nabla q^{\varepsilon, (\tau)}}{L^{\mu}(\QT)^d} \Norm{\nabla \psi}{L^{\mu'}(\QT)^d} \\
& \quad + \Norm{f_q(p^{\varepsilon, (\tau)}, q^{\varepsilon, (\tau)})}{L^{r/2}(\QT)} \Norm{\psi}{L^{r/(r-2)}(\QT)}, %
\end{split}
\end{equation}
where~$\mu' = \mu/(\mu - 1) \geq 2$ (namely, $\mu' = 4$ if~$d = 2$, and~$\mu' = 5$ if~$d = 3$) and~$r/(r-2) \le \mu'$. Since~$\Norm{p^{\varepsilon, (\tau)}}{L^{\infty}(\QT)}$ is bounded uniformly with respect to~$\varepsilon$ and~$\tau$ (see~\eqref{eq:unifbounds}), and the nonlinear terms~$\fp(p, q)$ and~$\fq(p, q)$ depend linearly on~$q$, we also have that 
\begin{equation}\label{eq:unifboundfpfq}
\text{$\Norm{f_p(p^{\varepsilon, (\tau)}, q^{\varepsilon, (\tau)})}{L^{r/2}(\QT)}$ and~$\Norm{f_q(p^{\varepsilon, (\tau)}, q^{\varepsilon, (\tau)})}{L^{r/2}(\QT)}$ are uniformly bounded with respect to~$\varepsilon$ and~$\tau$.}
\end{equation}
Then, from~\eqref{eq:bound-dtq}, together with~\eqref{eq:unifboundw}, \eqref{eq:unifboundGradq}, and~\eqref{eq:unifboundfpfq}, we conclude that
\begin{equation}
\label{eq:unifbounddptq}
\text{$\Norm{\dptau q^{\varepsilon, (\tau)}}{L^{\mu}(0, T; W^{2, \mu'}(\Omega)')}$ is uniformly bounded with respect to~$\varepsilon$ and~$\tau$.}
\end{equation}

By the Rellich--Kondrachov theorem, the space~$W^{1, \mu}(\Omega)$ is compactly embedded in~$L^{\mu}(\Omega)$. Since~$L^{\mu}(\Omega)$ is continuously embedded in~$W^{2, \mu'}(\Omega)'$, we can use again the Aubin lemma in the version in~\cite[Thm.~1]{Dreher_Jungel:2012} now with $p=\mu$, $r=1$, $X=W^{1, \mu}(\Omega)$, $B=L^{\mu}(\Omega)$, and~$Y=W^{2, \mu'}(\Omega)'$ (in the notation of~\cite[Thm.~1]{Dreher_Jungel:2012}), to conclude that there exists~$q \in L^{\mu}(\QT)$ such that, as~$(\varepsilon, \tau) \to (0, 0)$,
\begin{equation*}
q^{\varepsilon, (\tau)} \to q \quad \text{ strongly in~$L^{\mu}(\QT)$,}
\end{equation*}
and, due to~\cite[Thm.~4.9]{Brezis:2011}, up to a subsequence that is not relabeled, it also converges a.e. in~$\QT$. This, together with the uniform bound in~\eqref{eq:unifboundq} and~\cite[Lemma~5]{Chen_Jungel:2006}, leads to~\eqref{eq:strong-q}.
Furthermore, by weak compactness, the uniform bounds in~\eqref{eq:unifboundGradq} and~\eqref{eq:unifbounddptq} imply~\eqref{eq:weak-nabla-q} and~\eqref{eq:weak-dpt-q}.
The uniform boundedness in~\eqref{eq:unifboundfpfq}, the continuity of~$f_p(\cdot, \cdot)$ and~$f_q(\cdot, \cdot)$, and the fact that, up to subsequences, $p^{\varepsilon, (\tau)} \to p$ and~$q^{\varepsilon, (\tau)} \to q$ a.e. in~$\QT$ give~\eqref{eq:weak-fpq}. This completes the proof of~\eqref{eq:weak-w}--\eqref{eq:weak-fpq}.

Finally, the fact that~$(p, q)$ solves~\eqref{eq:final-weak-system} follows by taking suitable test functions and passing to the limit on each term of~\eqref{eq:h-limit-formulation}, and the proof is complete.
\end{proof}

\begin{remark}
As a byproduct of Theorem~\ref{thm:convergence-eps-tau}, we have that problem~\eqref{EQN::HETERODIMER} admits a weak solution~$(p,q)$ in the sense of~\eqref{eq:final-weak-system}.
\eremk
\end{remark}

\begin{remark}[Structure of the convergence analysis]
Our convergence analysis follows the same structure of the one presented in~\cite{Gomez-Jungel-Perugia:2024} for cross-diffusion systems.
In Sections~\ref{SEC::h-convergence} and~\ref{SEC::tau-varepsilon}, we first show convergence as~$h \to 0$, and then proceed to show that the resulting semidiscrete-in-time limit problem converges to a weak solution to system~\eqref{SEC::MODEL}  as~$\tau \to 0$ and~$\varepsilon \to 0$. 
The main bottleneck in proving simultaneous convergence in space and time results from the fact that the entropy stability estimate in Theorem~\ref{TH::DISCRETE_ENTROPY} provides a bound on~$\Norm{\whstar}{L^{\infty}(\Omega)}$ (with~$\star = p, q$) that is uniform only with respect to~$h$.
In addition, while~$\sigma_p = - \nabla p$ and~$\sigma_q = -\nabla q$ hold on the continuous level, these identities hold only weakly on the discrete level.
Consequently, the boundedness of~$\sigmahp$ and~$\sigmahq$ derived from the entropy stability estimate in Theorem~\ref{TH::DISCRETE_ENTROPY} is insufficient to close the discrete compactness argument. 
\eremk
\end{remark}

\section{Numerical results}\label{sec:results}
In this section, we present some numerical tests on two-dimensional space domains to assess the accuracy of the proposed structure-preserving LDG method. 
In Section~\ref{sec:test_case_1}, we discuss the convergence properties in space for a smooth exact solution that is linear in time. Then, in Section~\ref{sec:test_case_2}, we study the accuracy of the method in simulating a traveling-wave solution.
\par
All numerical simulations in this section are performed with the \texttt{lymph} library \cite{antonietti_lymph_2024}, implementing DG methods. We employ structured simplicial meshes with element diameter~$h$ for the space domain, and uniform time steps~$\tau$. 
Since method~\eqref{EQ::VARIATIONAL} %
is highly nonlinear, 
and the entropy stability estimate in Theorem~\ref{TH::DISCRETE_ENTROPY} is proven %
under the assumption of exact integration,
we employ a sufficiently accurate quadrature %
in the simulations. In particular, we %
use a Gauss--Jacobi quadrature rule with $(2\ell+1)^2$ nodes over triangles.
Moreover, the weight parameter~$\gamma_F$ is set to~$1/2$ for all internal facets. For the Newton iterations, given a small tolerance~$\mathsf{tol} >0$, we adopt the following stopping criterion:
\begin{equation}
\label{eq:stopping-criteria}
\min\left\{\sqrt{\|\whp^{k+1}-\whp^k\|_{L^2(\Omega)}^2+\|\whq^{k+1}-\whq^k\|_{L^2(\Omega)}^2}, \ |\mathsf{res}_{k+1}^p+\mathsf{res}_{k+1}^q|\right\} \leq \mathsf{tol},
\end{equation} 
where $\mathsf{res}_{k+1}^\star$ %
(with $\star=p,q$) is the residual of the algebraic system~\eqref{EQ::MATRICIAL_W} for the approximation of~$(\whp^{(n+1)},\whq^{(n+1)})$ at the~$(k+1)$th Newton's iteration.
In the convergence tests reported below, we measure the following~$L^2(\Omega)$ errors at the final time for the concentrations and the fluxes, respectively, of the two variables of the system:
\begin{equation*}
E_\star \coloneqq  \Norm{\star(\cdot, T) - u(\whstar^{(N)})}{L^2(\Omega)} \quad \text{and} \quad E_{\boldsymbol{\sigma}_\star} : = \Norm{\nabla\star(\cdot, T) + \sigmahstar^{(N)}}{L^2(\Omega)^d} \quad \text{with }\star=p,q.
\end{equation*}
While the convergence analysis in Section~\ref{sec:convergence} requires the constraint~$\ell \geq 2$, our numerical experiments perform well even for piecewise linear approximations ($\ell = 1$), as previously observed for similar applications in~\cite{Gomez-Jungel-Perugia:2024,Antonietti-Corti-Gomez-Perugia_2026}.
In all numerical experiments reported in this section, the penalty parameter~$\varepsilon$ is set to~$0$, as the Newton method shows stable and reliable convergence even without the penalty term (see Remark~\ref{RMK:PENALTY-TERM}). A posteriori, this indicates that the computed solutions remain sufficiently far from the singular limits. This is likely connected to the continuous structure of the problem, where the equilibrium~$(\equilp, 0)$ lying on the singular limit is unstable, while the stable equilibrium~$(\lambda_q/\mu_{pq}, \equilq)$ is well-separated from such limits; see Footnote~\ref{footnote:equilibria} below. We emphasize that this is not guaranteed in general, and taking~$\varepsilon>0$ may be necessary in practical computations, as also observed in previous works employing the structure-preserving LDG method applied to other model problems~\cite{Gomez-Jungel-Perugia:2024,Antonietti-Corti-Gomez-Perugia_2026}.

\subsection{Test case 1: Convergence analysis}
\label{sec:test_case_1}
For the numerical tests in this section, 
we consider the space domain~$\Omega=(0,1)^2$
and homogeneous Neumann boundary conditions on the boundary~$\Gamma \times (0, T)$. 
For the nonlinear Newton solver, we 
adopt the stopping criterion~\eqref{eq:stopping-criteria} with~$\mathsf{tol} = 10^{-10}$. 
For both species, we select the diffusion tensor~$\mathbf{D}=\mathbb{I}_2$, where~$\mathbb{I}_2$ represents the identity matrix of size~$2$. Concerning the reaction coefficients, we fix~$\lambda_p=\lambda_q= 1$ and the conversion term~$\mu_{pq}=0.5$. 

For this test case, instead of using a constant production rate~$\kappa_p$ for~$p$, we allow nonconstant source terms in both equations of~\eqref{EQN::HETERODIMER} in order to construct a manufactured solution that is linear in time. This choice allows us to highlight the properties of the space discretization, neglecting the error due to the time integration scheme. Then, we set $\equilp = \equilq = 1$ in the change of variables~\eqref{EQ::HETER_UPUQ}, independently of the other parameters, since the choice of nonconstant functions on the right-hand side of both equations breaks the equilibrium structure of system~\eqref{EQN::HETERODIMER}. Thus, we consider system~\eqref{EQN::HETERODIMER} with initial conditions and additional source terms on the two right-hand sides chosen so that the problem admits the following exact solution:
\begin{equation*}
    p(x,y,t) = q(x,y,t) = \frac{1}{4}\left(\cos(2 \pi x)\cos(2 \pi y)+2\right) (1-t).
\end{equation*}
Choosing the same expression for~$p$ and~$q$ is useful for assessing the impact of the two different changes of variables, which take values over different ranges.
\begin{figure}[t!]
    \begin{subfigure}[b]{0.5\textwidth}
          \resizebox{\textwidth}{!}{\definecolor{mycolor2}{rgb}{0.00000,1.00000,1.00000}%
\begin{tikzpicture}
\begin{axis}[%
width=3.875in,
height=2.50in,
at={(2.6in,1.099in)},
scale only axis,
xmode=log,
xmin=4.4200e-2,
xmax=3.5355e-1,
xminorticks=true,
xlabel = {$h$ [-]},
ylabel = {$\|p(\cdot,T)-u_p(w_{p,h}^{(N)})\|_{L^2(\Omega)}$},
xticklabel={\pgfmathparse{exp(\tick)}\pgfmathprintnumber{\pgfmathresult}},
x tick label style={
/pgf/number format/.cd, fixed, fixed zerofill,
precision=2},
ymode=log,
ymin=1e-10,
ymax=1e-1,
yminorticks=true,
axis background/.style={fill=white},
title style={font=\bfseries},
title={Errors $E_{p}$},
xmajorgrids,
xminorgrids,
ymajorgrids,
yminorgrids,
legend style={at={(0.96,0.40)},legend cell align=left, draw=white!15!black}
]
              
\addplot [color=red, line width=2.0pt, mark=*]
  table[row sep=crcr]{%
3.5355e-01  4.7805e-02\\
1.7678e-01  1.6187e-02\\
8.8388e-02  4.4109e-03\\
4.4200e-02  1.1000e-03\\
};
\addlegendentry{$\ell=1$}

\addplot [color=violet, line width=2.0pt, mark=*]
  table[row sep=crcr]{%
3.5355e-01  6.7666e-03\\
1.7678e-01  9.2947e-04\\
8.8388e-02  1.1861e-04\\
4.4200e-02  1.4932e-05\\
};
\addlegendentry{$\ell=2$}

\addplot [color=blue, line width=2.0pt, mark=*] 
  table[row sep=crcr]{%
3.5355e-01  1.3870e-03\\
1.7678e-01  8.2299e-05\\
8.8388e-02  5.1475e-06\\
4.4200e-02  3.1741e-07\\
};
\addlegendentry{$\ell=3$}

\addplot [color=teal, line width=2.0pt, mark=*]
  table[row sep=crcr]{%
3.5355e-01  1.2721e-04\\
1.7678e-01  8.8423e-06\\
8.8388e-02  3.1380e-07\\
4.4200e-02  1.0224e-08\\
};
\addlegendentry{$\ell=4$}

\addplot [color=green, line width=2.0pt, mark=*]
  table[row sep=crcr]{%
3.5355e-01  6.4379e-05\\
1.7678e-01  1.5087e-06\\
8.8388e-02  2.3978e-08\\
4.4200e-02  3.8407e-10\\
};
\addlegendentry{$\ell=5$}

\node[right, align=left, text=black, font=\footnotesize]
at (axis cs:0.1005,0.0015) {$2$};

\addplot [color=black, line width=1.5pt]
  table[row sep=crcr]{%
0.100   0.002\\
0.075   0.001125\\
0.100   0.001125\\
0.100   0.002\\
};

\node[right, align=left, text=black, font=\footnotesize]
at (axis cs:0.1005,7e-5) {$3$};

\addplot [color=black, line width=1.5pt]
  table[row sep=crcr]{%
0.100   1.0e-04\\
0.075   4.2e-05\\
0.100   4.2e-05\\
0.100   1.0e-04\\
};

\node[right, align=left, text=black, font=\footnotesize]
at (axis cs:0.1005,3.5e-6) {$4$};

\addplot [color=black, line width=1.5pt]
  table[row sep=crcr]{%
0.100   6e-06\\
0.075   1.9e-6\\
0.100   1.9e-6\\
0.100   6e-6\\
};

\node[right, align=left, text=black, font=\footnotesize]
at (axis cs:0.1005,1.5e-7) {$5$};

\addplot [color=black, line width=1.5pt]
  table[row sep=crcr]{%
0.100   3e-7\\
0.075   7.11e-8\\
0.100   7.11e-8\\
0.100   3e-7\\
};

\node[right, align=left, text=black, font=\footnotesize]
at (axis cs:0.1005,8e-9) {$6$};

\addplot [color=black, line width=1.5pt]
  table[row sep=crcr]{%
0.100   2e-8\\
0.075   3.56e-9\\
0.100   3.56e-9\\
0.100   2e-8\\
};

\end{axis}
\end{tikzpicture}%
}
          \caption{Computed errors $E_{p}$ w.r.t.~the mesh size~$h$.}
        \label{fig:Tria_errors2D_h_L2_p}
    \end{subfigure}
    \begin{subfigure}[b]{0.5\textwidth}
        \resizebox{\textwidth}{!}{\definecolor{mycolor2}{rgb}{0.00000,1.00000,1.00000}%
\begin{tikzpicture}
\begin{axis}[%
width=3.875in,
height=2.50in,
at={(2.6in,1.099in)},
scale only axis,
xmode=log,
xmin=4.420e-2,
xmax=3.5355e-1,
xminorticks=true,
xlabel = {$h$ [-]},
ylabel = {$\|q(\cdot,T)-u_q(w_{q,h}^{(N)})\|_{L^2(\Omega)}$},
xticklabel={\pgfmathparse{exp(\tick)}\pgfmathprintnumber{\pgfmathresult}},
x tick label style={
/pgf/number format/.cd, fixed, fixed zerofill,
precision=2},
ymode=log,
ymin=1e-10,
ymax=1e-1,
yminorticks=true,
axis background/.style={fill=white},
title style={font=\bfseries},
title={Errors $E_{q}$},
xmajorgrids,
xminorgrids,
ymajorgrids,
yminorgrids,
legend style={at={(0.96,0.40)},legend cell align=left, draw=white!15!black}
]
              
\addplot [color=red, line width=2.0pt, mark=*]
  table[row sep=crcr]{%
3.5355e-01  4.3822e-02\\
1.7678e-01  1.5089e-02\\
8.8388e-02  4.1710e-03\\
4.4200e-02  1.1000e-03\\
};
\addlegendentry{$\ell=1$}

\addplot [color=violet, line width=2.0pt, mark=*]
  table[row sep=crcr]{%
3.5355e-01  6.7513e-03\\
1.7678e-01  9.2854e-04\\
8.8388e-02  1.1861e-04\\
4.4200e-02  1.4548e-05\\
};
\addlegendentry{$\ell=2$}

\addplot [color=blue, line width=2.0pt, mark=*] 
  table[row sep=crcr]{%
3.5355e-01  1.6608e-03\\
1.7678e-01  9.9865e-05\\
8.8388e-02  6.1861e-06\\
4.4200e-02  3.7942e-07\\
};
\addlegendentry{$\ell=3$}

\addplot [color=teal, line width=2.0pt, mark=*]
  table[row sep=crcr]{%
3.5355e-01  1.7306e-04\\
1.7678e-01  1.1905e-05\\
8.8388e-02  4.1350e-07\\
4.4200e-02  1.3362e-08\\
};
\addlegendentry{$\ell=4$}

\addplot [color=green, line width=2.0pt, mark=*]
  table[row sep=crcr]{%
3.5355e-01  9.6947e-05\\
1.7678e-01  1.9052e-06\\
8.8388e-02  2.9901e-08\\
4.4200e-02  4.7653e-10\\
};
\addlegendentry{$\ell=5$}

\node[right, align=left, text=black, font=\footnotesize]
at (axis cs:0.1005,0.0015) {$2$};

\addplot [color=black, line width=1.5pt]
  table[row sep=crcr]{%
0.100   0.002\\
0.075   0.001125\\
0.100   0.001125\\
0.100   0.002\\
};

\node[right, align=left, text=black, font=\footnotesize]
at (axis cs:0.1005,5e-5) {$3$};

\addplot [color=black, line width=1.5pt]
  table[row sep=crcr]{%
0.100   8e-05\\
0.075   3.36e-05\\
0.100   3.36e-05\\
0.100   8e-05\\
};

\node[right, align=left, text=black, font=\footnotesize]
at (axis cs:0.1005,3.5e-6) {$4$};

\addplot [color=black, line width=1.5pt]
  table[row sep=crcr]{%
0.100   6e-06\\
0.075   1.9e-6\\
0.100   1.9e-6\\
0.100   6e-6\\
};

\node[right, align=left, text=black, font=\footnotesize]
at (axis cs:0.1005,2e-7) {$5$};

\addplot [color=black, line width=1.5pt]
  table[row sep=crcr]{%
0.100   4e-7\\
0.075   9.48e-8\\
0.100   9.48e-8\\
0.100   4e-7\\
};

\node[right, align=left, text=black, font=\footnotesize]
at (axis cs:0.1005,8e-9) {$6$};

\addplot [color=black, line width=1.5pt]
  table[row sep=crcr]{%
0.100   2e-8\\
0.075   3.56e-9\\
0.100   3.56e-9\\
0.100   2e-8\\
};

\end{axis}
\end{tikzpicture}%
}
          \caption{Computed errors $E_{q}$
          w.r.t.~the mesh size $h$.}
        \label{fig:Tria_errors2D_h_L2_q}
    \end{subfigure}
    \\[0.5mm]
    \begin{subfigure}[b]{0.5\textwidth}
          \resizebox{\textwidth}{!}{\definecolor{mycolor2}{rgb}{0.00000,1.00000,1.00000}%
\begin{tikzpicture}

\begin{axis}[%
width=3.875in,
height=2.50in,
at={(2.6in,1.099in)},
scale only axis,
xmode=log,
xmin=4.4200e-2,
xmax=3.5355e-1,
xminorticks=true,
xlabel = {$h$ [-]},
ylabel = {$\|\nabla p(\cdot,T) + \sigma_{p,h}^{(N)}\|_{L^2(\Omega)^d}$},
xticklabel={\pgfmathparse{exp(\tick)}\pgfmathprintnumber{\pgfmathresult}},
x tick label style={
/pgf/number format/.cd, fixed, fixed zerofill,
precision=2},
ymode=log,
ymin=1e-7,
ymax=1,
yminorticks=true,
axis background/.style={fill=white},
title style={font=\bfseries},
title={Errors $E_{\boldsymbol{\sigma}_p}$},
xmajorgrids,
xminorgrids,
ymajorgrids,
yminorgrids,
legend style={at={(0.96,0.40)},legend cell align=left, draw=white!15!black}
]
              
\addplot [color=red, line width=2.0pt, mark=square*, dotted, mark options=solid]
  table[row sep=crcr]{%
3.5355e-01  6.3152e-01\\
1.7678e-01  3.6123e-01\\
8.8388e-02  1.8708e-01\\
4.4200e-02  9.4400e-02\\
};
\addlegendentry{$\ell=1$}

\addplot [color=violet, line width=2.0pt, mark=square*, dotted, mark options=solid]
  table[row sep=crcr]{
3.5355e-01  1.9427e-01\\
1.7678e-01  5.7444e-02\\
8.8388e-02  1.5031e-02\\
4.4200e-02  3.8000e-03\\
};
\addlegendentry{$\ell=2$}

\addplot [color=blue, line width=2.0pt, mark=square*, dotted, mark options=solid]
  table[row sep=crcr]{
3.5355e-01  5.2491e-02\\
1.7678e-01  6.5969e-03\\
8.8388e-02  8.5243e-04\\
4.4200e-02  1.0736e-04\\
};
\addlegendentry{$\ell=3$}

\addplot [color=teal, line width=2.0pt, mark=square*, dotted, mark options=solid]
  table[row sep=crcr]{
3.5355e-01  6.6443e-03\\
1.7678e-01  8.4984e-04\\
8.8388e-02  6.0060e-05\\
4.4200e-02  3.8670e-06\\
};
\addlegendentry{$\ell=4$}

\addplot [color=green, line width=2.0pt, mark=square*, dotted, mark options=solid]
  table[row sep=crcr]{%
3.5355e-01  3.6144e-03\\
1.7678e-01  1.7091e-04\\
8.8388e-02  5.3878e-06\\
4.4200e-02  1.7268e-07\\
};
\addlegendentry{$\ell=5$}

\node[right, align=left, text=black, font=\footnotesize]
at (axis cs:0.1005,0.072) {$1$};

\addplot [color=black, line width=1.5pt]
  table[row sep=crcr]{%
0.100   0.08\\
0.075   0.06\\
0.100   0.06\\
0.100   0.08\\
};

\node[right, align=left, text=black, font=\footnotesize]
at (axis cs:0.1005,0.006) {$2$};

\addplot [color=black, line width=1.5pt]
  table[row sep=crcr]{%
0.100   0.008\\
0.075   0.0045\\
0.100   0.0045\\
0.100   0.008\\
};

\node[right, align=left, text=black, font=\footnotesize]
at (axis cs:0.1005,5.5e-4) {$3$};

\addplot [color=black, line width=1.5pt]
  table[row sep=crcr]{%
0.100   8e-04\\
0.075   3.36e-04\\
0.100   3.36e-04\\
0.100   8e-04\\
};

\node[right, align=left, text=black, font=\footnotesize]
at (axis cs:0.1005,3.5e-5) {$4$};

\addplot [color=black, line width=1.5pt]
  table[row sep=crcr]{%
0.100   6.00e-05\\
0.075   1.89e-05\\
0.100   1.89e-05\\
0.100   6.00e-05\\
};

\node[right, align=left, text=black, font=\footnotesize]
at (axis cs:0.1005,3e-6) {$5$};

\addplot [color=black, line width=1.5pt]
  table[row sep=crcr]{%
0.100   6.00e-6\\
0.075   1.42e-6\\
0.100   1.42e-6\\
0.100   6.00e-6\\
};

\end{axis}
\end{tikzpicture}%
}
          \caption{Computed errors $E_{\boldsymbol{\sigma}_p}$ w.r.t.~the mesh size~$h$.}
        \label{fig:Tria_errors2D_h_DG_p}
    \end{subfigure}
    \begin{subfigure}[b]{0.5\textwidth}
        \resizebox{\textwidth}{!}{\definecolor{mycolor2}{rgb}{0.00000,1.00000,1.00000}%
\begin{tikzpicture}

\begin{axis}[%
width=3.875in,
height=2.50in,
at={(2.6in,1.099in)},
scale only axis,
xmode=log,
xmin=4.4200e-2,
xmax=3.5355e-1,
xminorticks=true,
xlabel = {$h$ [-]},
ylabel = {$\|\nabla q(\cdot,T) + \sigma_{q,h}^{(N)}\|_{L^2(\Omega)^d}$},
xticklabel={\pgfmathparse{exp(\tick)}\pgfmathprintnumber{\pgfmathresult}},
x tick label style={
/pgf/number format/.cd, fixed, fixed zerofill,
precision=2},
ymode=log,
ymin=1e-7,
ymax=1,
yminorticks=true,
axis background/.style={fill=white},
title style={font=\bfseries},
title={Errors $E_{\boldsymbol{\sigma}_q}$},
xmajorgrids,
xminorgrids,
ymajorgrids,
yminorgrids,
legend style={at={(0.96,0.40)},legend cell align=left, draw=white!15!black}
]
              
\addplot [color=red, line width=2.0pt, mark=square*, dotted, mark options=solid]
  table[row sep=crcr]{%
3.5355e-01  6.1621e-01\\
1.7678e-01  3.5339e-01\\
8.8388e-02  1.8389e-01\\
4.4200e-02  9.3000e-02\\
};
\addlegendentry{$\ell=1$}

\addplot [color=violet, line width=2.0pt, mark=square*, dotted, mark options=solid]
  table[row sep=crcr]{
3.5355e-01  1.9863e-01\\
1.7678e-01  5.8437e-02\\
8.8388e-02  1.5241e-02\\
4.4200e-02  3.9000e-03\\
};
\addlegendentry{$\ell=2$}

\addplot [color=blue, line width=2.0pt, mark=square*, dotted, mark options=solid]
  table[row sep=crcr]{
3.5355e-01  6.2625e-02\\
1.7678e-01  7.9476e-03\\
8.8388e-02  1.0262e-03\\
4.4200e-02  1.2916e-04\\
};
\addlegendentry{$\ell=3$}

\addplot [color=teal, line width=2.0pt, mark=square*, dotted, mark options=solid]
  table[row sep=crcr]{
3.5355e-01  9.2894e-03\\
1.7678e-01  1.1568e-03\\
8.8388e-02  7.9291e-05\\
4.4200e-02  5.0659e-06\\
};
\addlegendentry{$\ell=4$}

\addplot [color=green, line width=2.0pt, mark=square*, dotted, mark options=solid]
  table[row sep=crcr]{%
3.5355e-01  5.6031e-03\\
1.7678e-01  2.1761e-04\\
8.8388e-02  6.7555e-06\\
4.4200e-02  2.1507e-07\\
};
\addlegendentry{$\ell=5$}

\node[right, align=left, text=black, font=\footnotesize]
at (axis cs:0.1005,0.072) {$1$};

\addplot [color=black, line width=1.5pt]
  table[row sep=crcr]{%
0.100   0.08\\
0.075   0.06\\
0.100   0.06\\
0.100   0.08\\
};

\node[right, align=left, text=black, font=\footnotesize]
at (axis cs:0.1005,0.006) {$2$};

\addplot [color=black, line width=1.5pt]
  table[row sep=crcr]{%
0.100   0.008\\
0.075   0.0045\\
0.100   0.0045\\
0.100   0.008\\
};

\node[right, align=left, text=black, font=\footnotesize]
at (axis cs:0.1005,5.5e-4) {$3$};

\addplot [color=black, line width=1.5pt]
  table[row sep=crcr]{%
0.100   8e-04\\
0.075   3.36e-04\\
0.100   3.36e-04\\
0.100   8e-04\\
};

\node[right, align=left, text=black, font=\footnotesize]
at (axis cs:0.1005,3.5e-5) {$4$};

\addplot [color=black, line width=1.5pt]
  table[row sep=crcr]{%
0.100   6.00e-05\\
0.075   1.89e-05\\
0.100   1.89e-05\\
0.100   6.00e-05\\
};

\node[right, align=left, text=black, font=\footnotesize]
at (axis cs:0.1005,3e-6) {$5$};

\addplot [color=black, line width=1.5pt]
  table[row sep=crcr]{%
0.100   6.00e-6\\
0.075   1.42e-6\\
0.100   1.42e-6\\
0.100   6.00e-6\\
};

\end{axis}
\end{tikzpicture}%
}
          \caption{Computed errors $E_{\boldsymbol{\sigma}_q}$
          w.r.t.~the mesh size $h$.}
        \label{fig:Tria_errors2D_h_DG_q}
    \end{subfigure}
    \caption{Test case 1: Computed errors and convergence rates w.r.t.~the mesh size~$h$.}
    \label{fig:Tria_errors2D}
\end{figure}
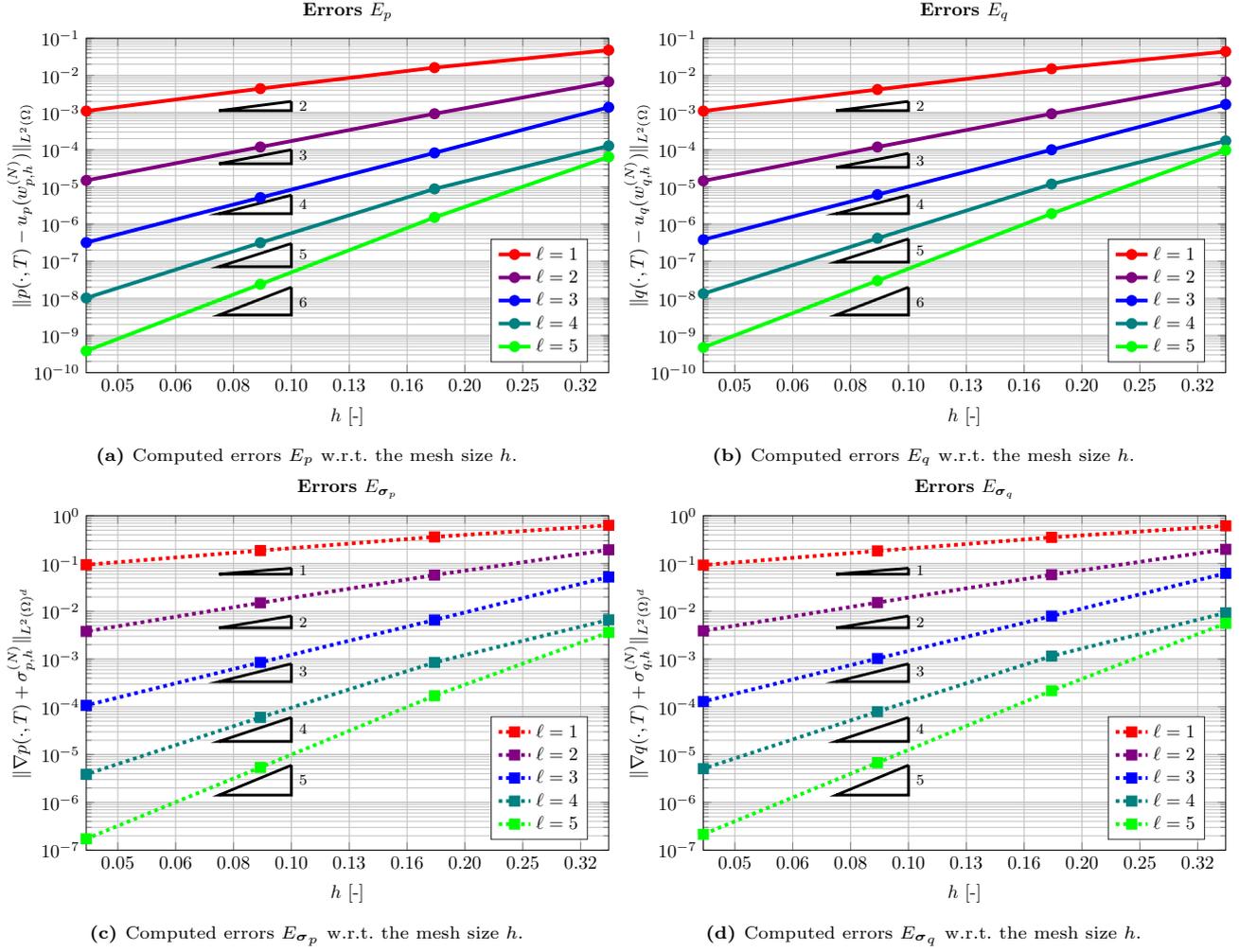
\subsubsection*{Convergence with respect to the mesh size}
We perform a convergence test keeping fixed the polynomial degree of the space approximation $\ell=1,2,3,4,5$ and using, for each degree, different mesh refinements with number of elements $N_\mathrm{el}= 32,\, 128,\, 512,\, 2048$. 
Concerning the time discretization, we take~$\tau= 10^{-3}$ and a final time $T=5\times10^{-2}$. In Figures \ref{fig:Tria_errors2D_h_L2_p} and \ref{fig:Tria_errors2D_h_L2_q}, we report the computed errors $E_{p}$ and~$E_{q}$ %
for the primal variables~$p$ and~$q$, respectively. Moreover, in Figures \ref{fig:Tria_errors2D_h_DG_p} and \ref{fig:Tria_errors2D_h_DG_q}, we report the computed errors $E_{\boldsymbol{\sigma}_p}$ and~$E_{\boldsymbol{\sigma}_q}$ for the approximations of~$\nabla p$ and~$\nabla q$, respectively. In all cases, the errors decrease with optimal convergence rates, namely, of order~$\mathcal{O}(h^{\ell+1})$ for~$E_{\star}$, and of order~$\mathcal{O}(h^\ell)$ for~$E_{\boldsymbol{\sigma}_\star}$.
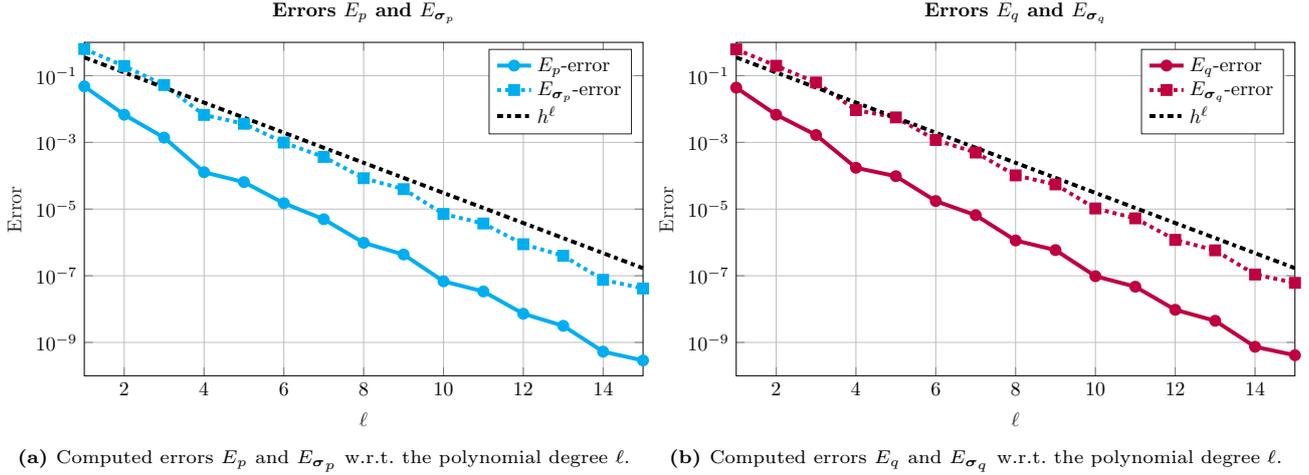
\begin{figure}[t!]
    \begin{subfigure}[b]{0.5\textwidth}
          \resizebox{\textwidth}{!}{\begin{tikzpicture}
\begin{axis}[%
width=3.875in,
height=2.33in,
at={(2.6in,1.099in)},
scale only axis,
xmin=1,
xmax=15,
xlabel style={font=\color{white!15!black}},
xlabel={$\ell$},
ymode=log,
ymin=1e-10,
ymax=1,
yminorticks=true,
ylabel style={font=\color{white!15!black}},
ylabel={Error},
axis background/.style={fill=white},
title style={font=\bfseries},
title={Errors $E_{p}$ and $E_{\boldsymbol{\sigma}_p}$},
xmajorgrids,
xminorgrids,
ymajorgrids,
yminorgrids,
legend style={legend cell align=left, align=left, draw=white!15!black}
]

\addplot [color=cyan, line width=2.0pt, mark=*]
  table[row sep=crcr]{%
    1   4.7805e-02 \\
    2   6.7666e-03 \\
    3   1.3870e-03 \\
    4   1.2721e-04 \\
    5   6.4379e-05 \\
    6   1.4988e-05 \\
    7   4.9600e-06 \\
    8   9.7322e-07 \\
    9   4.3360e-07 \\
    10  6.8096e-08 \\
    11  3.3774e-08 \\
    12  7.2471e-09 \\
    13  3.1455e-09 \\
    14  5.2881e-10 \\
    15  2.8763e-10 \\
};
\addlegendentry{$E_{p}$-error}

\addplot [color=cyan, line width=2.0pt, mark=square*, dotted, mark options=solid]
  table[row sep=crcr]{%
    1   6.3152e-01 \\
    2   1.9427e-01 \\
    3   5.2491e-02 \\
    4   6.6443e-03 \\
    5   3.6144e-03 \\
    6   9.8715e-04 \\
    7   3.6589e-04 \\
    8   8.4302e-05 \\
    9   3.9593e-05 \\
    10  7.1411e-06 \\
    11  3.6960e-06 \\
    12  8.8172e-07 \\
    13  3.9768e-07 \\
    14  7.5502e-08 \\
    15  4.1593e-08 \\
};
\addlegendentry{$E_{\boldsymbol{\sigma}_p}$-error}

\addplot [color=black, line width=2.0pt, dashdotted, mark options=solid]
  table[row sep=crcr]{%
    1   3.5355e-01 \\
    2   1.2500e-01 \\
    3   4.4194e-02 \\
    4   1.5625e-02 \\
    5  	5.5243e-03 \\
    6  	1.9531e-03 \\
    7  	6.9053e-04 \\
    8  	2.4414e-04 \\
    9   8.6317e-05 \\
    10  3.0518e-05 \\
    11  1.0790e-05 \\
    12  3.8147e-06 \\
    13  1.3487e-06 \\
    14  4.7684e-07 \\
    15  1.6859e-07 \\
};
\addlegendentry{$h^\ell$}

\end{axis}
\end{tikzpicture}%
}
          \caption{Computed errors $E_{p}$ and $E_{\boldsymbol{\sigma}_p}$ w.r.t.~the polynomial degree~$\ell$.}
        \label{fig:Tria_errors2D_l_p}
    \end{subfigure}
    \begin{subfigure}[b]{0.5\textwidth}
        \resizebox{\textwidth}{!}{\definecolor{mylime}{rgb}{0.70000,0.90000,0.00000}%

\begin{tikzpicture}
\begin{axis}[%
width=3.875in,
height=2.33in,
at={(2.6in,1.099in)},
scale only axis,
xmin=1,
xmax=15,
xlabel style={font=\color{white!15!black}},
xlabel={$\ell$},
ymode=log,
ymin=1e-10,
ymax=1,
yminorticks=true,
ylabel style={font=\color{white!15!black}},
ylabel={Error},
axis background/.style={fill=white},
title style={font=\bfseries},
title={Errors $E_{q}$ and $E_{\boldsymbol{\sigma}_q}$},
xmajorgrids,
xminorgrids,
ymajorgrids,
yminorgrids,
legend style={legend cell align=left, align=left, draw=white!15!black}
]

\addplot [color=purple, line width=2.0pt, mark=*]
  table[row sep=crcr]{%
    1   4.3822e-02 \\
    2   6.7513e-03 \\
    3   1.6608e-03 \\
    4   1.7306e-04 \\
    5   9.6947e-05 \\
    6   1.7420e-05 \\
    7   6.5459e-06 \\
    8   1.1427e-06 \\
    9   5.8849e-07 \\
    10  9.6959e-08 \\
    11  4.6954e-08 \\
    12  9.5249e-09 \\
    13  4.4635e-09 \\
    14  7.4017e-10 \\
    15  4.1333e-10 \\
};
\addlegendentry{$E_{q}$-error}

\addplot [color=purple, line width=2.0pt, mark=square*, dotted, mark options=solid]
  table[row sep=crcr]{%
    1   6.1621e-01 \\
    2   1.9863e-01 \\
    3   6.2625e-02 \\
    4   9.2894e-03 \\
    5   5.6031e-03 \\
    6   1.1805e-03 \\
    7   4.9611e-04 \\
    8   1.0245e-04 \\
    9   5.4891e-05 \\
    10  1.0431e-05 \\
    11  5.2716e-06 \\
    12  1.1926e-06 \\
    13  5.7868e-07 \\
    14  1.0894e-07 \\
    15  6.1242e-08 \\
};
\addlegendentry{$E_{\boldsymbol{\sigma}_q}$-error}

\addplot [color=black, line width=2.0pt, dashdotted, mark options=solid]
  table[row sep=crcr]{%
    1   3.5355e-01 \\
    2   1.2500e-01 \\
    3   4.4194e-02 \\
    4   1.5625e-02 \\
    5  	5.5243e-03 \\
    6  	1.9531e-03 \\
    7  	6.9053e-04 \\
    8  	2.4414e-04 \\
    9   8.6317e-05 \\
    10  3.0518e-05 \\
    11  1.0790e-05 \\
    12  3.8147e-06 \\
    13  1.3487e-06 \\
    14  4.7684e-07 \\
    15  1.6859e-07 \\
};
\addlegendentry{$h^\ell$}

\end{axis}
\end{tikzpicture}%
}
          \caption{Computed errors $E_{q}$ and $E_{\boldsymbol{\sigma}_q}$ w.r.t.~the polynomial degree~$\ell$.}
        \label{fig:Tria_errors2D_l_q}
    \end{subfigure}
    \caption{Test case 1: Computed errors and convergence rates w.r.t.~the polynomial degree~$\ell$. }
    \label{fig:Tria_errors2D_l}
\end{figure}
\par
\subsubsection*{Convergence with respect to the polynomial degree}
Then, we develop a convergence analysis with respect to the polynomial degree $\ell$. 
To do so, we consider the coarsest triangular mesh in space with~$32$ elements. 
The errors~$E_{p}$ and~$E_{\boldsymbol{\sigma}_p}$ are reported in Figure~\ref{fig:Tria_errors2D_l_p}, and the errors~$E_{q}$ and~$E_{\boldsymbol{\sigma}_q}$ are reported in Figure~\ref{fig:Tria_errors2D_l_q}.
In all plots, we observe spectral convergence with respect to the polynomial degree $\ell$, as expected since the solution is analytic.
\subsection{Test case 2: Traveling-wave solution}
\label{sec:test_case_2}
In this section, we analyze the capabilities of our method for accurately simulating
a traveling-wave solution, while respecting the physical bounds pointwise. For a positive constant~${\mathrm{d}}$, we fix a constant, isotropic diffusion tensor~$\D={\mathrm{d}}\,\mathbb{I}_2$, as well as constant coefficients~$\lambda_p,\, \lambda_q,\, \mu_{pq},\, \kappa_{p}$ in the reaction terms. Then, we consider a solution to the equations of system~\eqref{EQN::HETERODIMER} of the form
\begin{equation*}
\begin{aligned}
    p(x,y,t) = & \psi_p(x-vt) = \psi_p(\xi),\\
    q(x,y,t) = & \psi_q(x-vt) = \psi_q(\xi),
\end{aligned}
\end{equation*}
where~$v$ is a wave speed depending on the physical parameters and defined by~$v \coloneqq  10\,{\mathrm{d}}$. If~$x\in \R$ or~$t\in\R$, then~$\xi\in\R$.
Substituting~$p$ and~$q$ in the two equations of~\eqref{EQN::HETERODIMER}, we obtain the following equivalent system of ordinary differential equations (ODEs):
\begin{equation}
\label{eq:wave_eq}
    \begin{dcases}
        {\mathrm{d}} \psi_p''(\xi) + v\psi_p'(\xi) - \lambda_p \psi_p(\xi) - \mu_{pq}\psi_p(\xi)\psi_q(\xi) + \kappa_p = 0, & \xi\in \R, 
        \\
        {\mathrm{d}} \psi_q''(\xi) + v \psi_q'(\xi) - \lambda_q \psi_q(\xi) + \mu_{pq}\psi_p(\xi)\psi_q(\xi) = 0, & 
        \xi\in \R.
        \\
    \end{dcases}
\end{equation}
Under the additional assumptions~$\lambda_q=\lambda_p$ and ${\mathrm{d}} = {(\kappa_p\,\mu_{pq} - \lambda_p^2)}/{(24 \lambda_p)}$, it can be verified that system~\eqref{eq:wave_eq} admits the following solution: 
\begin{equation*}
\begin{aligned}
    \psi_p(\xi) = & \dfrac{\lambda_p^2 + 3 \kappa_p \mu_{pq} + (\lambda_p^2 - \kappa_p \mu_{pq})(\tanh(\xi)^2 - 2 \tanh(\xi))}{4 \lambda_p \mu_{pq}}, \\
    \psi_q(\xi) = & -\dfrac{(\lambda_p^2 - \kappa_p \mu_{pq})}{4 \lambda_p \mu_{pq}}  \left(1 - \tanh(\xi)\right)^2.
\end{aligned}
\end{equation*}
The functions~$\psi_p$ and~$\psi_q$ are positive for all~$\xi\in\mathbb{R}$, under the assumption~$\Upsilon_{pq}>0$ already discussed in Section~\ref{SEC::MODEL}.
They satisfy a homogeneous Neumann boundary condition at the limits~$\xi\rightarrow\pm\infty$, which is equivalent to~$x\rightarrow\pm\infty$ for each fixed value of~$t\in(0,T)$. 
The homogeneous Neumann boundary condition is also satisfied in the $y$-direction, as %
$p$ and~$q$ are both independent of~$y$.
\par
This solution respects the equilibria of system~\eqref{EQN::HETERODIMER}.\footnote{\label{footnote:equilibria}
For~$\Upsilon_{pq}>0$, $(\equilp,0)$ and~$(\lambda_q/\mu_{pq},\equilq)$ are admissible equilibria, the former unstable and the latter stable; see~\cite[\S2.1]{Mattia}.} Indeed, considering  the limit $t\rightarrow -\infty$ (or, equivalently, $\xi\rightarrow +\infty$) we find the \emph{unstable} equilibrium solutions:
\begin{equation*}
    \lim_{\xi\rightarrow+\infty} \psi_p(\xi) = \dfrac{\kappa_p}{\lambda_p} = \equilp \qquad \mathrm{and} \qquad \lim_{\xi\rightarrow+\infty} \psi_q(\xi) = 0.
\end{equation*}
Moreover, taking the limit $t\rightarrow+\infty$ (or, equivalently, $\xi\rightarrow-\infty$), we recover the \emph{stable} equilibrium solutions: 
\begin{equation*}
    \lim_{\xi\rightarrow-\infty} \psi_p(\xi) = \dfrac{\lambda_p}{\mu_{pq}} = \dfrac{\lambda_q}{\mu_{pq}} \qquad \mathrm{and} \qquad \lim_{\xi\rightarrow-\infty} \psi_q(\xi) =\dfrac{\kappa_p \mu_{pq}-\lambda_p^2}{\lambda_p \mu_{pq}} = \dfrac{\kappa_p \mu_{pq} - \lambda_p\lambda_q}{ \lambda_q \mu_{pq}} = \equilq.
\end{equation*}
\par
For this test case, we consider a rectangular space domain~$\Omega = (-10,10)\times(0,5)$, and the final time~$T=1$. We impose homogeneous Neumann boundary conditions not only at~$y=0$ and~$y=5$, but also at~$x=-10$ and~$x=10$. Concerning the problem coefficients, we fix $\kappa_p=0.1$, $\lambda_p=\lambda_q=0.1$, and $\mu_{pq}=1$, with corresponding diffusion coefficient ${\mathrm{d}} = 3.75 \times 10^{-2}$, velocity $v = 3.75 \times 10^{-1}$, and constant~$C_f \approx 5.15$ for the bound in Proposition~\ref{PROP::HETER_FBOUND} on the reaction term.
\begin{figure}[t!]
\begin{subfigure}[b]{\textwidth}
	\centering
	{\includegraphics[width=\textwidth]{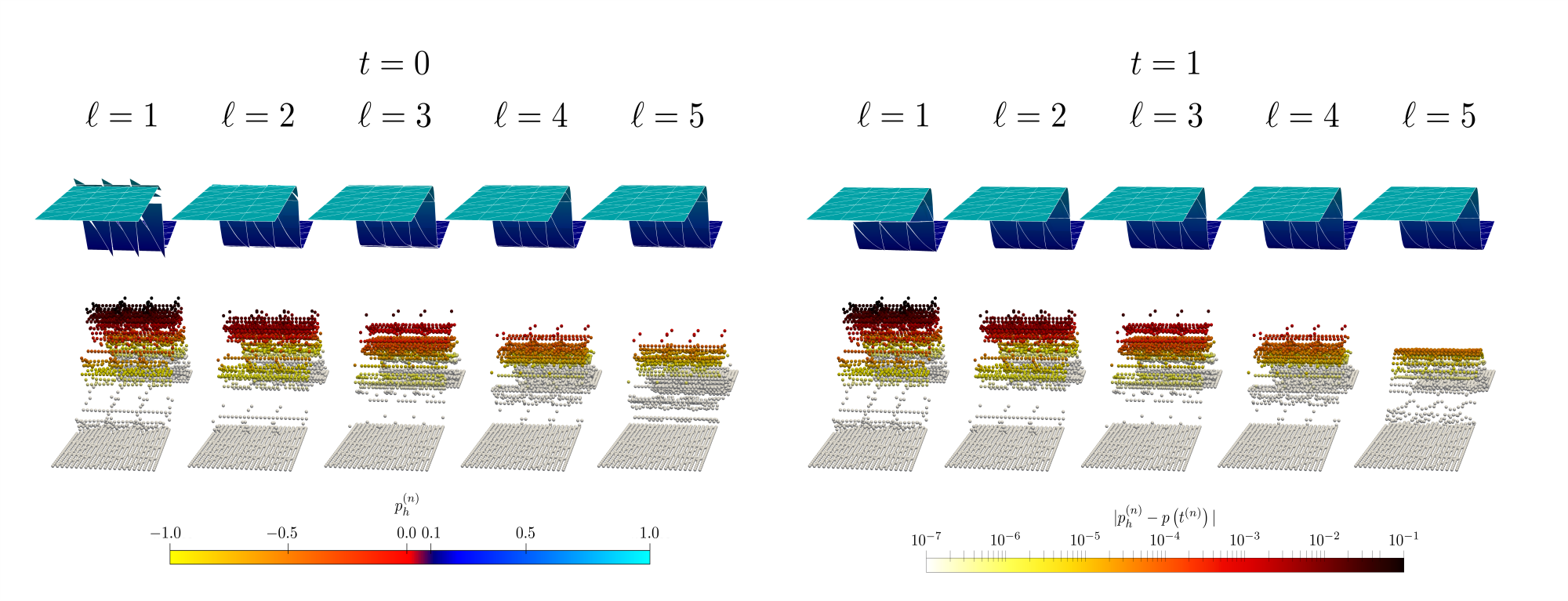}}
    \caption{Solutions (first row) with associated approximation errors (second row) for the variable $p$.}
    \label{fig:waves2Dp}
\end{subfigure}
\begin{subfigure}[b]{\textwidth}
	\centering
    {\includegraphics[width=\textwidth]{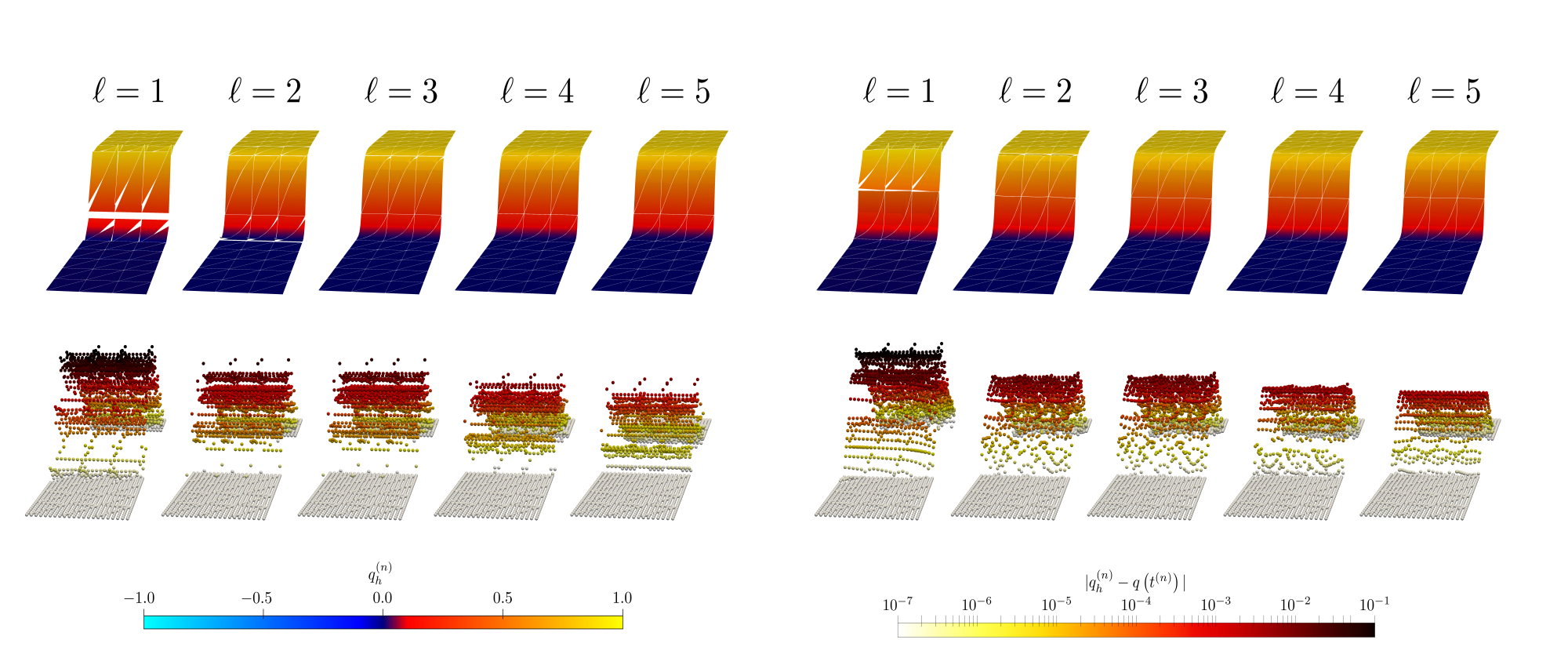}}	\caption{Solutions (first row) with associated approximation errors (second row) for the variable $q$.}
    \label{fig:waves2Dq}
\end{subfigure}
\caption{Test case 2: Initial conditions ($t=0$) and solutions at $t=1$ (first row) for different polynomial degrees $\ell=1,...,5$ with associated approximation errors (second row) for the variables $p$ (a) and $q$ (b).}
\label{fig:waves2D}
\end{figure}

\begin{figure}[t!]
    \begin{subfigure}[b]{0.5\textwidth}
          \resizebox{\textwidth}{!}{\definecolor{mycolor2}{rgb}{0.00000,1.00000,1.00000}%
\begin{tikzpicture}
\begin{axis}[%
width=3.875in,
height=2.50in,
at={(2.6in,1.099in)},
scale only axis,
xmode=log,
xmin=5.8926e-01,
xmax=2.3570e+00,
xminorticks=true,
xlabel = {$h$ [-]},
ylabel = {$\|p(\cdot,T)-u_p(w_{p,h}^{(N)})\|_{L^2(\Omega)}$},
xticklabel={\pgfmathparse{exp(\tick)}\pgfmathprintnumber{\pgfmathresult}},
x tick label style={
/pgf/number format/.cd, fixed, fixed zerofill,
precision=2},
ymode=log,
ymin=1e-6,
ymax=2e-1,
yminorticks=true,
axis background/.style={fill=white},
title style={font=\bfseries},
title={Errors $E_{p}$},
xmajorgrids,
xminorgrids,
ymajorgrids,
yminorgrids,
legend style={at={(0.96,0.40)},legend cell align=left, draw=white!15!black}
]
              
\addplot [color=red, line width=2.0pt, mark=*]
  table[row sep=crcr]{%
2.3570e+00  9.0032e-02\\
1.1785e+00  1.8738e-02\\
7.8568e-01  8.3132e-03\\
5.8926e-01  4.7536e-03\\
};
\addlegendentry{$\ell=1$}

\addplot [color=violet, line width=2.0pt, mark=*]
  table[row sep=crcr]{%
2.3570e+00  1.4859e-02\\
1.1785e+00  1.8099e-03\\
7.8568e-01  5.4908e-04\\
5.8926e-01  2.1704e-04\\
};
\addlegendentry{$\ell=2$}

\addplot [color=blue, line width=2.0pt, mark=*] 
  table[row sep=crcr]{
2.3570e+00  6.7302e-03\\
1.1785e+00  3.9786e-04\\
7.8568e-01  8.2138e-05\\
5.8926e-01  2.5720e-05\\
};
\addlegendentry{$\ell=3$}

\addplot [color=teal, line width=2.0pt, mark=*]
  table[row sep=crcr]{%
2.3570e+00  3.1437e-03\\
1.1785e+00  9.8729e-05\\
7.8568e-01  1.0169e-05\\
5.8926e-01  2.4126e-06\\
};
\addlegendentry{$\ell=4$}

\node[right, align=left, text=black, font=\footnotesize]
at (axis cs:1.005,0.0075) {$2$};

\addplot [color=black, line width=1.5pt]
  table[row sep=crcr]{%
1.00   1e-2\\
0.75   5.625e-3\\
1.00   5.625e-3\\
1.00   1e-2\\
};

\node[right, align=left, text=black, font=\footnotesize]
at (axis cs:1.005,5e-4) {$3$};

\addplot [color=black, line width=1.5pt]
  table[row sep=crcr]{%
1.00   8.0e-04\\
0.75   3.36e-04\\
1.00   3.36e-04\\
1.00   8.0e-04\\
};

\node[right, align=left, text=black, font=\footnotesize]
at (axis cs:1.005,9e-5) {$4$};

\addplot [color=black, line width=1.5pt]
  table[row sep=crcr]{%
1.00   1.50e-4\\
0.75   4.74e-5\\
1.00   4.74e-5\\
1.00   1.50e-4\\
};

\node[right, align=left, text=black, font=\footnotesize]
at (axis cs:1.005,1e-5) {$5$};

\addplot [color=black, line width=1.5pt]
  table[row sep=crcr]{%
1.00   2.00e-5\\
0.75   4.72e-6\\
1.00   4.72e-6\\
1.00   2.00e-5\\
};

\end{axis}
\end{tikzpicture}%
}
          \caption{Computed errors $E_{p}$ w.r.t.~the mesh size~$h$.}
        \label{fig:Tria_waves_errors2D_h_L2_p}
    \end{subfigure}
    \begin{subfigure}[b]{0.5\textwidth}
        \resizebox{\textwidth}{!}{\definecolor{mycolor2}{rgb}{0.00000,1.00000,1.00000}%
\begin{tikzpicture}
\begin{axis}[%
width=3.875in,
height=2.50in,
at={(2.6in,1.099in)},
scale only axis,
xmode=log,
xmin=5.8926e-01,
xmax=2.3570e+00,
xminorticks=true,
xlabel = {$h$ [-]},
ylabel = {$\|q(\cdot,T)-u_q(w_{q,h}^{(N)})\|_{L^2(\Omega)}$},
xticklabel={\pgfmathparse{exp(\tick)}\pgfmathprintnumber{\pgfmathresult}},
x tick label style={
/pgf/number format/.cd, fixed, fixed zerofill,
precision=2},
ymode=log,
ymin=1e-6,
ymax=2e-1,
yminorticks=true,
axis background/.style={fill=white},
title style={font=\bfseries},
title={Errors $E_{q}$},
xmajorgrids,
xminorgrids,
ymajorgrids,
yminorgrids,
legend style={at={(0.96,0.40)},legend cell align=left, draw=white!15!black}
]
              
\addplot [color=red, line width=2.0pt, mark=*]
  table[row sep=crcr]{%
2.3570e+00  1.3264e-01\\
1.1785e+00  4.1911e-02\\
7.8568e-01  1.9033e-02\\
5.8926e-01  1.0806e-02\\
};
\addlegendentry{$\ell=1$}

\addplot [color=violet, line width=2.0pt, mark=*]
  table[row sep=crcr]{%
2.3570e+00  3.5972e-02\\
1.1785e+00  5.4823e-03\\
7.8568e-01  1.6876e-03\\
5.8926e-01  7.1155e-04\\
};
\addlegendentry{$\ell=2$}

\addplot [color=blue, line width=2.0pt, mark=*] 
  table[row sep=crcr]{
2.3570e+00  8.2346e-03\\
1.1785e+00  5.0127e-04\\
7.8568e-01  1.0858e-04\\
5.8926e-01  3.4443e-05\\
};
\addlegendentry{$\ell=3$}

\addplot [color=teal, line width=2.0pt, mark=*]
  table[row sep=crcr]{%
2.3570e+00  3.1918e-03\\
1.1785e+00  1.0360e-04\\
7.8568e-01  1.1021e-05\\
5.8926e-01  2.6142e-06\\
};
\addlegendentry{$\ell=4$}

\node[right, align=left, text=black, font=\footnotesize]
at (axis cs:1.005,0.015) {$2$};

\addplot [color=black, line width=1.5pt]
  table[row sep=crcr]{%
1.00   2.0e-2\\
0.75   1.125e-2\\
1.00   1.125e-2\\
1.00   2.0e-2\\
};

\node[right, align=left, text=black, font=\footnotesize]
at (axis cs:1.005,1.25e-3) {$3$};

\addplot [color=black, line width=1.5pt]
  table[row sep=crcr]{%
1.00   2.0e-03\\
0.75   8.4e-04\\
1.00   8.4e-04\\
1.00   2.0e-03\\
};

\node[right, align=left, text=black, font=\footnotesize]
at (axis cs:1.005,9e-5) {$4$};

\addplot [color=black, line width=1.5pt]
  table[row sep=crcr]{%
1.00   1.80e-4\\
0.75   5.40e-5\\
1.00   5.40e-5\\
1.00   1.80e-4\\
};

\node[right, align=left, text=black, font=\footnotesize]
at (axis cs:1.005,1e-5) {$5$};

\addplot [color=black, line width=1.5pt]
  table[row sep=crcr]{%
1.00   2.00e-5\\
0.75   4.72e-6\\
1.00   4.72e-6\\
1.00   2.00e-5\\
};

\end{axis}
\end{tikzpicture}%
}
          \caption{Computed errors $E_{q}$
          w.r.t.~the mesh size $h$.}
        \label{fig:Tria_waves_errors2D_h_L2_q}
    \end{subfigure}
    \\[0.5mm]
    \begin{subfigure}[b]{0.5\textwidth}
          \resizebox{\textwidth}{!}{\definecolor{mycolor2}{rgb}{0.00000,1.00000,1.00000}%
\begin{tikzpicture}
\begin{axis}[%
width=3.875in,
height=2.50in,
at={(2.6in,1.099in)},
scale only axis,
xmode=log,
xmin=5.8926e-01,
xmax=2.3570e+00,
xminorticks=true,
xlabel = {$h$ [-]},
ylabel = {$\|\nabla p(\cdot,T) + \sigma_{p,h}^{(N)}\|_{L^2(\Omega)^d}$},
xticklabel={\pgfmathparse{exp(\tick)}\pgfmathprintnumber{\pgfmathresult}},
x tick label style={
/pgf/number format/.cd, fixed, fixed zerofill,
precision=2},
ymode=log,
ymin=5e-6,
ymax=1e+0,
yminorticks=true,
axis background/.style={fill=white},
title style={font=\bfseries},
title={Errors $E_{\boldsymbol{\sigma}_p}$},
xmajorgrids,
xminorgrids,
ymajorgrids,
yminorgrids,
legend style={at={(0.96,0.40)},legend cell align=left, draw=white!15!black}
]
              
\addplot [color=red, line width=2.0pt, mark=square*, dotted, mark options=solid]
  table[row sep=crcr]{%
2.3570e+00  2.8875e-01\\
1.1785e+00  1.5408e-01\\
7.8568e-01  1.0439e-01\\
5.8926e-01  7.8800e-02\\
};
\addlegendentry{$\ell=1$}

\addplot [color=violet, line width=2.0pt, mark=square*, dotted, mark options=solid]
  table[row sep=crcr]{%
2.3570e+00  2.8890e-02\\
1.1785e+00  7.1012e-03\\
7.8568e-01  2.8343e-03\\
5.8926e-01  1.5000e-03\\
};
\addlegendentry{$\ell=2$}

\addplot [color=blue, line width=2.0pt, mark=square*, dotted, mark options=solid]
  table[row sep=crcr]{%
2.3570e+00  1.8913e-02\\
1.1785e+00  1.2891e-03\\
7.8568e-01  3.4854e-04\\
5.8926e-01  1.4196e-04\\
};
\addlegendentry{$\ell=3$}

\addplot [color=teal, line width=2.0pt, mark=square*, dotted, mark options=solid]
  table[row sep=crcr]{
2.3570e+00  6.5771e-03\\
1.1785e+00  2.8729e-04\\
7.8568e-01  4.0557e-05\\
5.8926e-01  9.6412e-06\\
};
\addlegendentry{$\ell=4$}

\node[right, align=left, text=black, font=\footnotesize]
at (axis cs:1.005,9e-2) {$1$};

\addplot [color=black, line width=1.5pt]
  table[row sep=crcr]{%
1.00   1.0e-1\\
0.75   7.5e-2\\
1.00   7.5e-2\\
1.00   1.0e-1\\
};

\node[right, align=left, text=black, font=\footnotesize]
at (axis cs:1.005,2.25e-3) {$2$};

\addplot [color=black, line width=1.5pt]
  table[row sep=crcr]{%
1.00   3.0000e-3\\
0.75   1.6875e-3\\
1.00   1.6875e-3\\
1.00   3.0000e-3\\
};

\node[right, align=left, text=black, font=\footnotesize]
at (axis cs:1.005,3.25e-4) {$3$};

\addplot [color=black, line width=1.5pt]
  table[row sep=crcr]{%
1.00   5.00e-04\\
0.75   2.10e-04\\
1.00   2.10e-04\\
1.00   5.00e-04\\
};

\node[right, align=left, text=black, font=\footnotesize]
at (axis cs:1.005,3.5e-5) {$4$};

\addplot [color=black, line width=1.5pt]
  table[row sep=crcr]{%
1.00   6.00e-5\\
0.75   1.88e-5\\
1.00   1.88e-5\\
1.00   6.00e-5\\
};

\end{axis}
\end{tikzpicture}%
}
          \caption{Computed errors $E_{\boldsymbol{\sigma}_p}$ w.r.t.~the mesh size~$h$.}
        \label{fig:Tria_waves_errors2D_h_DG_p}
    \end{subfigure}
    \begin{subfigure}[b]{0.5\textwidth}
        \resizebox{\textwidth}{!}{\definecolor{mycolor2}{rgb}{0.00000,1.00000,1.00000}%
\begin{tikzpicture}
\begin{axis}[%
width=3.875in,
height=2.50in,
at={(2.6in,1.099in)},
scale only axis,
xmode=log,
xmin=5.8926e-01,
xmax=2.3570e+00,
xminorticks=true,
xlabel = {$h$ [-]},
ylabel = {$\|\nabla q(\cdot,T) + \sigma_{q,h}^{(N)}\|_{L^2(\Omega)^d}$},
xticklabel={\pgfmathparse{exp(\tick)}\pgfmathprintnumber{\pgfmathresult}},
x tick label style={
/pgf/number format/.cd, fixed, fixed zerofill,
precision=2},
ymode=log,
ymin=5e-6,
ymax=1e+0,
yminorticks=true,
axis background/.style={fill=white},
title style={font=\bfseries},
title={Errors $E_{\boldsymbol{\sigma}_q}$},
xmajorgrids,
xminorgrids,
ymajorgrids,
yminorgrids,
legend style={at={(0.96,0.40)},legend cell align=left, draw=white!15!black}
]
              
\addplot [color=red, line width=2.0pt, mark=square*, dotted, mark options=solid]
  table[row sep=crcr]{%
2.3570e+00  6.1449e-01\\
1.1785e+00  3.9395e-01\\
7.8568e-01  2.6584e-01\\
5.8926e-01  2.1486e-01\\
};
\addlegendentry{$\ell=1$}

\addplot [color=violet, line width=2.0pt, mark=square*, dotted, mark options=solid]
  table[row sep=crcr]{%
2.3570e+00  1.8103e-01\\
1.1785e+00  4.8300e-02\\
7.8568e-01  2.1647e-02\\
5.8926e-01  1.2200e-02\\
};
\addlegendentry{$\ell=2$}

\addplot [color=blue, line width=2.0pt, mark=square*, dotted, mark options=solid]
  table[row sep=crcr]{%
2.3570e+00  3.9101e-02\\
1.1785e+00  4.0979e-03\\
7.8568e-01  1.3000e-03\\
5.8926e-01  5.5590e-04\\
};
\addlegendentry{$\ell=3$}

\addplot [color=teal, line width=2.0pt, mark=square*, dotted, mark options=solid]
  table[row sep=crcr]{
2.3570e+00  7.6115e-03\\
1.1785e+00  5.0603e-04\\
7.8568e-01  1.1021e-04\\
5.8926e-01  3.4871e-05\\
};
\addlegendentry{$\ell=4$}

\node[right, align=left, text=black, font=\footnotesize]
at (axis cs:1.005,1.75e-1) {$1$};

\addplot [color=black, line width=1.5pt]
  table[row sep=crcr]{%
1.00   2.0e-1\\
0.75   1.5e-1\\
1.00   1.5e-1\\
1.00   2.0e-1\\
};

\node[right, align=left, text=black, font=\footnotesize]
at (axis cs:1.005,1.45e-2) {$2$};

\addplot [color=black, line width=1.5pt]
  table[row sep=crcr]{%
1.00   2.000e-2\\
0.75   1.125e-2\\
1.00   1.125e-2\\
1.00   2.000e-2\\
};

\node[right, align=left, text=black, font=\footnotesize]
at (axis cs:1.005,9.5e-4) {$3$};

\addplot [color=black, line width=1.5pt]
  table[row sep=crcr]{%
1.00   1.50e-3\\
0.75   6.30e-4\\
1.00   6.30e-4\\
1.00   1.50e-3\\
};

\node[right, align=left, text=black, font=\footnotesize]
at (axis cs:1.005,8e-5) {$4$};

\addplot [color=black, line width=1.5pt]
  table[row sep=crcr]{%
1.00   1.50e-4\\
0.75   4.74e-5\\
1.00   4.74e-5\\
1.00   1.50e-4\\
};

\end{axis}
\end{tikzpicture}}
          \caption{Computed errors $E_{\boldsymbol{\sigma}_q}$
          w.r.t.~the mesh size $h$.}
        \label{fig:Tria_waves_errors2D_h_DG_q}
    \end{subfigure}
    \caption{Test case 2: Computed errors and convergence rates w.r.t.~the polynomial degree~$\ell$.}
    \label{fig:Tria_waves_errors2D}
\end{figure}
\subsubsection*{Convergence with respect to the mesh size}
We %
analyze the convergence with respect to the mesh size~$h$. For the space discretization, we adopt structured triangular meshes with~$N_{\mathsf{el}} = 72,\, 288,\, 648,\, 1152$. 
We take as final time~$T = 1$, and the step of the time discretization is set as $\tau\sim\mathcal{O}(h^{\ell+1})$, so as to equilibrate the errors in space and time.
\par
In the first rows of Figures~\ref{fig:waves2Dp} and~\ref{fig:waves2Dq}, we report the numerical %
approximations $p_{h}^{(n)}$ and $q_{h}^{(n)}$, 
respectively, obtained with the mesh of $72$ elements for different %
values of the polynomial degree %
$\ell=1,...,5$. In the left column of both figures, we report the initial conditions $p_{h}^{(0)}$ and $q_{h}^{(0)}$ at time $t=0$. In the second rows of Figures~\ref{fig:waves2Dp} and~\ref{fig:waves2Dq}, we report the associated approximation errors. We can observe that higher polynomial degrees result in a reduction of the projection error of the initial condition and of the %
approximation error at the final time. Moreover, we highlight that most of the error is spatially located near the wavefront, as expected. We recall that, in method~\eqref{EQ::VARIATIONAL}, the initial conditions are imposed weakly, which slightly differs from the standard backward Euler time-stepping scheme.
\par
To quantify the convergence properties of the discretization scheme, we perform an $h$-convergence analysis for $\ell=1,...,4$. In Figures  \ref{fig:Tria_waves_errors2D_h_L2_p} and \ref{fig:Tria_waves_errors2D_h_L2_q}, we report the errors~$E_p$ and $E_q$, respectively. Additionally, in Figures  \ref{fig:Tria_waves_errors2D_h_L2_p} and \ref{fig:Tria_waves_errors2D_h_L2_q}, we report the errors $E_{\boldsymbol{\sigma}_p}$ and $E_{\boldsymbol{\sigma}_q}$. We can observe that the 
the errors~$E_\star$ and~$E_{\boldsymbol{\sigma}_\star}$ %
decrease with optimal %
convergence rates~$\mathcal{O}(h^{\ell+1})$ and~$\mathcal{O}(h^{\ell})$, respectively. 
\subsubsection*{Comparison with an interior penalty DG method}
\begin{table}[t]
    \centering
    \setlength{\extrarowheight}{1.5pt}
    \begin{tabular}{|c|c c c c c|}
    \multicolumn{6}{c}{$\boldsymbol{h \approx 2.3570}$ \textbf{and} $\boldsymbol{\tau = 2.5\times10^{-1}}$} \\[1pt]
    \hline
    \textbf{Method} & $\ell=1$ & $\ell=2$ & $\ell=3$ & $\ell=4$ & $\ell=5$
    \\ \hline 
    SP-LDG 
    & $2.66\times10^{-1}$ 
    & $2.28\times10^{-1}$  
    & $2.35\times10^{-1}$
    & $2.35\times10^{-1}$
    & $2.35\times10^{-1}$
    \\ 
    IPDG 
    & $9.71\times10^{-1}$ 
    & $1.81\times10^{-1}$  
    & $1.80\times10^{-1}$ 
    & $1.75\times10^{-1}$
    & $1.74\times10^{-1}$
    \\ \hline  
    IPDG overshoot 
    & $+1.34\times10^{-1}$ 
    & $+1.10\times10^{-3}$  
    & $+5.01\times10^{-4}$ 
    & $+1.05\times10^{-4}$
    & $+2.03\times10^{-7}$
    \\ \hline  
    \end{tabular}
    \begin{tabular}{|c|c c c c c|}
    \multicolumn{6}{c}{$\boldsymbol{h \approx 2.3570}$ \textbf{and} $\boldsymbol{\tau = 1.0\times10^{-1}}$} \\[1pt]
    \hline
    \textbf{Method} & $\ell=1$ & $\ell=2$ & $\ell=3$ & $\ell=4$ & $\ell=5$
    \\ \hline 
    SP-LDG 
    & $2.23\times10^{-1}$ 
    & $8.01\times10^{-2}$  
    & $8.61\times10^{-2}$ 
    & $8.66\times10^{-2}$ 
    & $8.65\times10^{-2}$ 
    \\ 
    IPDG 
    & $1.13\times10^{+0}$ 
    & $8.99\times10^{-2}$  
    & $8.30\times10^{-2}$ 
    & $7.73\times10^{-2}$
    & $7.70\times10^{-2}$
    \\ \hline 
    IPDG overshoot 
    & $+6.89\times10^{-1}$ 
    & $+7.90\times10^{-3}$  
    & $+6.01\times10^{-4}$ 
    & $+1.83\times10^{-4}$
    & $+2.20\times10^{-6}$
    \\ \hline   
    \end{tabular}
    \caption{Test case 2: Computed errors in the $L^2(\Omega)$ norm of solution $p$ at final time $T=5$ with %
    the structure-preserving LDG (SP-LDG) and IPDG~\cite{Antonietti_Bonizzoni_Corti_Dallolio:2024} methods.
    }    
    \label{tab:errors_p_waves}
\end{table}
\begin{table}[t]
    \centering
    \setlength{\extrarowheight}{1.5pt}
    \begin{tabular}{|c|c c c c c|}
    \multicolumn{6}{c}{$\boldsymbol{h \approx 2.3570}$ \textbf{and} $\boldsymbol{\tau = 2.5\times10^{-1}}$} \\[1pt]
    \hline
    \textbf{Method} & $\ell=1$ & $\ell=2$ & $\ell=3$ & $\ell=4$ & $\ell=5$
    \\ \hline 
    SP-LDG 
    & $2.98\times10^{-1}$ 
    & $2.28\times10^{-1}$  
    & $2.35\times10^{-1}$ 
    & $2.35\times10^{-1}$
    & $2.35\times10^{-1}$
    \\ 
    IPDG 
    & $1.02\times10^{+0}$ 
    & $4.64\times10^{-1}$  
    & $4.64\times10^{-1}$ 
    & $4.61\times10^{-1}$
    & $4.61\times10^{-1}$
    \\ \hline 
    IPDG undershoot 
    & $-1.23\times10^{-1}$ 
    & $-1.44\times10^{-3}$  
    & $-5.01\times10^{-4}$ 
    & $-1.15\times10^{-4}$
    & $-2.97\times10^{-7}$
    \\ \hline   
    \end{tabular}
        \begin{tabular}{|c|c c c c c|}
    \multicolumn{6}{c}{$\boldsymbol{h \approx 2.3570}$ \textbf{and} $\boldsymbol{\tau = 1.0\times10^{-1}}$} \\[1pt]
    \hline
    \textbf{Method} & $\ell=1$ & $\ell=2$ & $\ell=3$ & $\ell=4$ & $\ell=5$
    \\ \hline 
    SP-LDG 
    & $2.67\times10^{-1}$ 
    & $8.45\times10^{-2}$  
    & $8.62\times10^{-2}$ 
    & $8.66\times10^{-2}$ 
    & $8.65\times10^{-2}$ 
    \\ 
    IPDG 
    & $1.11\times10^{+0}$ 
    & $2.17\times10^{-1}$  
    & $2.14\times10^{-1}$ 
    & $2.10\times10^{-1}$
    & $2.10\times10^{-1}$
    \\ \hline 
    IPDG undershoot 
    & $-6.32\times10^{-1}$ 
    & $-7.70\times10^{-3}$  
    & $-5.18\times10^{-4}$ 
    & $-5.48\times10^{-5}$
    & $-1.82\times10^{-6}$
    \\ \hline   
    \end{tabular}
    \caption{Test case 2: Computed errors in the $L^2(\Omega)$ norm of solution $q$ at final time $T=5$ with the structure-preserving LDG (SP-LDG) and IPDG~\cite{Antonietti_Bonizzoni_Corti_Dallolio:2024} methods.}    
    \label{tab:errors_q_waves}
\end{table}
As a final test, we compare the results obtained with our method against those of a non-structure-preserving method in the literature. In particular, we focus on the interior penalty DG (IPDG) method proposed in~\cite{Antonietti_Bonizzoni_Corti_Dallolio:2024}, which is able to approximate the analytical solution correctly only for a sufficiently refined space mesh or a high polynomial degree. To guarantee a fair comparison, we adopt an implicit Euler time-stepping scheme for the time discretization. In this simulation, we employ the structured mesh of $72$ triangular elements of the previous test, and two time steps, $\tau=0.10$ and $\tau = 0.25$.
\par
In Tables~\ref{tab:errors_p_waves} and~\ref{tab:errors_q_waves}, we report the errors in the~$L^2(\Omega)$ norm at the final time.
The results obtained show that the method proposed in~\cite{Antonietti_Bonizzoni_Corti_Dallolio:2024} 
approximates the exact solution accurately only for sufficiently high polynomial degrees. The two methods become comparable for higher-order approximations; however, our method can capture the solution more accurately and respect physical bounds, even with low polynomial degrees. Moreover, the IPDG method 
suffers from overshoots in the approximation of~$p$, as well as undershoots in the approximation of~$q$. Specifically, the numerical solution violates the bounds valid at the continuous level (see Equation~\eqref{EQ:HETER_BOUNDSCQ}),
whereas these bounds are enforced strongly in our LDG formulation.  It can be noted that these undershoots can be reduced with a higher space resolution, while they are not monotonically decreasing with the time step~$\tau$, even if the $L^2(\Omega)$ errors are smaller.
\subsection{Test case 3: simulations of different stable equilibrium behaviors}
As discussed in~\cite{Mattia}, for~$\Upsilon_{pq}>0$, system~\eqref{EQN::HETERODIMER} admits a stable equilibrium~$(p_E,q_E)\coloneqq (\lambda_q/\mu_{pq},\equilq)$. Depending on the problem coefficients, the approach of the solution to the equilibrium~$(p_E,q_E)$ could be either monotonic (stable \emph{node}) or oscillatory (stable \emph{focus}). The goal of this section is to show the ability of our method to predict the correct asymptotic behavior of the equilibrium in both cases.
\par
The analytical study of the bifurcation, which separates the two behaviors, is feasible only in the case of a one-dimensional propagating front and isotropic, constant diffusion $\D={\mathrm{d}}\,\mathbb{I}_2$. This analysis for system~\eqref{EQN::HETERODIMER} can be found in \cite{kabir_numerical_2022}. Following the steps used in~\cite[\S6]{kabir_numerical_2022} for two populations with equal and constant isotropic diffusion $\mathrm{d}$ and imposing the absence of imaginary parts of the Fourier modes, we can observe that the equilibrium is a stable node if and only if
\begin{equation}
4\lambda_p\lambda_q^3-4\kappa_p\lambda_q^2\mu_{pq}+\kappa_p^2\mu_{pq}^2 \geq 0,  
\end{equation}
independently of the diffusion coefficients applied.
\par
For both %
simulated tests, we consider a rectangular space domain~$\Omega = (-10,10)\times(0,5)$, and the final time~$T=25$. We impose homogeneous Neumann boundary conditions on~$\Gamma \times (0, T)$. Concerning the discretization, we %
we adopt a structured triangular mesh with~$N_{\mathsf{el}} = 288$ and polynomial degree $\ell=2$. The step of the time discretization is fixed as $\tau = 5\times 10^{-3}$. 
As initial conditions, we consider the following analytical continuous functions:
\begin{equation*}
    p_0(x,y) = \dfrac{39}{40}\equilp, \qquad q_0(x,y)=\dfrac{\equilq}{2}\,e^{-x^2}.
\end{equation*}
\par
To test the two distinct behaviors of the equilibrium, we vary the reaction parameters in the simulations. First, we fix $\kappa_p = 4$, $\lambda_p = 0.2$, $\lambda_q = 4.5$, and $\mu_{pq} = 1$ to obtain a stable focus. These parameters are associated with an unstable equilibrium $(\equilp,0)=(20,0)$ and a stable one $(\lambda_q/\mu_{pq},\equilq)\simeq(4.5,0.689)$. 
Moreover, we consider $\kappa_p=0.1$, $\lambda_p=\lambda_q=0.1$, and $\mu_{pq}=1$ to simulate a stable node. With these parameters, the equilibrium $(\equilp,0)=(1,0)$ is unstable, whereas~$(\lambda_q/\mu_{pq},\equilq)=(0.1,0.9)$ is stable.
\begin{figure}[t!]
    \centering
    {\includegraphics[width=\textwidth]{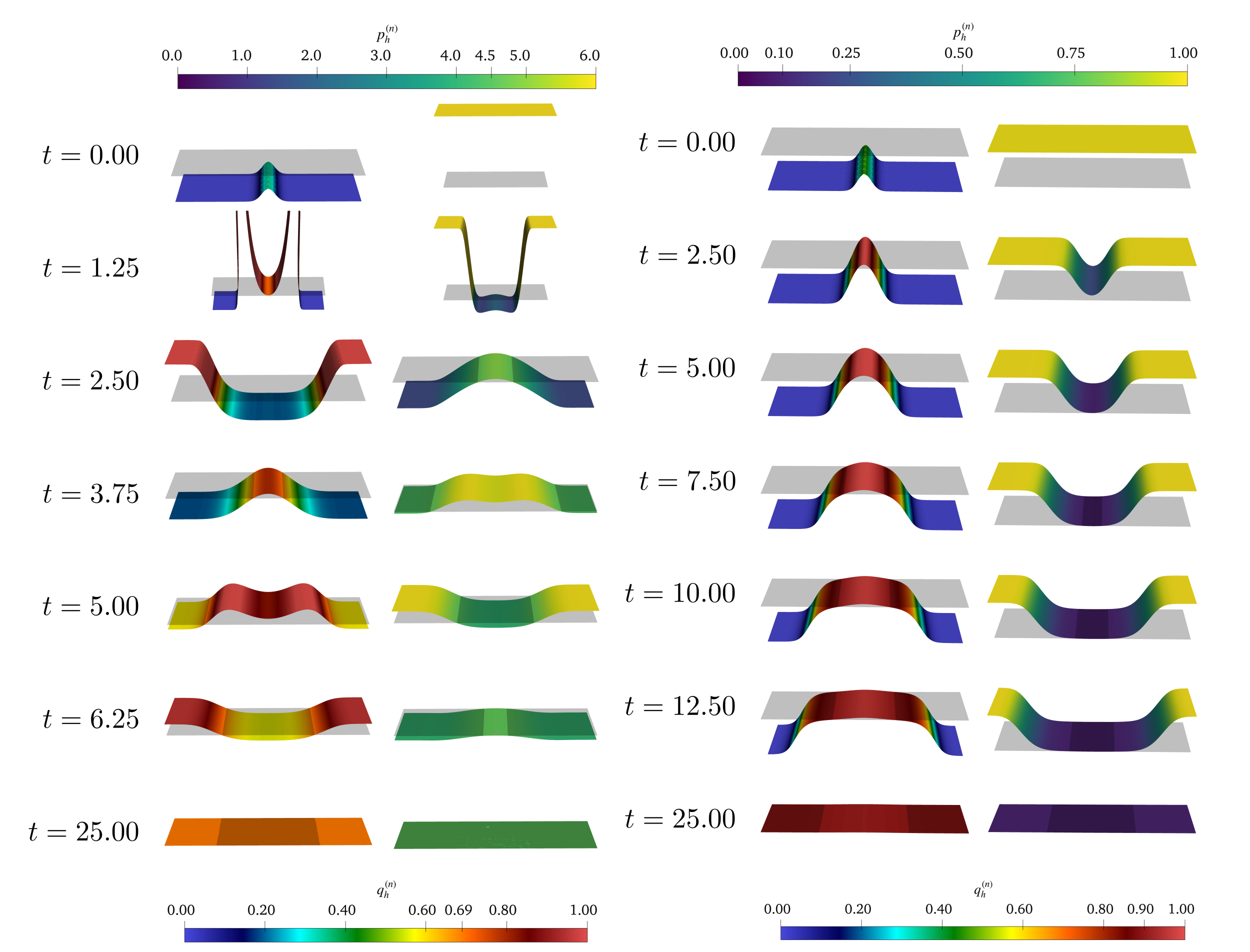}}
\caption{Test case 3: Numerical solutions $q_h^{(n)}$ (first column of each panel) and $p_h^{(n)}$ (second column of each panel) at different times in the case of stable focus (left panel) and stable node equilibrium (right panel).}    
\label{fig:TC3_solution}
\end{figure}
\par
In Figure~\ref{fig:TC3_solution}, we report the results of the numerical simulation, together with the plane associated with the stable equilibrium constants. In the left panel, we report the results for the stable focus behavior, and it can be observed that the solution converges to the equilibrium at long times ($t=25$), after exhibiting some damped oscillations around this value for both concentrations $q$ and $p$. On the other hand, in the results on the right panel for the stable node, we observe that the concentrations reach equilibrium monotonically, and without large peak values in the solutions at intermediate times.

Finally, we tested our method with initial condition given by the stable equilibrium~$(p_E,q_E)$. The solution remains constant and equal to~$(p_E,q_E)$, up to machine precision, during the whole simulation.
\subsection{Test case 4: Impact of diffusion on population spatial distributions}
 In this section, we discuss the impact of the diffusion coefficient~$\D$ on the spatial distribution of the solutions of system~\eqref{EQN::HETERODIMER} and test the capabilities of our method to reproduce anisotropic dynamics of the system. Moreover, we will show that, also in the case of a stable node equilibrium, it is fundamental to apply an unbounded transformation for the variable $q$, because the solution can locally overcome the equilibrium value $\equilq$ and asymptotically approach it from above in a monotonic way. Namely, we cannot provide a bound on the maximum value reached by $q$.
 \par
We consider a square space domain~$\Omega = (-2,2)^2$, and the final time~$T=15$. Moreover, we introduce four subdomains of $\Omega$ (see Figure~\ref{fig:TC4_subdomains}), namely, $\Omega_1 = (-2,0)^2$, $\Omega_2 = (0,2)\times(-2,0)$, $\Omega_3 = (-2,0)\times(0,2)$, and $\Omega_4 = (0,2)^2$. We impose homogeneous Neumann boundary conditions, and we fix the reaction parameters $\kappa_p=0.1$, $\lambda_p=\lambda_q=0.1$, and $\mu_{pq}=1$. These parameters are associated with an unstable equilibrium $(\equilp,0)=(1,0)$ and a stable one $(\lambda_q/\mu_{pq},\equilq)=(0.1,0.9)$. 
Analyzing the associated ODE with these parameter values, one can conclude that the stable equilibrium is a node, so the solution should monotonically approach the steady state. Concerning the diffusion coefficient, we consider four different cases: 
\begin{enumerate}[left = 0.2in, label = \textbf{(TC 4.\arabic*)}, ref = (TC 4.\arabic*)]
    \item \label{TC1} constant isotropic diffusion tensor with high diffusion $\boldsymbol{D}=5\times10^{-2} \mathbb{I}$;
    \item \label{TC2} constant isotropic diffusion tensor with low diffusion $\boldsymbol{D}=10^{-2} \mathbb{I}$;
    \item \label{TC3} discontinuous isotropic diffusion tensor: 
    \begin{equation*}
        \boldsymbol{D}=
        \begin{cases}
            10^{-3} \mathbb{I}, & \mathrm{in}\,\Omega_1, \\
            10^{-2} \mathbb{I}, & \mathrm{in}\,\Omega_2, \\
            5\times10^{-3} \mathbb{I}, & \mathrm{in}\,\Omega_3, \\
            5\times10^{-2} \mathbb{I}, & \mathrm{in}\,\Omega_4; \\
        \end{cases}
    \end{equation*}
    \item \label{TC4} discontinuous anisotropic diffusion tensor (see Figure~\ref{fig:TC4_directions}):
    \begin{equation*}
        \boldsymbol{D}=
        \begin{cases}
            10^{-3} \mathbb{I}, & \mathrm{in}\,\Omega_1, \\
            10^{-2} \mathbb{I}, & \mathrm{in}\,\Omega_2, \\
            10^{-3} \mathbb{I} + 5\times10^{-3} \boldsymbol{a}(x,y)\otimes\boldsymbol{a}(x,y), & \mathrm{in}\,\Omega_3, \\
            10^{-2} \mathbb{I} + 5\times10^{-3} \boldsymbol{a}(x,y)\otimes\boldsymbol{a}(x,y), & \mathrm{in}\,\Omega_4, 
        \end{cases}
        \quad \text{ with $\boldsymbol{a}(\boldsymbol{x}) = \left((1-y)^2+x^4\right)^{-1/2} \begin{pmatrix}
        1-y \\
        x^2
        \end{pmatrix}$.}
    \end{equation*}
\end{enumerate}
\begin{figure}[t!]
\centering
\begin{subfigure}[b]{0.3\textwidth}
	\centering
	{\includegraphics[width=\textwidth]{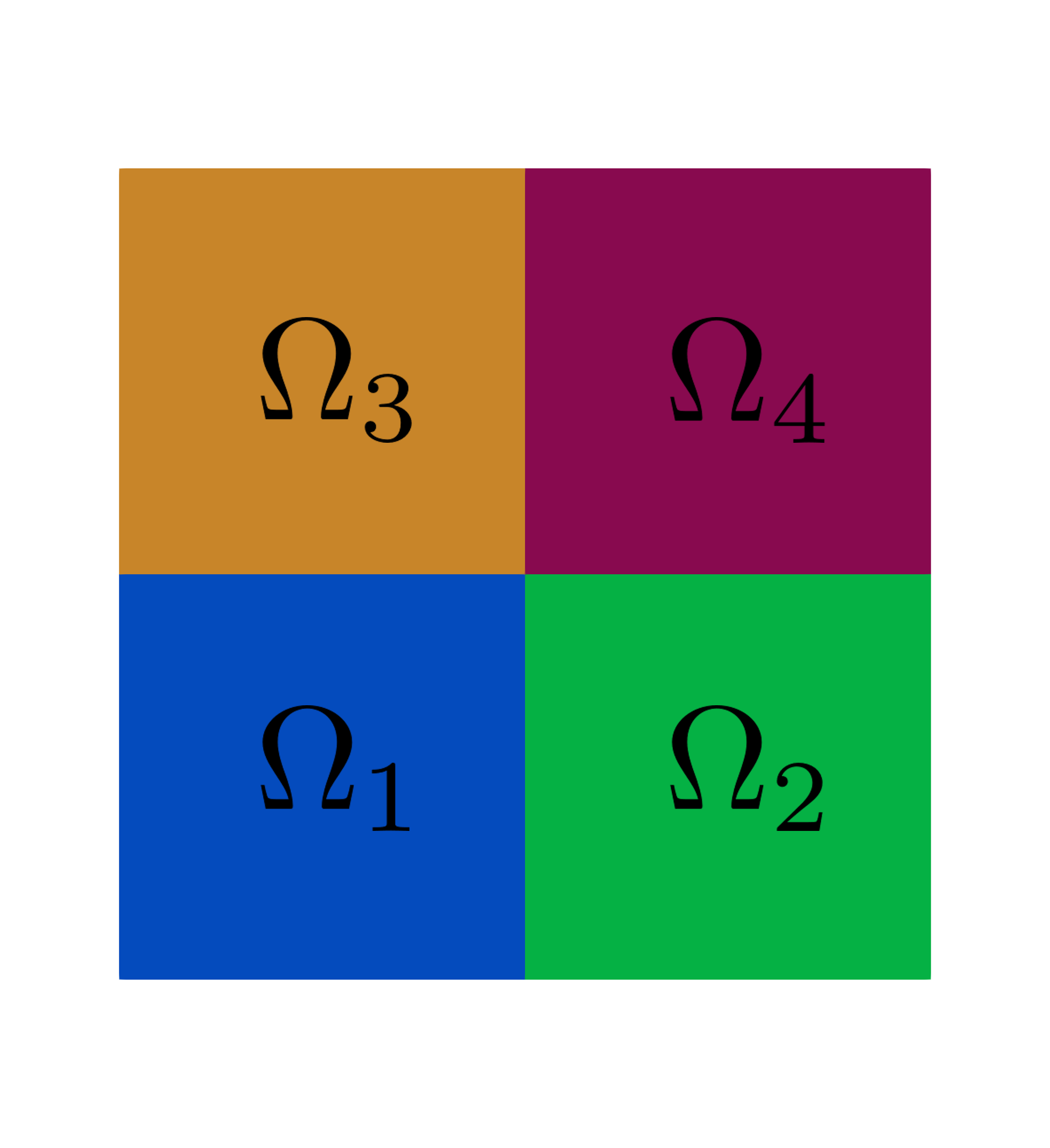}}
    \caption{$\Omega$ subdomains.}
    \label{fig:TC4_subdomains}
\end{subfigure}
\begin{subfigure}[b]{0.3\textwidth}
	\centering
	{\includegraphics[width=\textwidth]{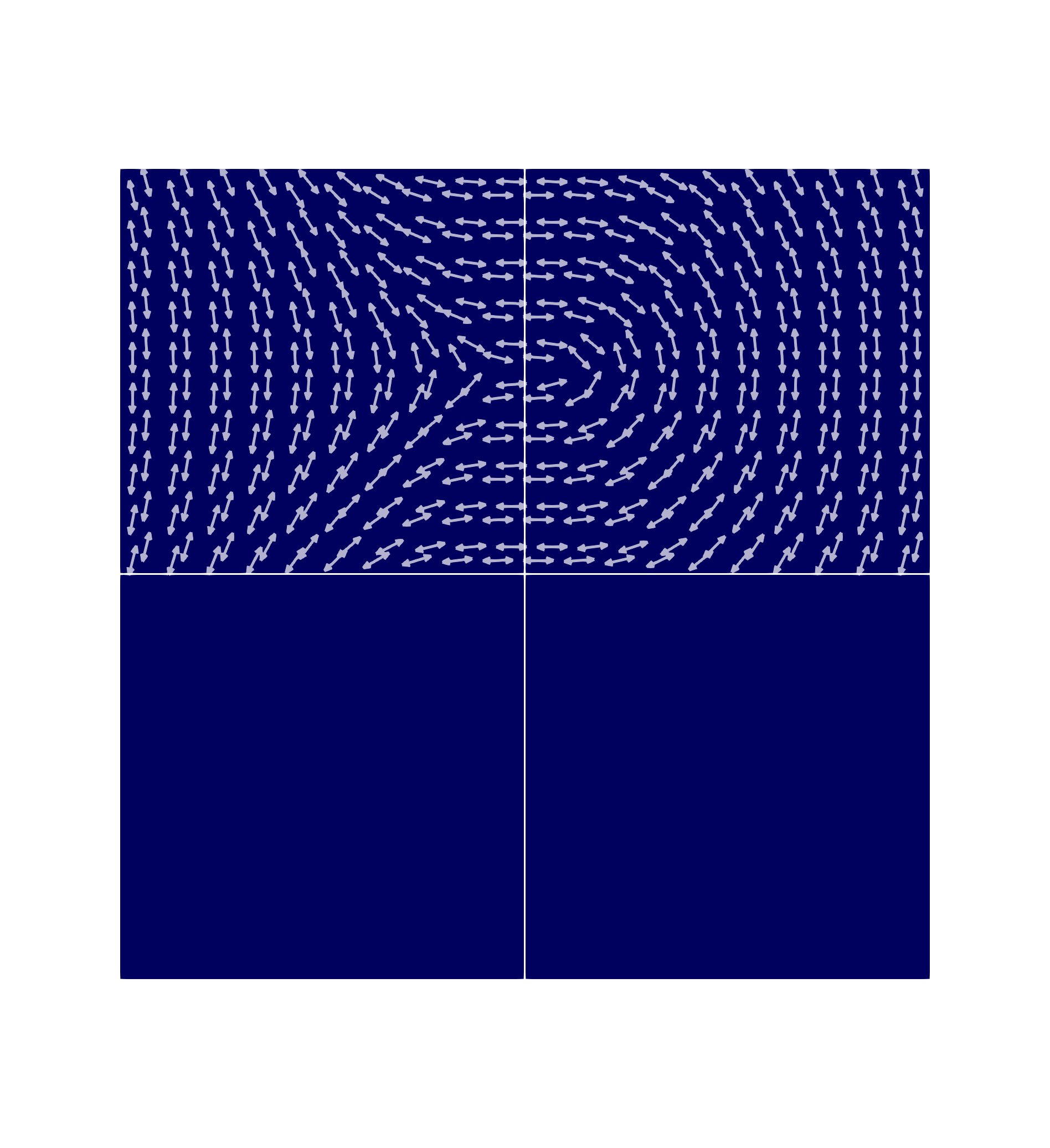}}
    \caption{Anisotropic directions $\boldsymbol{a}$.}
    \label{fig:TC4_directions}
\end{subfigure}
\begin{subfigure}[b]{0.3\textwidth}
	\centering
	{\includegraphics[width=\textwidth]{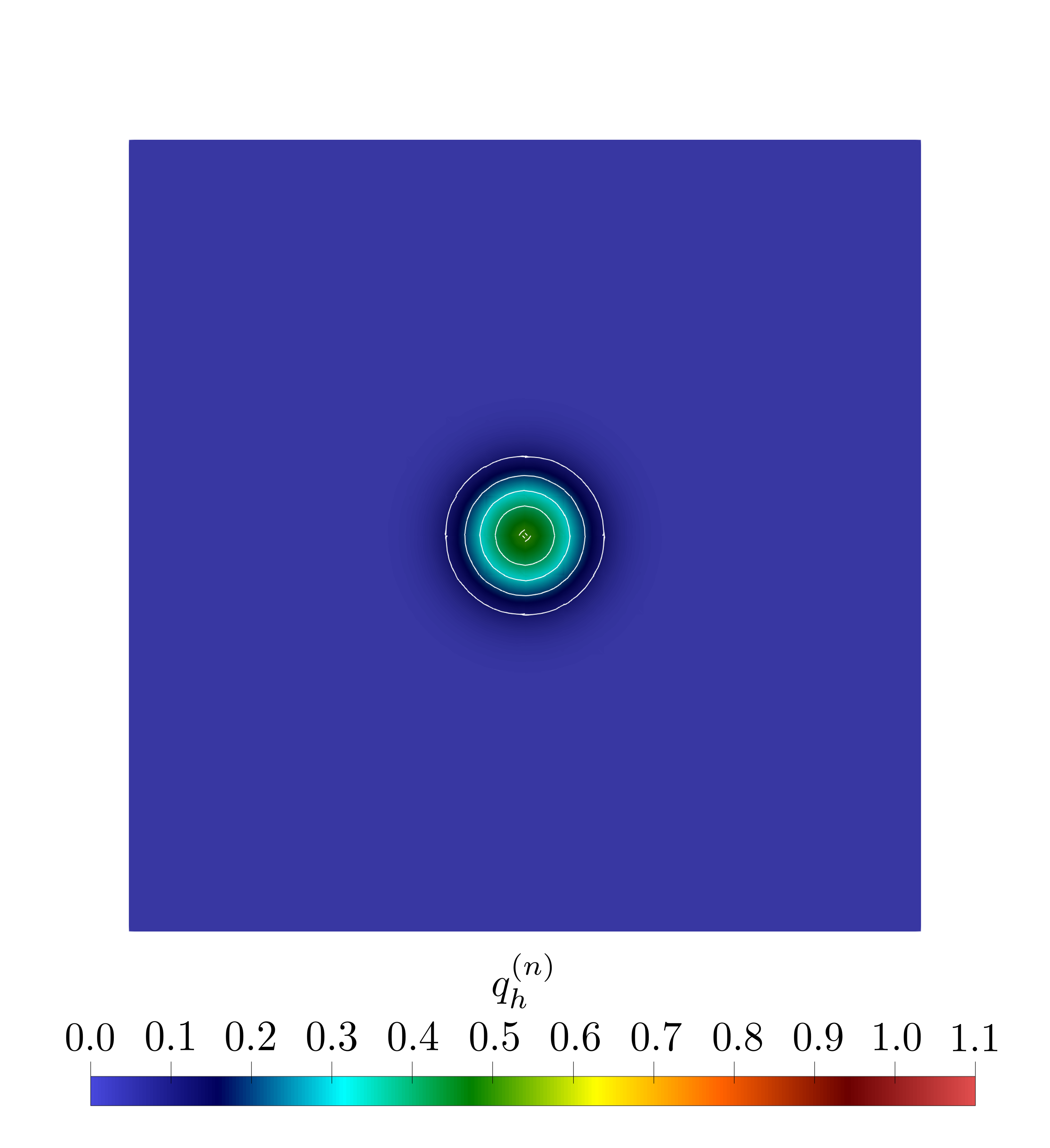}}
    \caption{Initial condition $q^{(0)}_h$.}
    \label{fig:TC4_initial}
\end{subfigure}
\caption{Test case 4: (a) subdomains of $\Omega$, (b) anisotropic directions $\boldsymbol{a}(x,y)$ in subdomains $\Omega_3$ and $\Omega_4$, and (c) discrete initial condition $q^{(0)}_h$.
}
\label{fig:TC4_setup}
\end{figure}
Concerning the discretization, %
we adopt a structured triangular mesh with~$N_{\mathsf{el}} = 800$ ($h \approx 0.2828$) and polynomial degree $\ell=2$. The step of the time discretization is fixed as $\tau = 5\times 10^{-3}$. As initial conditions, we consider 
the following analytical continuous functions:
\begin{equation*}
    p_0(x,y) = 1, \qquad q_0(x,y)=0.5\,e^{-10(x^2+y^2)}.
\end{equation*}
We report the plot of the initial condition for the variable $q$ in Figure~\ref{fig:TC4_initial}.
\begin{figure}[t!]
    \centering
    {\includegraphics[width=\textwidth]{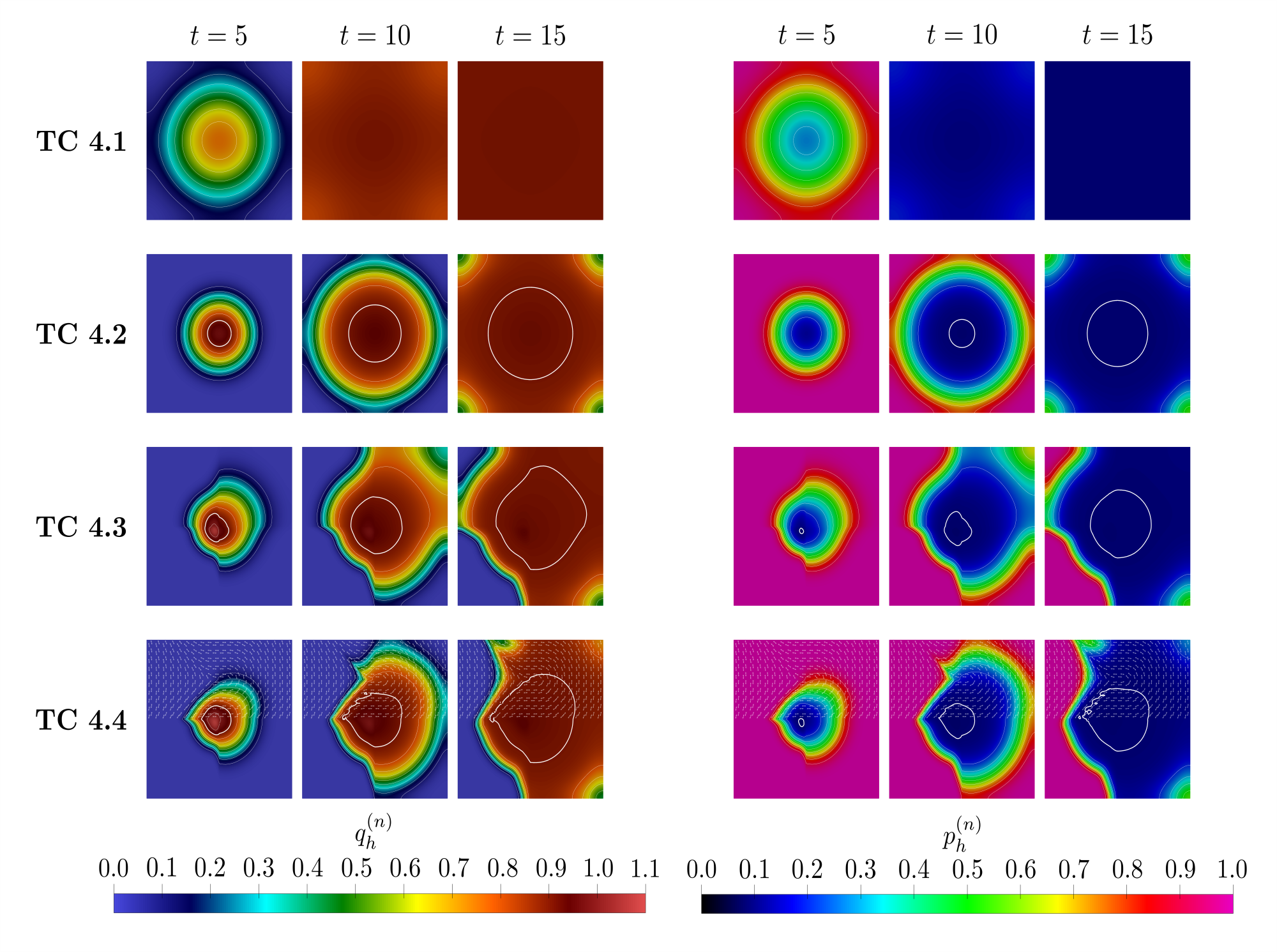}}
\caption{Test case 4: Numerical solutions $p_h^{(n)}$ (first column) and $q_h^{(n)}$ (second column) at different times $t=5,10,15$ considering four the different tested diffusion tensors.}    
\label{fig:TC4_solution}
\end{figure}
In Figure~\ref{fig:TC4_solution}, we report the numerical solution computed for \ref{TC1} to \ref{TC4}. 
The numerical solution is depicted at three different times $t=5,\, 10,\, 15$, and the isolines of the solutions at levels $\{0.1,0.2,...,1.0\}$ are also reported. In particular, the isolines associated with the stable equilibrium $(p,q)=(0.1,0.9)$ are reported as thicker lines in the visualization. As we can observe, in the first case~\ref{TC1}, the solution approaches the equilibrium points monotonically from below and above for~$q$ and~$p$, respectively. In contrast, a reduction of the diffusion value~\ref{TC2} causes a change in the dynamics, and the solution exceeds the equilibrium values and then approaches them monotonically from above for~$q$ and from below for~$p$.
\par
In the latter cases~\ref{TC3} and~\ref{TC4}, we have discontinuous diffusion tensors. These cases clearly demonstrate the impact of different diffusion regions on the solution. In areas with lower diffusion values, the wave fronts are sharper, and the solution spreads more slowly throughout the domain. We also notice that, in proximity to the diffusion tensor discontinuity, $\Omega_1$ presents the highest peak of $q$ (see $t=5,10$ in Figure~\ref{fig:TC4_solution}).
Finally, focusing on~\ref{TC4}, as a result of introducing a preferential direction of diffusion in the subdomain~$\Omega_3\cup\Omega_4$, we can observe that the propagating front moves fast along the direction~$\boldsymbol{a}$. At the same time, it is significantly slower along the orthogonal direction. %
Consequently, the fronts are sharp in the orthogonal direction and soft in the~$\boldsymbol{a}$ direction (see last row of Figure~\ref{fig:TC4_solution}). This expected behavior confirms the method's ability to approximate the solution, even in the presence of anisotropic diffusion tensors.
\section{Conclusions}
\label{sec: Conclusions}
In this work, we have analyzed a two-state conformational conversion system and proposed a novel structure-preserving numerical scheme that combines a local discontinuous Galerkin space discretization with a backward Euler time-integration method. The proposed approach guarantees essential physical and mathematical properties at the discrete level — namely, positivity, boundedness, and a discrete stability bound. We prove the convergence of the numerical solution (up to subsequences) under suitable regularity assumptions. As an additional outcome of the analysis presented, we show the existence of global weak solutions satisfying the problem's physical bounds. 
Numerical experiments validated the theoretical results and
highlighted the practical performance of the proposed schemes.
Possible further developments include the analysis of multi-state conformational systems, the introduction of high-order time integration schemes, and the development of adaptive strategies to enhance computational efficiency, while guaranteeing structural properties. 

\section*{Acknowledgments}
The authors are grateful to the anonymous Referees for their valuable comments and suggestions that significantly contributed to the improvement of the manuscript. 

\section*{Funding}
PFA, MC, and SG were supported by the European Union (ERC Synergy, NEMESIS, project number
101115663). Views and opinions expressed are, however, those of the authors only and do not necessarily
reflect those of the EU or the ERC Executive Agency.
The first three authors are members of the INdAM-GNCS group.
This research was funded in part by the Austrian Science Fund (FWF) project 10.55776/F65. The present research is part of the activities of Dipartimento di Eccellenza 2023-2027 (Dipartimento di Matematica, Politecnico di Milano).
\bibliographystyle{abbrv}
\bibliography{references}

\appendix
\section{Newton's iteration}\label{SEC::NEWTON}

In this section, we derive the linear systems resulting from Newton's iteration, thereby providing an implementable algorithm. The computational bottleneck in each Newton's iteration is the evaluation of the nonlinear terms within the multivariate function and its Jacobian matrix. As the nonlinearities in our approach do not affect the interface integrals, these terms can be computed separately for each mesh element. As a result, the Jacobians are block-diagonal, endowing the method with a naturally parallelizable structure. 

We set, for convenience,~$C\coloneqq M_D M_I^{-1}B$ and, for~$\star=p,q$,
\[
\begin{split}
\GG_1^\star\big(\cvSigmastar^{\npo},\cvWstar^{\npo}\big)&\coloneqq 
\mathcal{N}_{\star}\big(\cvWstar^{\npo}\big) \cvSigmastar^{\npo} +C \cvWstar^{\npo},\\[0.2cm]
\GG_2^\star\big(\cvSigmastar^{\npo},\cvWp^{\npo},\cvWq^{\npo}\big)&\coloneqq 
\varepsilon A_{\LDG} \cvWstar^{\npo} + \frac{1}{\tau_{n + 1}}\mathcal{U}_{\star}(\cvWstar^{\npo})  - C^T \cvSigmastar^{\npo} +  J \cvWstar^{(n + 1)}\\
& \qquad - \mathcal{F}_{\star} \big(\cvWp^{\npo}, \cvWq^{\npo}) -\frac{1}{\tau_{n + 1}}\mathcal{U}_{\star, h}^{(n)}.
\end{split}
\]
Denote by $\diffsigma$ and $\diffw$ the Jacobian matrices with respect to~$\mathbf{\Sigma}_\star$ and~$\mathbf{W}_\star$, respectively. Omitting the temporal index~$n+1$,
the step~$k\to k+1$ of Newton's iteration applied to system~\eqref{EQ::MATRICIAL_COMPACT} reads as follows: for~$\star=p,q$,
\[
\begin{split}
&\diffsigma\left(\GG_1^\star\left(\cvSigmastar^{k},\cvWstar^{k}\right)\right)\big(\cvSigmastar^{k+1}-\cvSigmastar^{k}\big)
+\diffw\left(\GG_1^\star\left(\cvSigmastar^{k},\cvWstar^{k}\right)\right)\big(\cvWstar^{k+1}-\cvWstar^{k}\big)
=-\GG_1^\star\left(\cvSigmastar^{k},\cvWstar^{k}\right),\\[0.2cm]
&\diffsigma\left(\GG_2^\star\left(\cvSigmastar^{k},\cvWp^{k},\cvWq^{k}\right)\right)\big(\cvSigmastar^{k+1}-\cvSigmastar^{k}\big)
+\diffwp\left(\GG_2^\star\left(\cvSigmastar^{k},\cvWp^{k},\cvWq^{k}\right)\right)\big(\cvWp^{k+1}-\cvWp^{k}\big)\\
&\hspace{4cm} +\diffwq\left(\GG_2^\star\left(\cvSigmastar^{k},\cvWp^{k},\cvWq^{k}\right)\right)\big(\cvWq^{k+1}-\cvWq^{k}\big)
=-\GG_2^\star\left(\cvSigmastar^{k},\cvWp^{k},\cvWq^{k}\right).
\end{split}
\]
We compute
\[
\begin{split}
\diffsigma\left(\GG_1^\star\left(\cvSigmastar^{k},\cvWstar^{k}\right)\right)
&=\mathcal{N}_{\star}\big(\cvWstar^{k}\big),\\
\diffw\left(\GG_1^\star\left(\cvSigmastar^{k},\cvWstar^{k}\right)\right)
&=\diffw\left(\mathcal{N}_{\star}\big(\cvWstar^{k}\big)\right)\cvSigmastar^{k}+C,\\
\diffsigma\left(\GG_2^\star\left(\cvSigmastar^{k},\cvWp^{k},\cvWq^{k}\right)\right)
&=-C^T,\\
\diffw\left(\GG_2^\star\left(\cvSigmastar^{k},\cvWp^{k},\cvWq^{k}\right)\right)&=\varepsilon A_{\LDG}+\frac{1}{\tau_{n + 1}}\diffw\left(\mathcal{U}_{\star}(\cvWstar^{k})\right)+J
-\diffw\left(\mathcal{F}_{\star} \big(\cvWp^{k}, \cvWq^{k})\right),\\
\diffwq\left(\GG_2^p\left(\cvSigmap^{k},\cvWp^{k},\cvWq^{k}\right)\right)&=-\diffwq\left(\mathcal{F}_{p} \big(\cvWp^{k}, \cvWq^{k})\right),\\
\diffwp\left(\GG_2^q\left(\cvSigmaq^{k},\cvWp^{k},\cvWq^{k}\right)\right)&=-\diffwp\left(\mathcal{F}_{q} \big(\cvWp^{k}, \cvWq^{k})\right).
\end{split}
\]
Therefore, the Newton iteration applied to system~\eqref{EQ::MATRICIAL_COMPACT} is as follows: for~$\star=p,q$,
\begin{equation}\label{EQ::NEWTON}
\begin{split}
&\mathcal{N}_{\star}\big(\cvWstar^{k}\big) \cvSigmastar^{k+1}=-\Big[\diffw\left(\mathcal{N}_{\star}\big(\cvWstar^{k}\big)\right)\cvSigmastar^{k}+C\Big]\cvWstar^{k+1}+\left(\diffw\left(\mathcal{N}_{\star}\big(\cvWstar^{k}\big)\right)\cvSigmastar^{k}\right)\cvWstar^{k},\\[0.3cm]
&-C^T \cvSigmastar^{k+1}+\Big[\varepsilon A_{\LDG}+\frac{1}{\tau_{n + 1}}\diffw\left(\mathcal{U}_{\star}(\cvWstar^{k})\right)+J\Big]\cvWstar^{k+1}\\
&\qquad\qquad\qquad -\diffwp\left(\mathcal{F}_{\star} \big(\cvWp^{k}, \cvWq^{k})\right)\cvWp^{k+1}-\diffwq\left(\mathcal{F}_{\star} \big(\cvWp^{k}, \cvWq^{k})\right)\cvWq^{k+1}\\
&\qquad
=
\frac{1}{\tau_{n + 1}}\diffw\left(\mathcal{U}_{\star}(\cvWstar^{k})\right)\cvWstar^{k}-\diffwp\left(\mathcal{F}_{\star} \big(\cvWp^{k}, \cvWq^{k})\right)\cvWp^{k}-\diffwq\left(\mathcal{F}_{\star} \big(\cvWp^{k}, \cvWq^{k})\right)\cvWq^{k}\\
&\qquad\qquad\qquad -\frac{1}{\tau_{n + 1}}\Big[\mathcal{U}_{\star}(\cvWstar^{k})-\mathcal{U}_{\star, h}^{(n)}\Big]
+ \mathcal{F}_{\star} \big(\cvWp^{k}, \cvWq^{k}) .
\end{split}
\end{equation}

\begin{remark}[Newton's iteration for~\eqref{EQ::MATRICIAL_W}]
Using~\eqref{EQ::MATRICIAL_COMPACT_SIGMA}, one can eliminate~$\cvSigmastar^{k}$ from~\eqref{EQ::NEWTON} and obtain, for~$\star=p,q$,
\begin{equation}\label{EQ::NEWTON_ONEFIELD}
\begin{split}
&\Big[\varepsilon A_{\LDG}+\frac{1}{\tau_{n + 1}}\diffw\left(\mathcal{U}_{\star}(\cvWstar^{k})\right)
+C^T\left(\mathcal{N}_{\star}\big(\cvWstar^{k}\big)\right)^{-1}C-\MMstar\cvWstar^{k}
+J\Big]\cvWstar^{k+1}\\
&\qquad\qquad\qquad -\diffwp\left(\mathcal{F}_{\star} \big(\cvWp^{k}, \cvWq^{k})\right)\cvWp^{k+1}-\diffwq\left(\mathcal{F}_{\star} \big(\cvWp^{k}, \cvWq^{k})\right)\cvWq^{k+1}\\
&\qquad
=
\frac{1}{\tau_{n + 1}}\diffw\left(\mathcal{U}_{\star}(\cvWstar^{k})\right)\cvWstar^{k}-\diffwp\left(\mathcal{F}_{\star} \big(\cvWp^{k}, \cvWq^{k})\right)\cvWp^{k}-\diffwq\left(\mathcal{F}_{\star} \big(\cvWp^{k}, \cvWq^{k})\right)\cvWq^{k}\\
&\qquad\qquad\qquad -\left(\MMstar\cvWstar^{k}\right)\cvWstar^{k}-\frac{1}{\tau_{n + 1}}\Big[\mathcal{U}_{\star}(\cvWstar^{k})-\mathcal{U}_{\star, h}^{(n)}\Big]
+ \mathcal{F}_{\star} \big(\cvWp^{k}, \cvWq^{k}),
\end{split}
\end{equation}
where~$\MMstar$ is the third-order tensor defined as
\[
\MMstar\coloneqq C^T\left(\mathcal{N}_{\star}\big(\cvWstar^{k}\big)\right)^{-1}\diffw\left(\mathcal{N}_{\star}\big(\cvWstar^{k}\big)\right)\left(\mathcal{N}_{\star}\big(\cvWstar^{k}\big)\right)^{-1}C.
\]
The expression in~\eqref{EQ::NEWTON_ONEFIELD} could also be obtained by applying Newton's iteration directly to the reformulation of system~\eqref{EQ::MATRICIAL_COMPACT} given in~\eqref{EQ::MATRICIAL_W}.
\eremk
\end{remark}

\end{document}